\documentclass[11pt,leqno,oneside,letterpaper]{amsart}
\usepackage[letterpaper,left=1.75truein, top=1truein]{geometry}
\usepackage{amsmath,amsfonts,amssymb,amscd,amsthm,amsbsy,epsf,graphics}

\usepackage{color}
\usepackage[colorlinks,linkcolor=blue,citecolor=blue]{hyperref}
\usepackage{epsfig}
\usepackage{graphicx}
\usepackage{mathtools}
\usepackage[utf8]{inputenc}
\usepackage[shadow,loadshadowlibrary]{todonotes}

\usepackage{pgf,tikz}

\usetikzlibrary{arrows}
\usetikzlibrary{patterns}
\usetikzlibrary{decorations.markings}

\DeclarePairedDelimiter\vbracket{\langle}{\rangle}
\DeclareMathOperator{\sh}{sh}
\DeclareMathOperator{\Homeo}{Homeo}

\DeclarePairedDelimiter{\floor}{\lfloor}{\rfloor}

\usepackage{enumitem}
\newlist{casesp}{enumerate}{3} 
\setlist[casesp]{align=left, 
     listparindent=\parindent, 
     parsep=\parskip, 
     font=\normalfont\bfseries, 
     leftmargin=0pt, 
     labelwidth=0pt, 
     itemindent=.4em,labelsep=.4em, 
     partopsep=0pt, 
     }
\setlist[casesp,1]{label=Case~\arabic*:,ref=\arabic*}
\setlist[casesp,2]{label=Subcase~\thecasespi(\roman*):,ref=\thecasespi.\roman*}
\setlist[casesp,3]{label=Case~\thecasespii.\alph*:,ref=\thecasespii.\alph*}

\newtheorem{theorem}{Theorem}[section]
\newtheorem{lemma}[theorem]{Lemma}
\newtheorem{proposition}[theorem]{Proposition}
\newtheorem{definition}[theorem]{Definition}

\newtheorem{corollary}[theorem]{Corollary}
\theoremstyle{definition} 

\newtheorem{remark}[theorem]{Remark}

\newcommand{\RR}{\mathbb{R}}
\newcommand{\QQ}{\mathbb{Q}}
\newcommand{\ZZ}{\mathbb{Z}}
\newcommand{\LO}{\mathrm{LO}}
\newcommand{\ooo}{\mathfrak{o}}
\newcommand{\Stab}{\mathrm{Stab}}
\newcommand{\PSL}{\mathrm{PSL}}

\newcommand{\LL}{\mathcal{L}}
\newcommand{\hstilde}{\mathrm{H\widetilde{ome}o}_+(S^1)}
\newcommand{\TT}[1]{\mathcal{T}(#1)}
\newcommand{\TTT}[1]{\mathcal{T}_{\text{str}}(#1)}
\newcommand{\cpq}{C_{p,q}}

\title{order-detection, representation-detection, and applications to cable knots}

\begin{document}

\emergencystretch 3em 

\author[A.~Clay]{Adam Clay}

\address{Department of Mathematics, 420 Machray Hall, University of Manitoba, Winnipeg, MB, R3T 2N2}
\email{Adam.Clay@umanitoba.ca}

\author[J.~Lu]{Junyu Lu}

\address{Department of Mathematics, 420 Machray Hall, University of Manitoba, Winnipeg, MB, R3T 2N2}
\email{luj9@myumanitoba.ca}

\begin{abstract}
    Given a $3$-manifold $M$ with multiple incompressible torus boundary components, we develop a general definition of order-detection of tuples of slopes on the boundary components of $M$.	In parallel, we arrive at a general definition of representation-detection of tuples of slopes, and show that these two kinds of slope detection are equivalent---in the sense that a tuple of slopes on the boundary of $M$ is order-detected if and only if it is representation-detected. We use these results, together with new ``relative gluing theorems," to show how the work of Eisenbud-Hirsch-Neumann, Jankins-Neumann and Naimi can be used to determine tuples of representation-detected slopes and, in turn, the behaviour of order-detected slopes on the boundary of a knot manifold with respect to cabling. Our cabling results improve upon work of the first author and Watson, and in particular, this new approach shows how one can use the equivalence between representation-detection and order-detection to derive orderability results that parallel known behaviour of L-spaces with respect to Dehn filling.
\end{abstract}

\maketitle

\tableofcontents

\section{Introduction}

The L-space conjecture asserts that for an irreducible rational homology $3$-sphere $M$, the properties of supporting a co-orientable taut foliation, having a left-orderable fundamental group, and not being a Heegaard-Floer homology L-space are equivalent (\cite[Conjecture 1]{BGW13} \cite[Conjecture 5]{Juh15}).

An essential tool for investigating the L-space conjecture for $3$-manifolds admitting a nontrivial JSJ decomposition is the notion of detected slopes, introduced in \cite{BC17} and developed further in \cite{RR17, HW23, BC23, BGYpreprinta, BGYpreprintb}. Slope detection is a method of recording the behaviour of left-orderings, co-orientable taut foliations, or Heegaard-Floer homology relative to the incompressible torus boundary components of a compact, connected, orientable $3$-manifold, and so slope detection naturally comes in three flavours: order-detection, foliation-detection, and NLS-detection (here, NLS stands for non-L-space). Connecting order-detection and foliation-detection is often done via representations of the fundamental group in $\mathrm{Homeo}_+(\mathbb{R})$, and so representation-detection is a fourth type of detection that, while related to a structure not explicitly mentioned in the L-space conjecture, will likely be key to its eventual resolution in the case of toroidal $3$-manifolds.

Each kind of detection comes in two flavours---regular detection (which we will simply call ``detection'', only including the word ``regular" when needed for emphasis) and strong detection.
Detection encodes the boundary information needed to establish left-orderability of the fundamental group of the manifold $W$ obtained by gluing together two 3-manifolds along incompressible torus boundary components, and whether $W$ admits a co-orientable taut foliation or is not a Heegaard Floer L-space.
Strong detection is a kind of detection adapted to deal with manifolds that arise from Dehn filling a boundary component and to similarly analyse the structures supported by the resulting manifolds.

Given a compact, connected $3$-manifold $M$ whose boundary consists of incompressible torus boundary components $T_1, \ldots, T_n$, let $\mathcal{S}(T_i) = H_1(T_i; \mathbb{R})/\{\pm 1\}$ denote the set of slopes on $T_i$, and set $\mathcal{S}(M) = \mathcal{S}(T_1) \times \dots \times \mathcal{S}(T_n)$. We use the shorthand $[\alpha_*]$ to denote a tuple of slopes $([\alpha_1], \dots, [\alpha_n]) \in \mathcal{S}(M)$ where $\alpha_i \in H_1(T_i; \mathbb{R}) \setminus \{ 0\}$ for $i = 1, \ldots , n$. Below, we provide a rough sketch of the program for using slope detection to attack the L-space conjecture in the case of toroidal irreducible rational homology $3$-spheres, in order to properly frame the contributions made in this manuscript. The program proceeds in three parts.
\begin{enumerate}
    \item Define sets of order-detected, foliation-detected and NLS-detected slopes in $\mathcal{S}(M)$.
    \item Show that for every compact, connected $3$-manifold $M$ whose boundary consists of a disjoint union of incompressible tori, the three sets in (1) are equal. This step seems to require that an ancillary notion of representation-detection be added to (1), as in \cite{BC17}.
    \item Given an irreducible rational homology $3$-sphere $W$, express $W$ as a union $\cup_{j}M_j$ of compact connected pieces glued together along a family $T_1, \dots, T_n$ of disjoint incompressible tori. Set $\mathcal{S}(W) = \mathcal{S}(T_1) \times \dots \times \mathcal{S}(T_n)$ and let $p_j : \mathcal{S}(W) \rightarrow \mathcal{S}(M_j)$ denote the map which extracts the slopes corresponding to the boundary tori of $M_j$. Show that there exists $[\alpha_*] \in \mathcal{S}(W)$ such that $p_j([\alpha_*])$ is $X$-detected (where $X$ is one of order-, foliation-, or NLS-) for all $j$ if and only if $W$ has the property corresponding to $X$.
\end{enumerate}

While steps (1) and (3) are relatively well understood in the cases of NLS-detection and foliation-detection \cite{RR17, Ras17, HW23, BGYpreprinta, BGYpreprintb}, the case of order-detection and the supporting notion of representation-detection are not so well developed, though there are significant treatises on each \cite{BGYpreprinta, BC17, BC23}. This underdevelopment is due to an additional technical complication that arises in the cases of order-detection and representation-detection: aside from regular detection and strong detection, there is a third kind of detection which is more natural to define and often easier to work with, but which is inadequate to carry out step (3) in the program outlined above. We call this third kind of detection weak detection.


For manifolds with a single boundary component, all three kinds of order-detection are introduced and studied in \cite{BC23}, with a brief mention of multiple incompressible torus boundary components. The manuscript \cite[Section 6]{BGYpreprinta} studies weak order-detection and order-detection, including for manifolds with multiple boundary components. The present manuscript establishes a definition of order-detection in full generality, allowing multiple incompressible torus boundary components, with each either having a slope that is weakly, strongly, or regularly order-detected (Definition \ref{def:order-detect}). Our definitions reduce to those of \cite{BGYpreprinta, BC23} if we restrict our attention to manifolds with a single boundary component, or if we ignore strong detection. In parallel, we establish a definition of representation-detection for manifolds with multiple incompressible torus boundary components, with each slope being either weakly, strongly, or regularly representation-detected (Definition \ref{representation-detect}).

Given a $3$-manifold $M$ as earlier in the introduction, our definition of detection therefore involves detection of $(J, K; [\alpha_*])$, where $J \subset K$ are subsets of $\{ 1, \dots, n\}$ that record which slopes of the tuple $[\alpha_*]$ are strongly detected and which are regularly detected, respectively, rather than just weakly detected. We prove the following.

\begin{theorem}
    \label{prop:equbetweenorddetectandrepdetect}
    Suppose $M$ is a compact connected irreducible orientable $3$-manifold whose boundary is a union of incompressible tori $T_1, \ldots, T_n$; fix $J \subset K \subset \{1, \dots, n\}$ and $[\alpha_*] \in \mathcal{S}(M)$. Then $(J, K; [\alpha_*])$ is order-detected if and only if it is representation-detected.
\end{theorem}

Our definitions and Theorem \ref{prop:equbetweenorddetectandrepdetect} together complete steps (1) and (2) in the program outlined above, for the cases of order-detection and representation-detection. In this article, it will be understood that the $3$-manifolds discussed are compact, connected, irreducible, and orientable unless specified otherwise.

We also establish a relative gluing theorem. This theorem roughly says that if $W = \cup_j M_j$ is as in (3) above, but $\partial W$ is nonempty, then the boundary behaviour of left-orderings and representations on each piece $M_i$ carries over to the manifold $W$. Here is a simplified version of our relative gluing theorem which requires the gluing map to identify two detected slopes satisfying an additional technical assumption; in general, we are able to weaken this technical assumption and are even able to glue along weakly detected slopes in some special cases. See Theorems \ref{thm:main_gluing_thm} and \ref{thm:special_gluing_case}.

\begin{theorem}\label{thm:introduction_gluing}
    Let $M_1$ and $M_2$ be 3-manifolds such that for $i=1,2$ the boundary $\partial M_i$ is a union of incompressible tori $T_{i,1}\sqcup T_{i,2}\sqcup\dots\sqcup T_{i,{r_i}}$, and such that there exists a left-ordering $\ooo_i$ of $\pi_1(M_i)$ that order-detects $(J_i,K_i;[\alpha_{i,1}],[\alpha_{i,2}],\dots,[\alpha_{i, r_i}])$ but does not order-detect $(J_i \cup \{1\},K_i;[\alpha_{i,1}],[\alpha_{i,2}],\dots,[\alpha_{i, r_i}])$. Further suppose that $f:T_{1,1}\to T_{2,1}$ is a homeomorphism that identifies $[\alpha_{1,1}]$ with $[\alpha_{2,1}]$. Reindex the boundary components $T_{1,2}, T_{1,3}, \ldots, T_{1,r_1}, T_{2,2}, T_{2,3}, \ldots , T_{2,r_2}$ of the manifold $M_1\cup_f M_2$ as $T_1, \ldots, T_{r_1+r_2-2}$ respectively. Set $J_1' = \{ n-1 : n \in J_1, n\geq2 \}$, $K_1' = \{ n -1 : n\in K_1, n\geq 2\}$, $J_2' = \{n+r_1-2 : n \in J_2, n \geq 2 \}$ and $K_2' = \{n+r_1-2 : n \in K_2, n \geq 2 \}$.
    If $1 \in K_i$ for $i = 1, 2$ then $\pi_1(M_1\cup_f M_2)$ is left-orderable and admits a left-ordering detecting $$(J_1'\cup J_2', K_1'\cup K_2';[\alpha_{1,2}],\dots,[\alpha_{1, r_1}], [\alpha_{2,2}],\dots,[\alpha_{2, r_2}]).$$
\end{theorem}

This result complements the gluing theorems in the literature that deal with the decompositions of closed manifolds \cite[Theorem 7.6]{BGYpreprinta}, \cite[Section 11]{BC17}, and expands upon the relative gluing theorem of \cite[Theorem 7.10]{BC23} by including considerations of weak and strong detection. However, the converse of each of these orderability gluing theorems is essential to the program of using detected slopes to address the L-space conjecture, and at present we do not know if the converse holds \cite[Conjecture 1.5]{BC23}.

To demonstrate the utility of allowing multiple boundary components, and of our relative gluing theorems, we provide a careful analysis of cable knots and, more generally, $(p,q)$-cable spaces glued to a $3$-manifold $M$ with a known set of detected slopes (see Theorem \ref{thm:cpqtord} and Corollary \ref{cor:specialcasedetect}). For the case of cable knots in $S^3$, we are able to produce order-detection results similar to the L-space results of Hedden and Hom \cite{Hedden09, Hom11} which describe the behaviour of L-space knots with respect to cabling.

\begin{theorem}[See Theorem \ref{thm:torusknot} and Corollary \ref{cor:cableintervals}]
    \label{thm:introcables}
    Suppose that $K$ is a nontrivial knot in $S^3$ and let $M$ denote its knot complement. Given integers $p\geq1$ and $q> 1$ that are coprime, let $M'$ denote the knot complement of $K'$, the $(p,q)$-cable of $K$.
    \begin{enumerate}
        \item If $2g(K)-1 < p/q$, then the set of order-detected slopes on $\partial M'$ contains $[-\infty, 2g(K')-1]$; the set of order-detected slopes on $\partial M'$ is equal to $[-\infty, 2g(K')-1]$ if the set of weakly order-detected slopes on $\partial M$ is contained in $[-\infty, 2g(K)-1]$.

        \item If $2g(K)-1 > p/q$, then the set of order-detected slopes on $\partial M'$ contains $[-\infty, pq-p] \cup [2g(K')-1, \infty]$; the set of order-detected slopes on $\partial M'$ is equal to $\RR \cup \{ \infty\}$ if the set of order-detected slopes on $\partial M$ contains $[-\infty, 2g(K)-1]$.

    \end{enumerate}
\end{theorem}

This improves upon the results of \cite{CW11}, which shows that the set of strongly order-detected slopes on $\partial M'$ contains $[-\infty, pq-p-q] \cap \mathbb{Q}$, and the set of weakly order-detected slopes does not contain sufficiently large positive slopes when $2g(K)-1 <p/q$. 

Our cabling result and its technique of proof are also similar to \cite[Section 5]{Ras17}, which describes the behaviour of NLS-detected slopes with respect to cabling \cite[Theorem 5.2]{Ras17}.
The technique is to use results of Jankins, Neumann and Naimi to precisely compute all representation-detected pairs of slopes on the boundary of a $(p,q)$-cable space (as remarked in \cite{Ras17}, this could also be done using \cite{CalWal11}). The results of \cite{Ras17} then follow from Jankins-Neumann-Naimi type calculations and L-space gluing theorems, while in our case we must deal with the additional technical obstacle of having three kinds of order-detection, and incomplete knowledge of the behaviour of left-orderings with respect to gluing and Dehn filling (e.g. \cite[Conjecture 1.5]{BC23} remains open, and it is unknown what form the set of order-detected slopes on the boundary of a knot manifold may take). Thus, in the general setting, we are only able to arrive at containments of sets of detected slopes (Theorem \ref{thm:cpqtord}), where \cite{Ras17} has equalities (\cite[Theorem 5.2]{Ras17}).

\subsection{Organisation of the paper} Section \ref{sec:defs} reviews notions related to orderability, and introduces the definitions of order-detection and representation-detection in Subsections \ref{subsec:order-detect} and \ref{subsec:rep-detect}. In Section \ref{sec:equiv} we prove Theorem \ref{prop:equbetweenorddetectandrepdetect}. In Section \ref{sec:gluing} we introduce the Bludov-Glass theorem for left-ordering amalgams of left-orderable groups, review known gluing theorems, and prove Theorem \ref{thm:main_gluing_thm}, our main gluing theorem. Section \ref{sec:seifertmanifold} reviews JN-realisability and introduces our conventions for Seifert fibered manifolds. In Section \ref{sec:cable} we discuss our conventions for cable spaces and introduce special cases of our main gluing theorem for use in analysing cable knots. Section \ref{sec:cpqcalculations} introduces intervals of relatively JN-realisable slopes, and computes their properties. In Section \ref{sec:finally cables}, we tie everything together to analyse the gluing of cable spaces to knot manifolds (Theorem \ref{thm:cpqtord}), and explicitly calculate intervals of relatively JN-realisable slopes in order to provide a more detailed analysis in the case of cable knots in $S^3$ (Theorem \ref{thm:torusknot}).

\section{Definitions}
\label{sec:defs}

A left-ordering $\mathfrak{o}$ of a group $G$ is determined by a strict, total ordering $<_{\mathfrak{o}}$ of its elements such that $g<_{\mathfrak{o}}h$ implies $fg <_{\mathfrak{o}} fh$ for all $f, g,h \in G$. Such an ordering can also be specified by a positive cone $P(\mathfrak{o})$, which is a subset of $G$ satisfying $G \setminus \{ id \} = P(\mathfrak{o}) \sqcup P(\mathfrak{o})^{-1}$ and $P(\mathfrak{o})\cdot P(\mathfrak{o}) \subset P(\mathfrak{o})$. The correspondence between strict total orderings and positive cones is determined by the prescription
\[ g<_{\mathfrak{o}} h \mbox{ if and only if } g^{-1}h \in P(\mathfrak{o}).
\]
We denote the set of all left-orderings of $G$ by $\LO(G)$, and topologise this set as follows. Using the correspondence between orderings and cones, we can view $\LO(G)$ as a subset of the power set $\mathcal{P}(G) = \{0, 1\}^G$. We topologise $ \{0, 1\}^G$ using the product topology and give $\LO(G)$ the subspace topology. This makes $\LO(G)$ into a compact space, since it turns out to be a closed subset of the compact space $ \{0, 1\}^G$. The space $\LO(G)$ is also totally disconnected and Hausdorff, since $ \{0, 1\}^G$ has these properties, and it is metrisable if $G$ is countable \cite[Proposition 1.3]{Sik04}.

By setting $P(g \cdot \ooo)=gP(\ooo)g^{-1}$ for every $g\in G$, we obtain a new left-ordering $g\cdot\ooo$, where $f<_{g \cdot \ooo}h$ if and only if $fg<_\ooo hg$ for all $f,h\in G$. This construction gives an action of $G$ on $\LO(G)$, one can check this is an action by homeomorphisms. More generally, if we have an injective homomorphism $\phi:H\to G$, then the left-ordering $\ooo$ on $G$ induces a left-ordering $\phi^{-1}(\ooo)$ on $H$ by defining $h_1<_{\phi^{-1}(\ooo)} h_2$ if $\phi(h_1)<_\ooo \phi(h_2)$ for $h_1,h_2\in H$. Note that $g \cdot \ooo$ is a special case of this construction, using an inner automorphism in place of $\phi$.

A subgroup $C$ of $G$ is \emph{convex} relative to the ordering $\mathfrak{o}$ of $G$ (or $\mathfrak{o}$-convex for short) if for all $c_1, c_2 \in C$ and $g \in G$, the implication $c_1 <_{\mathfrak{o}} g <_{\mathfrak{o}} c_2$ implies $g \in C$ holds. A subgroup $C$ of $G$ is \emph{relatively convex} if there exists an ordering $\mathfrak{o}$ of $G$ relative to which $C$ is convex. Whenever $C$ is relatively convex, the set of left cosets $\{gC\}_{g \in G}$ inherits a natural quotient ordering and becomes a left $G$-set with the natural action given by left-multiplication by elements of $G$, which is order-preserving with respect to the quotient ordering.






\subsection{Dynamic realisations and representations}
When $G$ is countable, left-orderings correspond to certain kinds of actions on $\mathbb{R}$ via the construction of dynamic realisations, which we summarize below.

Let $G$ be a countable group with a left-ordering $\mathfrak{o}$. A \textit{gap} in $(G, \mathfrak{o})$ is a pair of elements $g, h \in G$ with $g<_{\mathfrak{o}}h$ such that no element $f \in G$ satisfies $g<_{\mathfrak{o}}f<_{\mathfrak{o}}h$. Gaps exist if and only if the ordering $<_{\mathfrak{o}}$ of $G$ is discrete, or equivalently, if the positive cone $P(\ooo)$ admits a least element.

An order-preserving embedding $t: (G, \mathfrak{o}) \rightarrow (\mathbb{R}, <)$ is called a \textit{tight embedding} if, whenever $(a, b) \subset \mathbb{R} \setminus t(G),$ there is a gap $g, h \in G$ with $g <_{\mathfrak{o}} h$ such that $(a,b) \subset (t(g), t(h))$ (see e.g. \cite{BC23}). For ease of exposition, we note that by conjugating by an appropriate translation, we can always require that our tight embeddings satisfy $t(id) =0$. We make this assumption from here forward.

If $G$ is countable, there is a standard method for constructing a tight embedding from a given enumeration $\{g_0=id,g_1,g_2,\dots\}$ of $G$, see e.g. \cite{Navas10a}. The \emph{ tight embedding associated with the enumeration} can be constructed as follows: first set $t(g_0)=0$; if $t(g_0),t(g_1),\dots,t(g_i)$ have already been defined, then \[t(g_{i+1}) = \begin{cases}
        \max\{t(g_0),t(g_1),\dots,t(g_i)\}+1 & \quad\mbox{if } g_{i+1}>\max\{g_0,g_1,\dots,g_i\}  \\
        \min\{t(g_0),t(g_1),\dots,t(g_i)\}-1 & \quad\mbox{if } g_{i+1}<\min\{g_0,g_1,\dots,g_i\};
    \end{cases}\]
otherwise, $t(g_{i+1})=\frac{t(g_j)+t(g_k)}{2}$ if $g_j<g_{i+1} <g_k$ and $j$, $k$ are chosen so that there is no $l\in\{0,1,\dots,i\}$ such that $ g_j<g_l <g_k$.

Now given a tight embedding $t$, we define a homomorphism $\rho_{\mathfrak{o}} : G \rightarrow \mathrm{Homeo}_+(\mathbb{R})$ as follows. For each $g \in G$ and $x \in \mathbb{R}$, define $\rho_{\mathfrak{o}}(g)(x)$ according to:

\begin{itemize}
    \item If $x = t(h) \in t(G)$ for some $h \in G$, then define $\rho_{\mathfrak{o}}(g)(t(h)) = t(gh)$.
    \item If $x \in \overline{t(G)}\setminus t(G)$, then choose a sequence $\{t(g_i)\} \subset t(G)$ converging to $x$, and define $\rho_{\mathfrak{o}}(g)(x) = \lim t(g g_i)$.
    \item If $x \in \mathbb{R} \setminus \overline{t(G)}$, then there must be a gap $h, k$ in $G$ with $x \in (t(h), t(k))$. Write $x = (1-s)t(h) + st(k)$ for some $s \in (0,1)$ and set $ \rho_{\mathfrak{o}}(g)(x) = (1-s)t(gh) + st(gk).$
\end{itemize}

The resulting homomorphism $\rho_{\mathfrak{o}}$ is a \textit{dynamic realisation} of $(G,\mathfrak{o})$; one can check that it is well-defined (i.e. independent of the choice of tight embedding) up to conjugation by a homeomorphism of $\mathbb{R}$.


Moreover, from a dynamic realisation one can recover the ordering $\ooo$ by examining the orbit of $t(id) = 0$:
\[(\forall g, h \in G) [g<_{\mathfrak{o}}h \iff t(g) < t(h) \iff \rho_{\mathfrak{o}}(g)(0) < \rho_{\mathfrak{o}}(h)(0)].
\]


\subsection{Order-detection}

\label{subsec:order-detect}
The definitions in this section generalise those appearing in \cite{BC23} from the case of knot manifolds to $3$-manifolds with multiple incompressible torus boundary components.

Recall that for every left-ordering $\mathfrak{o}$ of $\mathbb{Z} \oplus \mathbb{Z}$, there is a corresponding line $\mathcal{L}(\mathfrak{o})$ in $\mathbb{R}^2 \cong (\mathbb{Z} \oplus \mathbb{Z}) \otimes \mathbb{R}$, which is completely determined by the prescription that all integer lattice points on either side of $\mathcal{L}(\mathfrak{o})$ must have the same sign (relative to the ordering $\mathfrak{o}$) \cite[Proposition 1.7]{Sik04}. If we topologise the set of lines through the origin in $\mathbb{R}^2$ in the usual way and write $[\ell]$ for the image of a line $\ell$ in $\RR^2$ in the resulting copy of $\RR P^1 \cong S^1$, we arrive at the following lemma \cite{BC23}.

\begin{lemma}\label{slope cts}
    The map $\mathcal{L} : \LO(\mathbb{Z} \oplus \mathbb{Z}) \rightarrow \mathbb{R}P^1$ given by $\mathcal{L}(\mathfrak{o}) = [\mathcal{L}(\mathfrak{o})]$ is continuous.
\end{lemma}




We use this map in the setting of $3$-manifolds as follows. Suppose that $M$ is a compact connected orientable $3$-manifold whose boundary is a union of incompressible tori, say $\partial M = T_1 \cup \dots \cup T_n$. A \emph{slope} on $T_i$ is an element $[\alpha] \in \mathbb{P}H_1(T_i; \mathbb{R})$ (the projective space of $H_1(T_i; \mathbb{R})$), where $\alpha \in H_1(T_i; \mathbb{R}) \setminus \{ 0\}$. Since the boundary tori are incompressible, there are inclusions $\pi_1(T_i) \rightarrow \pi_1(M)$ allowing us to implicitly identify each group $\pi_1(T_i)$ with a subgroup of $\pi_1(M)$ isomorphic to $\mathbb{Z} \oplus \mathbb{Z}$. We fix such an identification for each $i$ and from here forward simply write $\pi_1(T_i) \subset \pi_1(M)$.

We use $\mathcal{S}(T_i)$ to denote the set of all slopes on $T_i$, topologized so that $\mathcal{S}(T_i)$ is homeomorphic to $S^1$ as in the previous lemma, and write $\mathcal{S}(M) = \mathcal{S}(T_1) \times \dots \times \mathcal{S}(T_n)$. Using $r_i : \LO(\pi_1(M)) \rightarrow \LO(\pi_1(T_i))$ to denote the restriction map, we define the \emph{slope map} $s: \mathrm{LO}(\pi_1(M)) \rightarrow \mathcal{S}(M)$ by $s(\mathfrak{o}) = ([\mathcal{L}(r_1(\mathfrak{o}))], \ldots, [\mathcal{L}(r_n(\mathfrak{o}))])$. As the restriction map is continuous, so is the slope map by Lemma \ref{slope cts}. We will use $s_i : \LO(\pi_1(M)) \rightarrow \mathcal{S}(T_i)$ to denote the composition of $s$ with projection onto the $i$-th factor, equivalently, $s_i$ is the composition $\mathcal{L} \circ r_i$.

Identifying $H_1(T_i; \mathbb{Z})$ with the integer lattice points in $H_1(T_i; \mathbb{R})$, we define a slope $[\alpha]$ to be \emph{rational} if $\alpha \in H_1(T_i; \mathbb{Z})$, and \emph{irrational} otherwise. We will call a tuple $([\alpha_1], \dots , [\alpha_n])$ of slopes rational if $[\alpha_i]$ is rational for all $i$. If $[\alpha]$ is rational, then we always assume that $\alpha$ is primitive. In terms of slopes arising from orderings, one can show that this means $[\mathcal{L}(r_i(\mathfrak{o}))]$ is rational if $\mathcal{L}(r_i(\mathfrak{o})) \cap H_1(T_i; \mathbb{Z}) \cong \mathbb{Z}$, otherwise the slope is irrational.

\begin{definition}\label{def:order-detect}
    Suppose that $M$ is a compact connected orientable $3$-manifold with boundary $\partial M = T_1 \cup \dots \cup T_n$ a union of incompressible tori, and let $J \subset K \subset \{ 1, \ldots, n\}$ and $([\alpha_1], \ldots, [\alpha_n])\in \mathcal{S}(M)$. We say that $(J, K; [\alpha_1], \ldots, [\alpha_n])$ is \emph{order-detected} if there exists $\mathfrak{o} \in \LO(M)$ such that
    \begin{enumerate}
        \item[O1.] $s(\mathfrak{o}) = ([\alpha_1], \ldots, [\alpha_n])$;
        \item[O2.] for all $g \in \pi_1(M)$, we have $s(g \cdot \ooo) = ([\beta_1], \ldots, [\beta_n])$ where $[\beta_i] = [\alpha_i]$ for all $i \in K$;
        \item[O3.] there exists an $\mathfrak{o}$-convex normal subgroup $C$ such that if $[\alpha_i]$ is rational then $\pi_1(T_i)\cap C\leq \langle \alpha_i \rangle $ with $\pi_1(T_i)\cap C = \langle \alpha_i \rangle $ whenever $i \in J$, and if $[\alpha_i]$ is irrational then $\pi_1(T_i)\cap C=\{ id \}$.

    \end{enumerate}
\end{definition}

In this case, we also say $(J, K; [\alpha_1], \ldots, [\alpha_n])$ is order-detected by $\ooo$, or sometimes we say $\ooo$ order-detects $(J, K; [\alpha_1], \ldots, [\alpha_n])$. For short, we often write $(J, K; [\alpha_*])$ in place of $(J, K; [\alpha_1], \ldots, [\alpha_n])$. From this definition, if $(J, K; [\alpha_1], \ldots, [\alpha_n])$ is order-detected and $i \in K$ corresponds to an irrational slope $[\alpha_i]$, then $(J\cup\{i\}, K; [\alpha_1], \ldots, [\alpha_n])$ is also order-detected. We write $\mathcal{D}_{ord}(J,K;M)\subset \mathcal{S}(M)$ to denote the set of tuples $[\alpha_*]$ for which $(J,K;[\alpha_*])$ is order-detected. If $(J, K; [\alpha_1], \ldots, [\alpha_n])$ is order-detected
, we say that $[\alpha_i]$ weakly order-detected; it is strongly order-detected if $i \in J$, and (regularly) order-detected if $i \in K$.

A special case is when $M$ is the complement of a nontrivial knot in $S^3$, or more generally, when $M$ is a knot manifold, that is, a compact connected irreducible orientable 3-manifold not homeomorphic to $D^2\times S^1$ with incompressible torus boundary. In these special cases, the language we have just introduced (strong detection, weak detection, detection) agrees with \cite{BC23}.





\subsection{Representation-detection}
\label{subsec:rep-detect}
All representations of groups in our discussions will be assumed to have images in $\mathrm{Homeo}_+(\mathbb{R})$ unless otherwise stated. A \emph{pointed representation} of a group $G$ is a pair $(\rho, x_0)$ where $\rho$ is a representation $\rho:G \rightarrow \mathrm{Homeo}_+(\mathbb{R})$ and $x_0 \in \mathbb{R}$ is a choice of basepoint which is not a global fixed point for the action of $G$ on $\mathbb{R}$ determined by $\rho$. We use the notation $\mathcal{R}(G)$ to denote the set of all pointed representations of $G$.

\begin{lemma}\label{lemma:z2slope}
    If $G = \mathbb{Z} \oplus \mathbb{Z}$, then a pointed representation $(\rho, x_0)$ of $G$ determines a line $\mathcal{L}(\rho, x_0)$ in $\mathbb{R}^2 \cong (\mathbb{Z} \oplus \mathbb{Z}) \otimes \mathbb{R}$ according to the prescription that all elements $g$ of $\mathbb{Z} \oplus \mathbb{Z}$ lying to one side of $\mathcal{L}(\rho, x_0)$ must satisfy $\rho(g)(x_0) > x_0$, and those to the other side $\rho(g)(x_0) < x_0$.
\end{lemma}
\begin{proof}
    First consider the case where $\rho$ is injective. Let $\{y_0=x_0,y_1,y_2,\dots\}$ be an enumeration of a dense subset of $\RR$. This enumeration determines a left-ordering $\ooo$ on $\Homeo_+(\RR)$, namely, $f<_{\ooo}g$ if there is an $i$ such that $f(y_i)< g(y_i)$ and $f(y_j)=g(y_j)$ for all $j=0,1,\dots,i-1$. Let $\rho^{-1}(\ooo)$ be the pull back of $\ooo$. Then $\rho^{-1}(\ooo)$ is a left-ordering on $G$ and induces a line $\LL(\rho^{-1}(\ooo))=\LL(\rho,x_0)$. By definition, for an element $g\in G$, we have $id<_{\rho^{-1}(\ooo)} g$ if $\rho(id)(x_0)=x_0<\rho(g)(x_0)$, and $g <_{\rho^{-1}(\ooo)} id$ if $\rho(g)(x_0)<x_0$. Note that the line $\LL(\rho,x_0)$ is completely determined by these prescriptions and is independent of the choice of a dense countable subset of $\RR$.

    If $\rho$ is not injective, then $\ker(\rho)\cong \ZZ$, and there is an injective map $\rho':G/\ker(\rho) \cong \ZZ \to \mathrm{Homeo}_+(\mathbb{R})$ with $x_0$ not being a fixed point of any element in $G/\ker(\rho)$ except the identity. Let $s$ be a generator of $G/\ker(\rho)$. Then either $\rho'(s)(x_0)>0$ or $\rho'(s)(x_0)<0$. If $\rho'(s)(x_0)>0$, then $\rho'(s^n)(x_0)>0$ for $n>0$ and $\rho'(s^n)(x_0)<0$ for $n<0$. The elements of the cosets $s^n\ker(G)$ exhibit the same property. The same applies to the case $\rho'(s)(x_0)<0$. The statement of this lemma follows.
\end{proof}

\begin{definition}
    Suppose $G$ is a group and $\rho_1,\rho_2$ are representations of $G$. Then $\rho_1$ is said to be \emph{semi-conjugate} to $\rho_2$ if there is a proper\footnote{Here properness means the preimage of any bounded set is bounded.} non-decreasing map $h:\RR\to \RR$ such that $h\circ \rho_1(g)=\rho_2(g)\circ h$ for all $g \in G$.
\end{definition}
\begin{lemma}\label{lemma:semi-conj}
    Let $\ooo$ be a left-ordering of a group $G$ and $C$ an $\ooo$-convex normal subgroup of $G$. If $t: G\to\RR$ is a tight embedding associated to an enumeration $\{g_0=id,g_1,g_2,\dots\}$ of $G$ and $\rho$ its associated dynamic realisation, then there is a dynamic realisation $\eta: G/C\to \Homeo_+(\RR)$ such that $\rho$ is semi-conjugate to $\eta \circ p$ where $p:G\to G/C$ is the canonical projection. Moreover, the proper non-decreasing map $\nu :\RR \rightarrow \RR$ demonstrating the semiconjugacy satisfies $\nu(0) =0$.
\end{lemma}

\begin{proof}
    Since $C$ is convex, $P(\ooo)\backslash C$ is a union of left $C$-cosets. In other words, given two cosets $gC$ and $hC$, we have either $gc<_\ooo hc'$ for all $c', c\in C$ or $hc'<_\ooo gc$ for all $c', c\in C$. Therefore, $t(gC)$ is bounded for each coset $gC$ and so $\inf\{t(gC)\}$ and $\sup\{t(gC)\}$ exist.


    Fixing a tight embedding $\omega: G/C\to \RR$ with $\omega(C)=0$ and $\eta$ its dynamic realisation, we define $\nu: \RR \to \RR$ by the following prescription: if $r=t(g)$ for some $g\in G$, then $\nu(r)=\omega(gC)$; if $r\in (t(g),t(h))$ for some gap $g<_\ooo h$ in $G$, then $\nu(r)=\omega(hC)=\omega(gC)$ since $g^{-1}h$ is the least positive element in $G$ and hence $gC=hC$; finally if $r\in \overline{t(G)}\setminus t(G)$, then set $\nu(r) = \sup\{\nu(g) : t(g) < r \}.$

    Note that the image of $\nu$ is unbounded and $\nu$ is a well-defined non-decreasing function, which by definition satisfies $\nu(0) = 0$. We check that it provides the required function to prove the semiconjugacy relation claimed.

    For any bounded subset $I$ of $\RR$, choose cosets $gC$ and $hC$ such that $\omega(gC)$ and $\omega(hC)$ are lower and upper bounds of $I$ respectively. Then $\nu^{-1}(I)$ is bounded by $t(g')$ and $t(h')$, where $g'<_\ooo gc$ and $hc<_\ooo h'$ for any $c\in C$. This implies $\nu$ is proper. Fix an arbitrary element $g\in G$. It is left to check that $\nu\circ \rho(g) (x)=\eta \circ p(g)\circ \nu(x)$ for all $x\in \RR$. Firstly, note $\nu\circ \rho(g) (t(h)) =\nu(t(gh))=\omega(ghC) $ and $ \eta \circ p(g)\circ \nu(t(h))= \eta(gC) \omega(hC)=\omega(ghC)$ for any $g, h\in G$. In other words, $\nu\circ \rho(g)$ and $\eta \circ p(g)\circ \nu$ agree on $t(G)$. Secondly, if $x\notin \overline{t(G)}$, then $t(k)<x<t(h)$ and $x=(1-s)t(k)+s t(h)$ for some gap $k<_\ooo h$ and $s\in(0,1)$. It follows that $k,h\in kC=hC$ and so $\eta \circ p(g)\circ \nu(x)= \eta(gC)(\omega(hC)) =\omega(ghC)$. On the other hand, we have $\nu\circ \rho(g) (x)=\nu( (1-s)t(gk)+st(gh))=\omega(ghC)$. So $\nu\circ \rho(g)$ and $\eta \circ p(g)\circ \nu$ also agree on $\RR\setminus \overline{t(G)}$. Finally, consider the case $x\in \overline{t(G)}\setminus t(G)$. We have
    \begin{align*}
        \eta \circ p(g) \circ \nu(x) & = \eta(gC)(\sup\{\omega(hC) : t(h) < x\}) \\
                                     & = \sup \{ \omega(ghC) : t(h) <x \}        \\
                                     & = \sup \{ \nu(t(gh)) : t(h) <x \}         \\
                                     & = \sup \{ \nu(t(h)) : t(g^{-1}h) <x \}    \\
                                     & = \nu \circ \rho(g)(x).
    \end{align*}
\end{proof}

Let $M$ be a compact connected irreducible orientable $3$-manifold not homeomorphic to $D^2\times S^1$ or $S^1\times S^1 \times I$, whose boundary $\partial M = T_1 \cup \dots \cup T_n$ consists of incompressible tori as above. For each $i \in \{ 1, \ldots, n\}$, if $\pi_1(T_i) \not\subset \Stab_{\rho}(x_0)$ for some $(\rho, x_0) \in \mathcal{R}(\pi_1(M))$, then $(\rho, x_0)$ determines an element $(\rho|_{\pi_1(T_i)}, x_0) \in \mathcal{R}(\pi_1(T_i))$ via restriction of $\rho$ to the subgroup $\pi_1(T_i)$. We therefore focus on the subset $\mathcal{R}^*(\pi_1(M)) \subset \mathcal{R}(\pi_1(M))$, where
\[\mathcal{R}^*(\pi_1(M)) = \{(\rho, x_0) \in \mathcal{R}(\pi_1(M)) : \pi_1(T_i) \not\subset \Stab_{\rho}(x_0) \mbox{ for $i = 1, \ldots, n$} \}.
\]

\begin{definition}\label{representation-detect}
    Suppose $M$ is a compact connected irreducible orientable $3$-manifold with boundary $\partial M = T_1 \cup \dots \cup T_n$ a union of incompressible tori, and let $J \subset K \subset \{ 1, \ldots, n\}$ and $([\alpha_1], \ldots, [\alpha_n]) \in \mathcal{S}(M)$. We say that $(J, K; [\alpha_1], \ldots, [\alpha_n])$ is \emph{representation-detected} if there exists $(\rho, x_0) \in \mathcal{R}^*(\pi_1(M))$ such that
    \begin{enumerate}
        \item[R1.] $([\mathcal{L}(\rho|_{\pi_1(T_1)}, x_0)], \ldots,[\mathcal{L}(\rho|_{\pi_1(T_n)}, x_0)]) = ([\alpha_1], \ldots, [\alpha_n])$;
        \item[R2.] for every $g\in \pi_1(M)$, we have $(\rho, \rho(g)(x_0)) \in \mathcal{R}^*(\pi_1(M))$ with
              $$([\mathcal{L}(\rho|_{\pi_1(T_1)}, \rho(g)(x_0))], \ldots,[\mathcal{L}(\rho|_{\pi_1(T_n)}, \rho(g)(x_0))]) = ([\beta_1], \ldots, [\beta_n]),$$
              where $[\beta_i] = [\alpha_i]$ for all $i \in K$;
        \item[R3.] $\rho$ is semi-conjugate to a representation $\varphi: \pi_1(M)\to \Homeo_+(\RR)$ via some $\nu:\RR\to \RR$ such that $(\varphi, \nu(x_0))\in \mathcal{R}^*(\pi_1(M))$ and $[\mathcal{L}(\varphi|_{\pi_1(T_i)}, \nu(x_0))] = [\alpha_i]$, with $\varphi(\alpha_i)=id$ whenever $i\in J$ and $[\alpha_i]$ is rational.
    \end{enumerate}
\end{definition}

In this case we say that $(J, K; [\alpha_1], \ldots, [\alpha_n])$ is representation-detected by $(\rho, x_0)$. We write $\mathcal{D}_{rep}(J,K;M)\subset \mathcal{S}(M)$ to denote the set of tuples $[\alpha_*]$ for which $(J,K;[\alpha_*])$ is representation-detected. As in the case of orderings, if $(J,K;[\alpha_1], \ldots, [\alpha_n])$ is representation-detected, we say that $[\alpha_i]$ weakly representation-detected; it is strongly representation-detected if $i \in J$ and (regularly) representation-detected if $i \in K$.




\section{Equivalence of order- and representation-detection, and alternative definitions}
\label{sec:equiv}

We first begin with a lemma that allows us to restate the third condition of order-detection.

\begin{lemma}\label{lemma:equivO3}
    With the same assumptions as in Definition \ref{def:order-detect}, there exists a left-ordering $\ooo'$ of $\pi_1(M)$ satisfying O1 and O2 and \begin{itemize}
        \item [O3$'$.] there exists an $\mathfrak{o}'$-convex subgroup $H$ such that if $[\alpha_i]$ is rational then $\pi_1(T_i)\cap H \leq \langle \alpha_i \rangle $ with $\langle \langle \alpha_i \rangle \rangle \leq H $ whenever $i \in J$; and if $[\alpha_i]$ is irrational then $\pi_1(T_i)\cap H=\{ id \}$,
    \end{itemize} if and only if $(J, K; [\alpha_1], \ldots, [\alpha_n])$ is order-detected.
\end{lemma}

\begin{proof}
    Assume first that $(J, K; [\alpha_1], \ldots, [\alpha_n])$ is order-detected, say by an ordering $\ooo$ of $\pi_1(M)$. Then O1, O2 are satisfied and the subgroup $C$ from O3 automatically satisfies the properties required of $H$ in O3$'$. So in this direction, there is nothing to prove.

    Now assume there is a left-ordering $\ooo'$ satisfying O1, O2 and O3$'$. Since $H$ is $\ooo'$-convex, for every $g\in \pi_1(M)$ one can check that the conjugate $gHg^{-1}$ is $g \cdot \ooo'$-convex. For ease of exposition, write $G=\pi_1(M)$ and $C=\bigcap_{g\in G} gHg^{-1}$. Then $C$ is normal and $\ooo$-convex for some left-ordering $\ooo$ by \cite[Proposition 5.1.10]{KM96}; we will provide a construction of such an ordering $\ooo$ and verify that it satisfies O1, O2 and O3.

    Since $H$ is $\ooo'$-convex, the set of left cosets $G/H$ inherits a natural ordering $\prec$ and the canonical $G$-action from the left preserves this ordering, the kernel of this $G$-action being $C$. Choose a complete set of coset representatives $E=\{g_0=id, g_1,g_2,\dots\}$ and define the ordering $\ooo$ according to the rule $ g <_{\ooo} h $ if $ g g_i H \prec h g_i H $ where $g_i H$ is the first element in the enumeration for which $g g_iH \neq h g_i H$; or if $g g_i H= h g_i H$ for all $i$ and $g <_{\ooo'} h$.

    For any $i\in \{1,2,\dots,n\}$, let $\gamma\in \pi_1(T_i)$ with $\gamma\notin \langle \alpha_i \rangle$ if $[\alpha_i]$ is rational. From O3$'$, we have $\gamma\notin H$, and so $H\not= \gamma H$. Therefore, $id<_\ooo \gamma$ if and only if $H\prec \gamma H$ by the construction of $\ooo$, which is implied by $id<_{\ooo'} \gamma$. In other words, $id<_\ooo \gamma$ if and only if $id<_{\ooo'} \gamma$ and so $\ooo$ satisfies O1.

    Next, suppose that $i\in K$ and $\gamma \in \pi_1(T_i)$ with $\gamma\notin \langle \alpha_i \rangle$ if $[\alpha_i]$ is rational. Then $id<_{g \cdot \ooo} \gamma$ if and only if $id <_\ooo g^{-1}\gamma g$. If $g^{-1} \gamma g$ lies in $C$, then $id<_{\ooo} g^{-1}\gamma g$ if and only if $id<_{\ooo'} g^{-1}\gamma g$, which is equivalent to $id<_{g \cdot \ooo'}\gamma$. So in this case, $id<_{g \cdot \ooo} \gamma$ if and only if $id<_{g \cdot \ooo'}\gamma$. If $g^{-1}\gamma g$ does not lie in $C$, then $id<_\ooo g^{-1}\gamma g$ if and only if there is an index $j$ such that $g_jH\prec g^{-1}\gamma g g_jH$ and $g_s H=g^{-1}\gamma g g_s H$ for all $s<j$. Note that $g_jH\prec g^{-1}\gamma g g_jH$ if and only if $g_j <_{\ooo'} g^{-1}\gamma g g_j$, which is true if and only if $id <_{gg_j \cdot \ooo'} \gamma $. Combined with the property O2 of $\ooo'$, these two cases allow us to conclude that $id<_{g\cdot \ooo} \gamma$ if and only of $id<_{\ooo'} \gamma$, that is, $\ooo$ satisfies O2. Finally, O3 is automatically true, simply by the construction of $C$ and the property O3$'$ of $H$.
\end{proof}

We are now ready to prove Theorem \ref{prop:equbetweenorddetectandrepdetect} from the introduction.


\begin{proof}[Proof of Theorem \ref{prop:equbetweenorddetectandrepdetect}]
    First, we show that order-detection implies representation-detection. Assume $(J, K; [\alpha_*])$ is order-detected by $\ooo\in \LO(\pi_1(M))$. Let $t:\pi_1(M)\to \RR$ be a tight embedding associated with an enumeration $\{g_0=id,g_1, g_2,\ldots\}$ of $\pi_1(M)$. In this setup, $0=t(g_0)$ is not a fixed point of $\rho(g)$ for any non-identity element $g\in \pi_1(M)$. We claim that $(J, K; [\alpha_*])$ is representation-detected by $(\rho_\ooo, 0)$, where $\rho_\ooo$ is the dynamic realisation associated with $t$. Notice that we have $[\LL(\ooo|_{\pi_1(T_i)})]=[\LL(\rho_\ooo|_{\pi_1(T_i)},0)]$ by O1 and Lemma \ref{lemma:z2slope}, that is, R1 is satisfied.

    If $x=t(g)=\rho(g)(0)$ for some $g\in\pi_1(M)$, then $x$ is not a fixed point of $\rho(h)$ for any non-identity element $h\in \pi_1(M)$. Therefore, we have $(\rho,\rho(g)(0))\in \mathcal{R}^*(\pi_1(M))$ for every $g\in G$. Moreover, $\rho(g_1)(x)<\rho(g_2)(x)$ if and only if $t(g_1g)<t(g_2g)$. Since $\rho$ is constructed from the tight embedding $t$, this is equivalent to saying $g_1g<_\ooo g_2g$, in other words, $g_1 <_{g\cdot \ooo} g_2$. Then by O2, we have \[[\LL(\rho_\ooo|_{\pi_1(T_j)},x)]=[\LL((g \cdot \ooo)|_{\pi_1(T_j)})]=[\LL(\ooo|_{\pi_1(T_j)})]=[\alpha_j]\] for all $j\in K$. Hence, R2 is fulfilled.

    Applying Lemma \ref{lemma:semi-conj} to the convex subgroup $C \leq \pi_1(M)$ using the tight embedding $t$, we obtain a representation $\varphi=\eta \circ p:\pi_1(M)\to \Homeo_+(\RR)$, and $\rho$ is semi-conjugate to this representation by $\nu : \RR \rightarrow \RR$, where $\nu(0)=0$. For any $i\in\{1,\dots,n\}$, let $\gamma\in \pi_1(T_i)$ be given, with $\gamma\notin \vbracket{\alpha_i}$ if $[\alpha_i]$ is rational. Then we see that $\gamma\notin C$ by O3. Because of that $id <_\ooo \gamma$ if and only if $0<t(\gamma)$ and also that $\varphi(\gamma)(0)=\varphi(\gamma)\nu(0)=\nu \rho(\gamma)(0)=\nu(t(\gamma))>0$ if and only if $t(\gamma)>0$ by the construction of $\nu$ in Lemma \ref{lemma:semi-conj}, we see $id <_\ooo \gamma$ if and only if $\varphi(\gamma)(0)>0$. It follows that $(\varphi, \nu(0))\in \mathcal{R}^*(\pi_1(M))$ and, in fact, $[\mathcal{L}(\varphi|_{\pi_1(T_i)}, \nu(0))] = [\alpha_i]$. This proves R3 and finishes the proof in the first direction.

    For the other direction, assume $(J, K; [\alpha_*])$ is representation-detected by $(\rho,x_0)$ with $\varphi$ the representation whose existence is guaranteed by R3, and with $\nu :\RR \rightarrow \RR$ the proper, non-decreasing map demonstrating the semiconjugacy from $\rho$ to $\varphi$. Since $M$ is compact connected irreducible and orientable, a non-trivial representation $\rho:\pi_1(M)\to \Homeo_+(\RR)$ ensures that $\pi_1(M)$ is left-orderable \cite[Theorem 1.1]{BRW05}. By choosing a countable dense subset $E=\{r_0=x_0,r_1,r_2,\dots\}$ of $\RR$, we define a left-ordering $\ooo'$ on $\Homeo_+(\RR)$ in the usual way. If $\ker(\rho)$ is trivial, then $\rho^{-1}(\ooo')$ is a left-ordering on $\pi_1(M)$. If $\ker(\rho)$ is non-trivial, then we can give a left-ordering on $\pi_1(M)$ via the short exact sequence $0\to \ker(\rho)\to \pi_1(M)\to \rho(\pi_1(M))\to 0$, where we take any left-ordering of $\ker(\rho)$ and the restriction left-ordering $\ooo'|_{\rho(\pi_1(M))}$ on $\rho(\pi_1(M))$. By abuse of notation, we also denote this left-ordering of $\pi_1(M)$ as $\ooo'$. By the construction of $\ooo'$, we have $g<_{\ooo'} f$ for $g,f\in \pi_1(M)$, if there is an index $i$ such that $\rho(g)(r_i)<\rho(f)(r_i)$ and $\rho(g)(r_j)=\rho(f)(r_j)$ for all $j=0,1,\dots,i-1$.

    Set $H=\Stab_\varphi(\nu(x_0))$ and order the cosets of $H$ according to $gH \prec fH$ if $\varphi(g)(\nu(x_0))<\varphi(f)(\nu(x_0))$. Now we create the desired left-ordering $\ooo$ on $\pi_1(M)$ by defining $g<_\ooo f$ if either $gH\prec fH$ or $gH=fH$ and $id<_{\ooo'} g^{-1}f\in H$. We claim that $\ooo$ together with $H$ satisfies O1, O2 and O3'. Then by Lemma \ref{lemma:equivO3}, this direction is done. It remains to prove the claim.

    To show O1, we fix $i\in \{1,2,\dots,n\}$ and let $g\in \pi_1(T_i)$ such that if $[\alpha_i]$ is rational, then $g\notin \langle \alpha_i\rangle$. We have $g\notin \ker(\rho)$ by the definition of $\mathcal{R}^*(\pi_1(M))$. It suffices to show that $id<_\ooo g$ if and only if $\rho(g)(x_0)>x_0$. There are two cases depending on whether $g$ lies in $H$ or not. If $g\in H$, then $id<_\ooo g$ if and only if $id<_{\ooo'} g$. Since $g\not\in\ker(\rho)$, the latter is true if there is an index $j$ such that $\rho(g)(r_j)>r_j$ and $\rho(g)(r_s)=r_s$ for all $s=0,1,\dots,j-1$. By R1, the first $r_i$ where they differ is $r_0=x_0$ and so indeed $\rho(g)(x_0)>x_0$ if and only if $id<_\ooo g$. This finishes the first case. For the second, suppose $g \notin H$, then by the construction of $\ooo$, we have $id<_\ooo g$ if and only if $H\prec gH$, which is equivalent to $\varphi(g)(\nu(x_0))>\nu(x_0)$ or equivalently $\nu(\rho(g)(x_0))>\nu(x_0)$. Now if $\rho(g)(x_0)\leq x_0$, we would have $\nu(\rho(g)(x_0))\leq h\nu(x_0)$ since $\nu$ is non-decreasing, a contradiction. Therefore, it must be the case that $\rho(g)(x_0)>x_0$. Putting these two cases together proves O1.

    To show O2, fix an $i\in K$ and take $g\in \pi_1(T_i)$ in such a way that if $[\alpha_i]$ is rational, then $g$ does not lie in $\langle \alpha_i\rangle$. It suffices to show that for any $f\in \pi_1(M)$, we have $id<_\ooo g$ if and only if $id<_{f \cdot \ooo} g$. By the definition of $f \cdot \ooo$, the latter is equivalent to $id<_\ooo f^{-1}gf$. Again, we divide into two cases depending on whether $g$ is in $H$ or not.

    Consider the case where $g \notin H$ first. Then $id<_\ooo g$ if and only if $\varphi(g)(\nu(x_0))>\nu(x_0)$. By semiconjugacy, the latter is equivalent to $\nu(\rho(g)(x_0))>\nu(x_0)$. By the same argument as in the last paragraph, this implies $\rho(g)(x_0)>x_0$. By R2, it follows that $\rho(g)(\rho(f)(x_0))>\rho(f)(x_0)$, that is, $\rho(f^{-1}gf)(x_0)>x_0$. If $f^{-1}gf\in H$, then $\rho(f^{-1}gf)(x_0)>x_0$ implies $id<_{\ooo'} f^{-1}gf$, and so $id<_\ooo f^{-1}gf$. If $f^{-1}gf\notin H$, then $\rho(f^{-1}gf)(x_0)>x_0$ implies $\nu(\rho(f^{-1}gf)(x_0)) > \nu(x_0)$. It follows that $\varphi(f^{-1}gf)(\nu(x_0))>\nu(x_0)$ and therefore $id<_\ooo f^{-1}gf$.

    Now consider the second case $g\in H$. Then $id<_\ooo g$ if and only if $id<_{\ooo'} g$. Since $g\not\in\ker(\rho)$, The latter is equivalent to $\rho(g)(x_0)>x_0$ by R1. By R2, it follows that $\rho(g)(\rho(f)(x_0))>\rho(f)(x_0)$ for any $f\in \pi_1(M)$. The second case now follows from an argument identical to the final steps of the first case. Therefore, the property O2 is satisfied.

    To show O3$'$, we first note that $H$ is $\ooo$-convex by construction. Next we observe that $\ker(\varphi)< H$ and $\alpha_i\in \ker(\varphi)$ for any $i\in J$ with $[\alpha_i]$ rational, which implies $\langle \langle \alpha_i \rangle \rangle \leq H$. Since $(\varphi, \nu(x_0))\in \mathcal{R}^*(\pi_1(M))$, $H\cap \pi_1(T_i)$ is at most of rank 1, and since $[\mathcal{L}(\varphi|_{\pi_1(T_i)}, \nu(x_0))] = [\alpha_i]$ we know that $H \cap \pi_1(T_i) \leq \langle \alpha_i \rangle$ when $[\alpha_i]$ is rational and $H \cap \pi_1(T_i) = \{id \}$ when $[\alpha_i]$ is irrational. It follows that O3$'$ is satisfied.
\end{proof}

This result also allows us to rework our definition of representation-detection as follows.

\begin{lemma}\label{lemma:equivR2}
    With the same assumptions as in Definition \ref{representation-detect}, there exists $(\rho, x_0) \in \mathcal{R}^*(\pi_1(M))$ satisfying R1, R3 and \begin{itemize}
        \item [R2$'$.] for all $x \in \mathbb{R}$, if $(\rho, x) \in \mathcal{R}^*(\pi_1(M))$ then
              $$([\mathcal{L}(\rho|_{\pi_1(T_1)}, x)], \ldots,[\mathcal{L}(\rho|_{\pi_1(T_n)}, x)]) = ([\beta_1], \ldots, [\beta_n]),$$
              with $[\beta_i] = [\alpha_i]$ for all $i \in K$,
    \end{itemize} if and only if $(J, K; [\alpha_1], \ldots, [\alpha_n])$ is representation-detected.
\end{lemma}

\begin{proof}
    Given $(\rho, x_0)$ satisfying R1, R2$'$, R3, it is clear that R1, R2, R3 are satisfied by noting that $(\rho,\rho(g)(x_0))\in \mathcal{R}^*(\pi_1(M))$.

    For the other direction, as a result of Theorem \ref{prop:equbetweenorddetectandrepdetect} and its proof, it suffices to show that if $\ooo$ order-detects $(J, K; [\alpha_*])$ then $(\rho_{\ooo}, 0)$ satisfies R1, R2$'$ and R3, where $\rho_{\ooo}$ is the dynamic realisation. That R1 and R3 are satisfied is contained in the proof of Theorem \ref{prop:equbetweenorddetectandrepdetect}, so we only need to show that any dynamic realisation $\rho_{\ooo}$ satisfies R2$'$.

    To this end, let $t :\pi_1(M) \rightarrow \mathbb{R}$ be the tight embedding used to construct $\rho_{\ooo}$ satisfying $t(id) = 0$ and choose $x \in \mathbb{R}$ be such that $(\rho, x) \in \mathcal{R}^*(\pi_1(M))$. We consider three cases.

    First, if $x = t(g)$ for some $g \in \pi_1(M)$, then since $t(id) = 0$ we have $x = \rho_{\ooo}(g)(0)$ and R2$'$ in this case simply reduces to R2.

    Now suppose that there exists $i \in K$ such that $[\mathcal{L}(\rho_{\ooo}|_{\pi_1(T_i)}, x)] \neq [\alpha_i]$. In this case there must be $g \in \pi_1(T_i)$ such that $g>_{\ooo}id$ and $\rho_{\ooo}(g)(x) < x$.

    Now, if $x \in \overline{t(\pi_1(M))}$ then we may choose a sequence $\{t(g_i)\} \subset t(\pi_1(M))$ that converges to $x$. However, if $\rho_{\ooo}(g)(x) < x$ then by the continuity of $\rho_{\ooo}(g)$ there exists $j$ such that $\rho_{\ooo}(g)(t(g_j)) < t(g_j)$. But then $gg_j<_{\ooo}g_j$. This contradicts O2, since $s_i(g_j \cdot \ooo) = [\alpha_i]$.

    Last, suppose that $x \in (a,b) \subset \mathbb{R}\setminus t(\pi_1(M))$. Then there is a gap in $h,k \in \pi_1(M)$ such that $(a,b) \subset (t(h), t(k))$. Then $$x = (1-s)t(h) + st(k)$$ for some $s \in (0,1)$ and
    $$ \rho_{\ooo}(g)(x) = (1-s)t(gh) + st(gk),$$
    so that $\rho_{\ooo}(g)(x)<x$ implies
    \[ (1-s)(\rho_{\ooo}(g)(t(h)) - t(h)) + s(\rho_{\ooo}(g)(t(k)) - t(k)) < 0
    \]
    which means that at least one of $t(gh) - t(h)<0$ or $t(gk) - t(k)<0$ holds. In other words, either $gh <_{\ooo}h$ or $gk<_{\ooo}k$, and no matter which is true, this contradicts O2 as in the previous paragraph.
\end{proof}

\section{Gluing theorems}
\label{sec:gluing}
The goal of this section is to develop theorems that allow one to analyse the boundary behaviour of left-orderings of the fundamental group of a $3$-manifold $M$ in terms of tuples of slopes which are order-detected on the boundary tori of the JSJ pieces of $M$. We already have at our disposal various gluing theorems from the literature that deal with special cases of left-orderable groups and detection. We review these results and offer improvements, with a later focus on special cases which apply to cable knots. We also remark that while this section is written in the language of orderings of fundamental groups, each ``orderability gluing theorem" has a representation-theoretic counterpart, owing to our work in the previous section.

For a left-orderable group $G$, a family of left-orderings $N\subset \LO(G)$ is said to be \emph{normal} if it is invariant under the $G$-action on $\LO(G)$, namely, $P\in N$ implies $gPg^{-1}\in N$ for all $g\in G$. Underpinning all of our gluing theorems is the following result:

\begin{theorem}[Bludov-Glass \cite{BG09}]\label{thm:bludovglass}
    Suppose that $A,G$ and $H$ are groups equipped with injective homomorphisms $\phi_1: A\to G, \phi_2: A\to H$. The free product with amalgamation $G*_{\phi_i}H$ is left-orderable if and only if $G$ and $H$ are left-orderable and there exist normal families $N_1\subset \LO(G)$ and $N_2\subset \LO(H)$ such that for every $P\in N_i$, there is $Q\in N_j$ satisfying $\phi_i^{-1}(P)=\phi_j^{-1}(Q)$ whenever $\{i,j\} = \{1,2\}$. Moreover, if $P \in N_1$ and $Q \in N_2$ satisfy $\phi_1^{-1}(P) = \phi_2^{-1}(Q)$ then there is an ordering $\ooo$ of $G*_{\phi_i}H$ whose restriction to $G$ (resp. $H$) has the positive cone $P$ (resp. $Q$).
\end{theorem}

With the same setup as in the above theorem, normal families that satisfy the condition
\[ (\forall P \in N_i )(\exists Q \in N_j)(\phi_i^{-1}(P)=\phi_j^{-1}(Q))
\]
for $\{i, j \}= \{1, 2\}$ will be called \emph{compatible} with the maps $\phi_i$.

Our focus will be on generalising the following result.

\begin{theorem}\cite[Theorem 7.10]{BC23}\label{thm:gluedetectedslopes}
    Let $M_1$ and $M_2$ be 3-manifolds such that $\partial M_i$ is a nonempty union of incompressible tori $T_{i,1}\sqcup T_{i,2}\sqcup\dots\sqcup T_{i,{r_i}}$. Suppose that for each manifold $M_i$, $(\emptyset;\{1,2,\dots,r_i\};[\alpha_{i,1}],[\alpha_{i,2}],\dots,[\alpha_{i, r_i}])$ is order-detected by $\ooo_i\in \LO(M_i)$ and that $f:T_{1,1}\to T_{2,1}$ is a homeomorphism that identifies $[\alpha_{1,1}]$ with $[\alpha_{2,1}]$. Reindex the boundary components $T_{1,2}, T_{1,3}, \ldots, T_{1,r_1}, T_{2,2}, T_{2,3}, \ldots , T_{2,r_2}$ of $M_1\cup_f M_2$ as $T_1, \ldots, T_{r_1+r_2-2}$, respectively. Then $\pi_1(M_1\cup_f M_2)$ is left-orderable and admits a left-ordering $\ooo'$ that order-detects $$(\emptyset;\{1, 2, \ldots, r_1+r_2-2\};[\alpha_{1,2}],\dots,[\alpha_{1, r_1}], [\alpha_{2,2}],\dots,[\alpha_{2, r_2}]),$$ whose restriction to $\pi_1(M_i)$ is $\ooo_i$.
\end{theorem}

We begin our efforts with some preparatory lemmas. Let $M$ be a compact connected orientable 3-manifold with incompressible torus boundary components $T_1,\dots,T_n$, $n\geq 1$, and fix a choice of peripheral subgroup $\pi_1(T_i)\subset \pi_1(M)$ for each $i$. Recall that the map $s_i : \LO(\pi_1(M)) \rightarrow \mathcal{S}(T_i)$ is the composition $\mathcal{L} \circ r_i$, where $r_i$ is the restriction map to $\pi_1(T_i)$ and $\mathcal{L}(\ooo) = [\mathcal{L}(\ooo)]$. 

\begin{definition}[\cite{BC17}]\label{def:readytoglue}
    A set $\mathcal{O}$ of left-orderings of $\pi_1(M)$ is called \emph{ready to glue} on $T_i$, or ready to glue along $s_i(\mathcal{O})$ on $T_i$, if $\mathcal{O}$ is normal and for all $[\alpha]\in s_i(\mathcal{O})$ we have $\mathcal{L}^{-1}([\alpha]) \subset r_i(\mathcal{O})$. More generally, suppose that $G$ is a left-orderable group with a subgroup $H \cong \mathbb{Z} \oplus \mathbb{Z}$. A set $\mathcal{O}$ of left-orderings of $G$ is called \emph{ready to glue} on $H$ if $\mathcal{O}$ is normal and for all $[\alpha]\in s(\mathcal{O})$ we have $\mathcal{L}^{-1}([\alpha]) \subset r(\mathcal{O})$, where $s$ is the composition
    \[ \mathrm{LO}(G) \stackrel{r}{\rightarrow} \mathrm{LO}(H) \stackrel{\mathcal{L}}{\rightarrow} S^1.
    \]
\end{definition}

The following is a slight generalisation of \cite[Proposition 11.5]{BC17}, the proof being nearly identical (and so it is omitted). In the following, we use the map $s_i : \mathrm{LO}(G_i) \rightarrow \mathcal{S}(\phi_i(A))$ to denote the composition $\mathcal{L} \circ r_i$, where $r_i$ is the restriction map.

\begin{proposition}\label{prop:ready_to_glue}
    Suppose that $A,G_1$ and $G_2$ are groups equipped with injective homomorphisms $\phi_1: A\to G_1, \phi_2: A\to G_2$ and that $A \cong \mathbb{Z} \oplus \mathbb{Z}$. Suppose that $N_1 \subset \mathrm{LO}(G_1)$ and $N_2 \subset \mathrm{LO}(G_2)$ are normal families containing orderings $\ooo_1$ and $\ooo_2$ respectively, and that each $N_i$ is ready to glue along $\phi_i(A)$.
    If $s_1(N_1) = s_2(N_2)$ and $\phi_1^{-1}(P(\ooo_1)) = \phi_2^{-1}(P(\ooo_2))$, then $G_1*_{\phi_i}G_2$ admits a left-ordering $\ooo$ whose restriction to $G_i$ is $\ooo_i$.
\end{proposition}

The next lemma and proposition are essential for creating families of orderings that are ready to glue.

\begin{lemma}\label{convex swap}
    Suppose that $M$ is a $3$-manifold with $\partial M$ a nonempty union of disjoint incompressible tori $T_1, \dots, T_r$, and further suppose that $\ooo$ is a left-ordering of $\pi_1(M)$ that order-detects $(J,K;[\alpha_*])$ and $C$ is an $\ooo$-convex subgroup of $\pi_1(M)$. Then $Q= (P(\ooo) \cap C)^{-1} \cup (P(\ooo)\setminus C)$ is the positive cone of a left-ordering $\ooo'$ that also detects $(J,K;[\alpha_*])$.
\end{lemma}

\begin{proof}
    That $Q$ is a positive cone is a standard check. We verify that the ordering $\ooo'$ has the required properties. Fix $i \in \{1, \ldots, r\}$ and consider $C \cap \pi_1(T_i)$. Since $C$ is a convex subgroup, there are three possibilities.

    First, if $C\cap \pi_1(T_i) = \{ id \}$ then $\pi_1(T_i) \cap Q = \pi_1(T_i) \cap P(\ooo)$ so that $s_i(\ooo') = [\alpha_i]$. Second, if $C \cap \pi_1(T_i) \cong \mathbb{Z}$ then $[\alpha_i]$ must be rational with $C \cap \pi_1(T_i) = \langle \alpha_i \rangle$. It follows that $(\pi_1(T_i) \setminus \langle \alpha_i \rangle) \cap Q = P(\ooo) \cap (\pi_1(T_i) \setminus \langle \alpha_i \rangle)$ so that again $s_i(\ooo') = [\alpha_i]$. In the final case, if $\pi_1(T_i) \subset C$ then $Q \cap \pi_1(T_i) = P(\ooo)^{-1} \cap \pi_1(T_i)$ so that again $s_i(\ooo') = [\alpha_i]$. This shows O1.

    Next, assume $i \in K$ and let $g \in \pi_1(M)$. Consider $g \cdot \ooo'$ whose positive cone is $g Q g^{-1}$, there are three cases as in the previous paragraph. First, if $\pi_1(T_i) \cap gCg^{-1} = \{ id \}$, then $gQg^{-1} \cap \pi_1(T_i) = gP(\ooo)g^{-1} \cap \pi_1(T_i)$ and so $s_i(g \cdot \ooo') = s_i (g \cdot \ooo) = s_i( \ooo) = s_i(\ooo')$. Second, if $\pi_1(T_i) \cap gCg^{-1} \cong \mathbb{Z}$, then since $gCg^{-1}$ is also $g\cdot \ooo$-convex and $s_i (g \cdot \ooo) =[\alpha_i]$ since $i \in K$, we know that $\pi_1(T_i) \cap gCg^{-1} = \langle \alpha_i \rangle$. Then observe that $gQg^{-1} = (P(g \cdot \ooo) \setminus gCg^{-1}) \cup (P(g \cdot \ooo) \cap gCg^{-1})^{-1}$, so that
    \[ gQg^{-1} \cap (\pi_1(T_i) \setminus \langle \alpha_i \rangle) = (P(g \cdot \ooo) \setminus gCg^{-1}) \cap (\pi_1(T_i) \setminus \langle \alpha_i \rangle),\]
    from which it follows that $s_i(g \cdot \ooo') = s_i (g \cdot \ooo) = s_i( \ooo) = s_i(\ooo')$. Finally, if $\pi_1(T_i) \subset gCg^{-1}$, then repeating an argument that is nearly identical to the first case proves O2.

    Lastly, if $H$ is the normal subgroup of $\pi_1(M)$ that is $\ooo$-convex and satisfies O3, then either $H \subset C$ or $C \subset H$.
    In either case, it is easy to verify that $H$ is also $\ooo'$-convex and satisfies the required properties.
\end{proof}

\begin{proposition}\label{normal_construction}
    Suppose that $M$ is a $3$-manifold with $\partial M$ a nonempty union of disjoint incompressible tori $T_1, \dots, T_r$, and further suppose that $\ooo$ is a left-ordering of $\pi_1(M)$ that detects $(J,K;[\alpha_*])$, and that $[\alpha_1]$ is rational. Then there exists a left-ordering $\ooo'$ of $\pi_1(M)$ such that:
    \begin{enumerate}
        \item $P(\ooo) \cap (\pi_1(T_1) \setminus \langle \alpha_1 \rangle)= P(\ooo') \cap (\pi_1(T_1)\setminus \langle \alpha_1 \rangle)$,
        \item $s_j(g \cdot \ooo') = s_j(\ooo') = s_j(\ooo)$ for all $j \in K$ and $g \in \pi_1(M)$,
        \item if $\alpha_1 >_{\ooo} id$, then $\alpha_1 <_{\ooo'} id$; and if $\alpha_1 <_{\ooo} id$, then $\alpha_1 >_{\ooo'} id$.
    \end{enumerate}
\end{proposition}

\begin{proof}
    Let $\rho_{\ooo} : \pi_1(M) \rightarrow \Homeo_+(\mathbb{R})$ be the dynamic realisation of $\ooo$ defined from the tight embedding $t:G \rightarrow \mathbb{R}$ satisfying $t(id) =0$. Set
    \[ x_0 = \sup\{t(\alpha_1^k) : k \in \mathbb{Z} \}
    \]
    and set $C = \Stab_{\rho_{\ooo}}(x_0)$. In the language of \cite{BC23}, $x_0$ is an \emph{ideal point} of the dynamical realisation.

    By \cite[Lemma 3.7]{BC23}, there exists $\ooo' \in \overline{\{h \cdot \ooo : h \in \pi_1(M)\}}$, the closure of the orbit of $\ooo$ in $\mathrm{LO}(\pi_1(M))$, such that $C$ is $\ooo'$-convex.

    By \cite[Lemma 3.6]{BC23}, the ordering $\ooo'$ satisfies (1). To show (2), note that $g \cdot \ooo' \in \overline{\{h \cdot \ooo : h \in \pi_1(M)\}}$ for all $g \in \pi_1(M)$. Since the slope map $s_j :\mathrm{LO}(\pi_1(M)) \rightarrow \mathcal{S}(T_j)$ is continuous and takes the constant value $s_j(\ooo)$ on $\{h \cdot \ooo : h \in \pi_1(M)\}$ for all $j \in K$, (2) holds.

    Lastly, by Lemma \ref{convex swap} we can choose the sign of $\alpha_1$ so that (3) is also satisfied.
\end{proof}

\begin{remark}\label{more_general}
    Although Proposition \ref{normal_construction} as written deals only with $3$-manifold groups, an analogous proposition holds if we replace $\pi_1(M)$ with an arbitrary group $G$ and $\pi_1(T_i)$ with subgroups $H_i \subset G$ each isomorphic to $\mathbb{Z} \oplus \mathbb{Z}$.
\end{remark}

\begin{theorem}\label{thm:main_gluing_thm}
    Let $M_1$ and $M_2$ be 3-manifolds such that $\partial M_i$ is a union of incompressible tori $T_{i,1}\sqcup T_{i,2}\sqcup\dots\sqcup T_{i,{r_i}}$. Suppose that for $i=1,2$, $(J_i,K_i;[\alpha_{i,1}],[\alpha_{i,2}],\dots,[\alpha_{i, r_i}])$ is order-detected by $\ooo_i\in \LO(\pi_1(M_i))$ and that $f:T_{1,1}\to T_{2,1}$ is a homeomorphism that identifies $[\alpha_{1,1}]$ with $[\alpha_{2,1}]$. Re-index the boundary components $T_{1,2}, T_{1,3}, \ldots, T_{1,r_1}, T_{2,2}, T_{2,3}, \ldots , T_{2,r_2}$
    of the manifold $M_1\cup_f M_2$ as $T_1, \ldots, T_{r_1+r_2-2}$, respectively. Set $J_1' = \{ n-1 : n \in J_1, n\geq2 \}$, $K_1' = \{ n -1 : n\in K_1, n\geq 2\}$, $J_2' = \{n+r_1-2 : n \in J_2, n \geq 2 \}$ and $K_2' = \{n+r_1-2 : n \in K_2, n \geq 2 \}$.
    Suppose that $1 \in K_i$ for $i = 1, 2$, and let $C_i \leq \pi_1(M_i)$ be the $\ooo_i$-convex normal subgroup guaranteed by O3. Suppose that if $[\alpha_{i,1}]$ is rational and there exists $i \in \{1, 2\}$ such that $C_i \cap \pi_1(T_{i,1}) = \langle \alpha_{i,1} \rangle$, then either $(J_i \cup\{1\},K_i;[\alpha_{i,1}],[\alpha_{i,2}],\dots,[\alpha_{i, r_i}])$ is order-detected for $i = 1, 2$ or $C_i \cap \pi_1(T_{i,j}) = \{ id \}$ for all $j \neq 1$. 
    Then $\pi_1(M_1\cup_f M_2)$ is left-orderable and admits a left-ordering detecting $$(J_1'\cup J_2', K_1'\cup K_2';[\alpha_{1,2}],\dots,[\alpha_{1, r_1}], [\alpha_{2,2}],\dots,[\alpha_{2, r_2}]).$$
\end{theorem}




\begin{proof}
    We can largely mirror the proof of \cite[Theorem 7.10]{BC23}, making changes as needed. Observe as in \cite{BC23} that if either manifold is a product $T^2 \times [0,1]$ then the result holds trivially. Therefore, for each $i$, $\pi_1(T_{i,1})$ is a proper subgroup of $\pi_1(M_i)$, and we consider two cases. 

    For the first case, suppose each $[\alpha_{i,1}]$ is irrational. Let $C_i \subset \pi_1(M_i)$ be the $\ooo_i$-convex subgroup guaranteed by O3, and $q_i : \pi_1(M_i) \rightarrow \pi_1(M_i)/C_i$ the quotient map. The maps $q_1, q_2$ induce a map
    \[ q: \pi_1(M_1 \cup_f M_2) \rightarrow \pi_1(M_1)/C_1 *_{f_i} \pi_1(M_2)/C_2
    \]
    where $f_i:\mathbb{Z} \oplus \mathbb{Z} \rightarrow q_i(\pi_1(T_{i,1}))$ are isomorphisms satisfying $f_* \circ f_1=f_2$ where $f_* : q_1(\pi_1(T_{1,1})) \rightarrow q_2(\pi_1(T_{2,1}))$ is the isomorphism induced by $f$.

    Let $\hat{\ooo}_i$ denote the natural left-ordering of the quotient $\pi_1(M_i)/C_i$ induced by $\ooo_i$, and consider the normal families
    \[N_i = \{g \cdot \hat{\ooo}_i : g \in \pi_1(M_i)/C_i \} \cup \{ g \cdot \hat{\ooo}^{op}_i : g \in \pi_1(M_i)/C_i\} \subset \mathrm{LO}(\pi_1(M_i)/C_i)
    \]
    for $i=1, 2$. Since $1 \in K_i$ for $ i= 1, 2$, and therefore $s_{i,1}(g \cdot \ooo_i) = [\alpha_{i,1}]$ for all $g \in \pi_1(M_i)$, we conclude that the image of $N_i$ under the composition
    \[ \pi_1(M_i)/C_i \stackrel{r_{i,1}}{\longrightarrow} q_i(\pi_1(T_{i,1})) \stackrel{\mathcal{L}}{\longrightarrow} \mathcal{S}(\pi_1(T_{i,1}))
    \]
    is the singleton $\{[\alpha_{i,1}]\}$, here $r_{i,1}$ is the restriction map and $s_{i, 1} : \mathrm{LO}(\pi_1(M_i)) \rightarrow \mathcal{S}(T_{i,1})$ the slope map. Moreover, $\mathcal{L}^{-1}([\alpha_{i,1}]) \subset r_{i,1}(N_i)$. Therefore, the families $N_1, N_2$ are compatible with the maps $f_i$, in fact, they are ready to glue. So, by Theorem \ref{thm:bludovglass} there exists an ordering $\hat{\ooo}$ of $\pi_1(M_1)/C_1 *_{f_i} \pi_1(M_2)/C_2$ whose restriction to $\pi_1(M_i)/C_i$ is $\hat{\ooo}_i$.

    Construct a left-ordering $\ooo$ of $\pi_1(M_1 \cup_f M_2)$ lexicographically from the short exact sequence
    \[ \{ id \} \rightarrow \ker(q) \rightarrow \pi_1(M_1 \cup_f M_2) \stackrel{q}{\rightarrow} \pi_1(M_1)/C_1 *_{f_i} \pi_1(M_2)/C_2 \rightarrow \{ id \}
    \]
    using $\hat{\ooo}$ on the quotient and an arbitrary left-ordering of the kernel (note that, indeed, the kernel is left-orderable since it is a subgroup of $\pi_1(M_1 \cup_f M_2)$ which is left-orderable by Theorem \ref{thm:gluedetectedslopes}.) We verify that this ordering satisfies the required properties.

    First by construction, $s(\ooo) = ([\alpha_{1,2}],\dots,[\alpha_{1, r_1}], [\alpha_{2,2}],\dots,[\alpha_{2, r_2}])$ so that O1 holds. Moreover, from the proof of \cite[Theorem A]{BG09}, the left-ordering $\hat{\ooo}$ of $\pi_1(M_1)/C_1 *_{f_i} \pi_1(M_2)/C_2$ satisfies $\hat{\ooo}|_{h\pi_1(M_i)/C_i h^{-1}} \in N_i$ for all $h \in \pi_1(M_1)/C_1 *_{f_i} \pi_1(M_2)/C_2$. Therefore, the restriction of $\hat{\ooo}$ to $h\pi_1(M_i)/C_i h^{-1}$ is of the form $g \cdot \hat{\ooo}_i$ or $g \cdot \hat{\ooo}^{op}_i$ for some $g \in \pi_1(M_i)/C_i$ for all $h \in \pi_1(M_1)/C_1 *_{f_i} \pi_1(M_2)/C_2$. It follows that $s(h \cdot \ooo) = ([\beta_1], \dots, [\beta_{r_1+r_2-2}])$ where $[\beta_j] = s_j(\ooo)$ for all $j \in K_1' \cup K_2'$. Thus, O2 holds.
    Finally, the kernel $C = \ker(q)$ is the normal, $\ooo$-convex subgroup required by O3, so we conclude that $$(J_1'\cup J_2', K_1'\cup K_2';[\alpha_{1,2}],\dots,[\alpha_{1, r_1}], [\alpha_{2,2}],\dots,[\alpha_{2, r_2}])$$
    is detected by $\ooo$. This proves the theorem in the first case, where $[\alpha_{i,1}]$ is irrational.

    Next, we modify the proof to handle the case where $[\alpha_{i,1}]$ is rational. 

    First, suppose that there exists $i$ such that $C_i \cap \pi_1(T_{i,1}) = \langle \alpha_{i,1} \rangle$ and that $(J_i \cup\{1\},K_i;[\alpha_{i,1}],[\alpha_{i,2}],\dots,[\alpha_{i, r_i}])$ is order-detected for $i = 1, 2$. Then perhaps after replacing one of the orderings $\ooo_i$, we may assume that $(J_i \cup\{1\},K_i;[\alpha_{i,1}],[\alpha_{i,2}],\dots,[\alpha_{i, r_i}])$ is $\ooo_i$-detected for $i = 1, 2$.

    In this case, $q_1, q_2$ induce a map
    \[ q: \pi_1(M_1 \cup_f M_2) \rightarrow \pi_1(M_1)/C_1 *_{f_i} \pi_1(M_2)/C_2
    \]
    where $f_i:\mathbb{Z} \rightarrow q_i(\pi_1(T_{i,1}))$ are isomorphisms satisfying $f_* \circ f_1=f_2$ where $f_* : \mathbb{Z} \cong q_1(\pi_1(T_{1,1})) \rightarrow q_2(\pi_1(T_{2,1})) \cong \mathbb{Z}$ is the isomorphism induced by $f$. Since the image of $q$ is an amalgam of left-orderable groups along a cyclic subgroup, it is left-orderable, with any two compatible orderings of the factors extending to an ordering of the amalgam.

    Denoting the natural left-orderings of the quotients $\pi_1(M_i)/C_i$ again by $\hat{\ooo}_i$, we fix an ordering $\hat{\ooo}$ of $\pi_1(M_1)/C_1 *_{f_i} \pi_1(M_2)/C_2$ that restricts to $\hat{\ooo}_i$ on $\pi_1(M_i)/C_i$. As before, construct a left-ordering $\ooo$ of $\pi_1(M_1 \cup_f M_2)$ lexicographically from the short exact sequence
    \[ \{ id \} \rightarrow \ker(q) \rightarrow \pi_1(M_1 \cup_f M_2) \stackrel{q}{\rightarrow} \pi_1(M_1)/C_1 *_{f_i} \pi_1(M_2)/C_2 \rightarrow \{ id \}
    \]
    using $\hat{\ooo}$ on the quotient and an arbitrary left-ordering of the kernel. Arguing exactly as in the previous case, we see that $\ooo$ satisfies O1, O2, and O3, so that $$(J_1'\cup J_2', K_1'\cup K_2';[\alpha_{1,2}],\dots,[\alpha_{1, r_1}], [\alpha_{2,2}],\dots,[\alpha_{2, r_2}])$$
    is order-detected in this case.

    Last, suppose that $[\alpha_{i,1}]$ is rational 
    and that for each $i$, either $C_i \cap \pi_1(T_{i,1}) = \{id\}$ or $C_i \cap \pi_1(T_{i,1}) = \langle \alpha_{i,1} \rangle$ and $C_i \cap \pi_1(T_{i,j}) = \{ id \}$ for all $j \neq 1$. Then, as in the previous cases, we begin with convex subgroups $C_1, C_2$ arising from O3, but we replace $C_i$ with $C_i = \{ id\}$ if $C_i \cap \pi_1(T_{i,1}) = \langle \alpha_{i,1} \rangle$ and $C_i \cap \pi_1(T_{i,j}) = \{ id \}$ for all $j \neq 1$. Then again we have quotient maps $q_1, q_2$, and
    \[ q: \pi_1(M_1 \cup_f M_2) \rightarrow \pi_1(M_1)/C_1 *_{f_i} \pi_1(M_2)/C_2
    \]
    where $f_i:\mathbb{Z} \oplus \mathbb{Z} \rightarrow q_i(\pi_1(T_{i,1}))$ are isomorphisms satisfying $f_2 \circ f_*=f_1$ where $f_* : q_1(\pi_1(T_{1,1})) \rightarrow q_2(\pi_1(T_{2,1}))$ is the isomorphism induced by $f$. Note that our maps $f_i$ are isomorphisms precisely because we chose $C_i$ such that $\alpha_{i,1} \notin C_i$ for $i =1, 2$. Let $\hat{\ooo}_i$ denote the natural ordering of the quotient $\pi_1(M_i)/C_i$ induced by $\ooo_i$. By Remark \ref{more_general} and Proposition \ref{normal_construction} there is an ordering $\hat{\ooo}_i'$ of each quotient $\pi_1(M_i)/C_i$ such that:
    \begin{enumerate}
        \item $P(\hat{\ooo}_i) \cap (q_i(\pi_1(T_{i,1}))\setminus \langle q(\alpha_{i,1}) \rangle) = P(\hat{\ooo}_i') \cap (q_i(\pi_1(T_{i,1}))\setminus \langle q(\alpha_{i,1}) \rangle)$,
        \item when restricted to $q_i(\pi_1(T_{i,j}))$ where $j \in K$, the orderings $g \cdot \hat{\ooo}_i'$ and $\hat{\ooo}_i$ determine the same slope, and
        \item if $\alpha_{i,1} >_{\hat{\ooo}_i} id$ then $\alpha_{i,1} <_{\hat{\ooo}_i'} id$ and if $\alpha_{i,1} <_{\hat{\ooo}_i} id$ then $\alpha >_{\hat{\ooo}_i'} id$.
    \end{enumerate}

    Now define
    \begin{multline*}
        N_i = \{g \cdot \hat{\ooo}_i : g \in \pi_1(M_i)/C_i \} \cup
        \{ g \cdot \hat{\ooo}^{op}_i : g \in \pi_1(M_i)/C_i\} \\
        \cup \{g \cdot \hat{\ooo}_i' : g \in \pi_1(M_i)/C_i \} \cup
        \{ g \cdot (\hat{\ooo}'_i)^{op} : g \in \pi_1(M_i)/C_i\}
    \end{multline*}

    By construction, the families $N_i$ are compatible with the maps $f_i$ and are ready to glue along $f_i(\pi_1(T_{i,1}))$. So by Proposition \ref{prop:ready_to_glue} there is a left-ordering $\hat{\ooo}$ of $\pi_1(M_1)/C_1 *_{f_i} \pi_1(M_2)/C_2$ whose restriction to $\pi_1(M_i)/C_i$ is $\hat{\ooo}_i$, and we can construct $\ooo$ of $\pi_1(M_1 \cup_f M_2)$ lexicographically as in the first case. By arguments identical to the first case, the ordering $\ooo$ detects
    $$(J_1'\cup J_2', K_1'\cup K_2';[\alpha_{1,2}],\dots,[\alpha_{1, r_1}], [\alpha_{2,2}],\dots,[\alpha_{2, r_2}]),$$
    as required.
\end{proof}

We can also glue, in some circumstances, along slopes that are weakly detected. However, in doing so, we cannot control the boundary behaviour of the resulting left-ordering to the same degree as in Theorem \ref{thm:main_gluing_thm}. We begin with lemmas that allow us to construct families of orderings which are ready to glue. The next lemma can be viewed as a special case of Proposition \ref{normal_construction} which is sufficient for our purposes here.

\begin{lemma}\cite[Lemma 11.10]{BC17}\label{lemma:readytoglueprep}
    Suppose that $\ooo$ is a left-ordering of $G$ and $g,f\in G$. If $\{g^k : k\in \ZZ\}$ is bounded above by $f$ in the left-ordering $\ooo$, then there is a left-ordering $\ooo'$ of $G$ and a proper $\ooo'$-convex subgroup $C$ of $G$ that contains $g$ but not $f$. 
\end{lemma}

For the statement of the next lemma, recall that $r_i : \LO(\pi_1(M)) \rightarrow \LO(\pi_1(T_i))$ denotes the restriction map, and $\mathcal{L}_i$ the map that associates a slope with each ordering of $\pi_1(T_i)$; the map $s_i$ denotes their composition.

\begin{lemma}\label{lemma:familyofaslope}
    Suppose that $\mathcal{O}$ is a family of left-orderings of $\pi_1(M)$. Then for each fixed $i=1,\dots,n$, there is a family $R_i(\mathcal{O})$ of left-orderings of $\pi_1(M)$ that contains $\mathcal{O}$ and satisfies $\mathcal{L}^{-1}([\alpha]) \subset r_i(R_i(\mathcal{O}))$ for all $[\alpha] \in s_i(R_i(\mathcal{O}))$.
\end{lemma}

\begin{proof}
    This is a restatement of \cite[Lemma 7.8]{BC23} in the case where $M$ has multiple boundary components; where the original lemma is stated for knot manifolds. We proceed as follows:

    For every slope $[\gamma]\in s_i(\mathcal{O})$, define a set $R([\gamma])$ as follows. If $[\gamma]$ is irrational, choose $\ooo\in \mathcal{O}$ with $s_i(\ooo)=[\gamma]$ and set $R([\gamma])=\{\ooo,\ooo^{op}\}$ so that $r_i(R([\gamma]))=\mathcal{L}^{-1}([\gamma])$. If $[\gamma]$ is rational, choose $\ooo\in \mathcal{O}$ with $s_i(\ooo)=[\gamma]$. By Lemma \ref{lemma:readytoglueprep}, there exists a left-ordering $\ooo'$ and a proper $\ooo'$-convex subgroup $C$ of $\pi_1(M)$ such that $\pi_1(T_i)\cap C=\vbracket{\gamma}$. So in this case, we can create a set $R([\gamma])$ of four left-orderings of $\pi_1(M)$ satisfying $r_i(R([\gamma]))=\mathcal{L}^{-1}([\gamma])$ as in Lemma \ref{convex swap}. Finally, set
    $$R_i(\mathcal{O})=\mathcal{O}\cup \left(\bigcup_{[\gamma]\in s_i(\mathcal{O})} R([\gamma])\right)$$
    and by construction, the set $R_i(\mathcal{O})$ satisfies the required condition.
\end{proof}

\begin{proposition}\label{prop:readytoglue}
    Fix a left-ordering $\ooo$ of $\pi_1(M)$ and $i\in \{1,\dots,n\}$. Then there exists a family $\mathcal{O}$ of left-orderings of $\pi_1(M)$ containing $\ooo$ that is ready to glue along $s_i(\mathcal{O})$ on $T_i$.
\end{proposition}

\begin{proof}
    Given a set $S$ of left-orderings of $\pi_1(M)$, write $N(S)$ for the set $\{g\cdot \ooo : \ooo\in S \mbox{ and }g\in \pi_1(M)\}$. Set $X_0=\{\ooo\}$ and for $j\geq 0$ set $X_{j+1}=N(R_i(X_j))$, where $R_i(S)$ is defined as in Lemma \ref{lemma:familyofaslope}. Set $\mathcal{O}=\cup_{j=0}^\infty X_j$. By construction, $\mathcal{O}$ is both normal and ready to glue along $s_i(\mathcal{O})$ on $T_i$.
\end{proof}

\begin{theorem}
    \label{thm:special_gluing_case}
    Let $M_1$ and $M_2$ be 3-manifolds such that $\partial M_i$ is a union of incompressible tori $T_{i,1}\sqcup T_{i,2}\sqcup\dots\sqcup T_{i,{r_i}}$. Suppose that for each manifold $M_i$, $(J_i,K_i;[\alpha_{i,1}],[\alpha_{i,2}],\dots,[\alpha_{i, r_i}])$ is order-detected by $\ooo_i\in \LO(M_i)$ and that $f:T_{1,1}\to T_{2,1}$ is a homeomorphism that identifies $[\alpha_{1,1}]$ with $[\alpha_{2,1}]$. Re-index the boundary components $T_{1,2}, T_{1,3}, \ldots, T_{1,r_1}, T_{2,2}, T_{2,3}, \ldots , T_{2,r_2}$
    of the manifold $M_1\cup_f M_2$ as $T_1, \ldots, T_{r_1+r_2-2}$ respectively. Suppose that $K_2=\{1, \dots, r_2\}$, and for every $[\alpha] \in \mathcal{S}(T_{2,1})$ there exists a left-ordering $\ooo$ of $\pi_1(M_2)$ detecting some $(J_2', K_2'; [\alpha_*])$ with $1 \in K_2'$ and $s_{2,1}(\ooo) = [\alpha]$.

    Then $\pi_1(M_1\cup_f M_2)$ is left-orderable and admits a left-ordering $\ooo$ detecting $$(\emptyset, \emptyset;[\alpha_{1,2}],\dots,[\alpha_{1, r_1}], [\alpha_{2,2}],\dots,[\alpha_{2, r_2}]);$$
    moreover, the restriction of $\ooo$ to $\pi_1(M_1)$ is $\ooo_1$.
\end{theorem}
\begin{proof}
    By Proposition \ref{prop:readytoglue}, there exists $N_1 \subset \mathrm{LO}(\pi_1(M_1))$ containing $\ooo_1$ that is ready to glue along $T_{1,1}$.

    Set
    \[ S = \{ [\alpha] \in \mathcal{S}(T_{2,1}) : \exists [\gamma] \in s_{1,1}(N_1) \mbox{ such that } f_*([\gamma]) = [\alpha] \}
    \]
    where $s_{1,1} : \mathrm{LO}(\pi_1(M_1)) \rightarrow \mathcal{S}(T_{1,1})$ is the slope map. For each $[\alpha] \in S$ with $[\alpha] \neq f_*(s_{1,1}(\ooo_1))$, choose an ordering $\ooo_{[\alpha]}$ of $\pi_1(M_2)$ that detects some $(J, K; [\alpha_*])$ with $s_{2,1}(\ooo_{[\alpha]}) = [\alpha]$. When $[\alpha] = f_*(s_{1,1}(\ooo_1))$, choose $\ooo_{[\alpha]} = \ooo_2$. Using Proposition \ref{normal_construction}, for each $\ooo_{[\alpha]}$ choose $\hat{\ooo}_{[\alpha]}$ that satisfies the conditions of Proposition \ref{normal_construction} and set
    \begin{multline*}
        N_{[\alpha]} = \{g \cdot \ooo_{[\alpha]} : g \in \pi_1(M_2) \} \cup
        \{ g \cdot \ooo_{[\alpha]}^{op} : g \in \pi_1(M_2)\} \\
        \cup \{g \cdot \hat{\ooo}_{[\alpha]} : g \in \pi_1(M_2) \} \cup
        \{ g \cdot (\hat{\ooo}_{[\alpha]})^{op} : g \in \pi_1(M_2)\},
    \end{multline*}
    by construction $N_{[\alpha]}$ is ready to glue along $T_{2,1}$. Now set
    \[ N_2 = \bigcup_{[\alpha] \in S} N_{[\alpha]},
    \]
    by construction $N_2$ is ready to glue along $T_{2,1}$. Choosing maps $f_i : \mathbb{Z} \oplus \mathbb{Z} \rightarrow \pi_1(M_i)$ for $i=1,2$ such that $f_i(\mathbb{Z} \oplus \mathbb{Z}) = \pi_1(T_{i,1})$ and $f_* \circ f_1 = f_2$, the normal families $N_1, N_2$ are compatible with the maps $f_i$ by construction, and so $\pi_1(M_1 \cup_f M_2) = \pi_1(M_1) *_{f_i} \pi_1(M_2)$ is left-orderable by Proposition \ref{prop:ready_to_glue}. Moreover, there exists $\ooo_2' \in \{ \ooo_{s_{1,1}(\ooo_1)}, \ooo_{s_{1,1}(\ooo_1)}^{op}, \hat{\ooo}_{s_{1,1}(\ooo_1)}, (\hat{\ooo}_{s_{1,1}(\ooo_1)})^{op}\}$ satisfying $f_1^{-1}(P(\ooo_1)) = f_2^{-1}(P(\ooo_2'))$, so that we may choose a left-ordering $\ooo$ of $\pi_1(M_1 \cup_f M_2)$ that restricts to $\ooo_1$ on $\pi_1(M_1)$ and $\ooo_2'$ on $\pi_1(M_2)$. Since $K_2=\{1, \dots, r_2\}$ and $\ooo_2'$ arises from an application of Proposition \ref{normal_construction}, the ordering $\ooo$ detects
    $$(\emptyset, \emptyset;[\alpha_{1,2}],\dots,[\alpha_{1, r_1}], [\alpha_{2,2}],\dots,[\alpha_{2, r_2}]).$$
\end{proof}

\section{JN-realisability and representation-detection}\label{sec:seifertmanifold}

Here we introduce Seifert fibered manifolds and JN-realisability, following \cite{BC17}. This serves as a way of computing which tuples of slopes on the boundary of a Seifert fibered manifold are representation-detected, see Proposition \ref{prop:jnimplyrepdetected}.

Throughout this section, we use $M$ to denote an orientable Seifert manifold with base orbifold $P(a_1,a_2,\dots,a_n)$, where $P$ is a punctured 2-sphere, and whose boundary is a nonempty collection of incompressible tori $T_1,\dots, T_r$. We use $h\in \pi_1(M)$ to denote the class of a regular Seifert fiber of $M$. Suppose that the Seifert invariants of the exceptional fibers of $M$ are given by $(a_1,b_1),\dots,(a_n,b_n)$ with $0<b_i<a_i$ for all $i=1,2,\dots,n$. Set \[\gamma_i=\frac{b_i}{a_i}\in (0,1).\]
The fundamental group of $M$ has a presentation \[\pi_1(M)=\vbracket{y_1,\dots,y_n,x_1,\dots,x_r,h \mid h \mbox{ central, } y_i^{a_i}=h^{b_i}, y_1y_2\dots y_nx_1x_2\dots x_r=1}.\]
In this presentation, $x_j$ is a dual class to $h$ on each $T_j$ for $1\leq j\leq r$. We say $[\alpha_*]\in \mathcal{S}(M)$ is rational if each $[\alpha_i]$ is rational, and we call $[\alpha_*]$ \emph{horizontal} if no $[\alpha_i]$ coincides with the slope of the fiber class $[h]$.

For $\gamma\in \RR$, denote by $\sh(\gamma)$ translation by $\gamma$, that is, $\sh(\gamma)(x) =x+\gamma$ for all $x \in \RR$. The universal cover $\hstilde$ of $\Homeo_+(S^1)$ is canonically isomorphic to the subgroup of $\Homeo_+(\RR)$ which consists of homeomorphisms that commute with $\sh(1)$:\[ \hstilde=\{f\in \Homeo_+(\RR): f(x+1)=f(x)+1, \forall x\in\RR\}. \]

We define the translation number $\tau:\hstilde\to \RR$
by
\[ \tau(h) = \lim_{n \to \infty}\frac{h^n(0)}{n},
\]
which satisfies $\tau(\sh(\gamma))=\gamma$.
The translation number map $\tau$ is invariant under conjugation and becomes a homomorphism when restricted to any abelian subgroup of $\hstilde$; moreover, $\tau(f)=0$ if and only if $f$ has a fixed point.


\begin{definition}\label{def:jnreal}\cite{BC17} For a subset $J\subset \{1,\dots,r\}$ and an $r$-tuple $(\tau_1,\dots,\tau_r)$ of real numbers, we say $(J;0;\gamma_1,\dots,\gamma_n;\tau_1,\dots,\tau_r)$ is \emph{JN-realisable} if there is a homomorphism $\rho:\pi_1(M)\to \hstilde$ such that \begin{enumerate}
        \item $\rho(h)=\sh(1)$;
        \item $\tau_j=\tau(\rho(x_j))$ for $1\leq j\leq r$;
        \item $\rho(x_j)$ is conjugate to $\sh(\tau_j)$ for each $j\in J$.
    \end{enumerate}

    More generally, $(J;b;\gamma_1,\dots,\gamma_n;\tau_1,\dots,\tau_r)$ is \emph{JN-realisable} if there are $f_i, g_j\in \hstilde$ such that \begin{enumerate}
        \item $f_i$ is conjugate to $\sh(\gamma_i)$ for all $i=1,2,\dots,n$;
        \item $\tau_j=\tau(g_j)$ for $1\leq j\leq r$;
        \item $g_j$ is conjugate to $\sh(\tau_j)$ for each $j\in J$;
        \item $f_1\circ \dots\circ f_n\circ g_1\circ \dots\circ g_r=\sh(b)$.
    \end{enumerate}
\end{definition}

For each $\tau_j=\tau(\rho(x_j))$, let $[\alpha_j]= [x_j-\tau_j h] \in H_1(T_j; \RR)$ and write $[\tau_*]=(\tau_1,\dots,\tau_r)$. Note that $[\alpha_*]$ is horizontal. 

\begin{proposition}\label{prop:jnimplyrepdetected}
    If $(J;0;\gamma_1,\dots,\gamma_n;\tau_1,\dots,\tau_r)$ is JN-realisable, then $(J, \{1, \ldots, r\}; [\alpha_*])$ is representation-detected.
\end{proposition}
\begin{proof}
    Suppose $(J;0;\gamma_1,\dots,\gamma_n;\tau_1,\dots,\tau_r)$ is JN-realisable with the homomorphism $\rho: \pi_1(M)\to \hstilde$ satisfying the conditions of JN-realisability. We will show that $(J, \{1, \ldots, r\}; [\alpha_*])$ is representation-detected by $(\rho, 0)$.

    Note that $(\rho,0) \in \mathcal{R}^*(\pi_1(M))$, since $\rho(h) = \sh(1)$.
    Since the translation number is a group homomorphism on abelian subgroups, the map $(\tau\circ \rho|_{\pi_1(T_i)}) \otimes \mathbf{1}_{\RR}:\pi_1(T_i) \otimes \RR \to \RR$ is a nontrivial linear map for all $i=1,2,\dots, r$. The kernel of this map contains $[\alpha_j]$ and divides $H_1(T_i;\RR)$ into a disjoint union $H_+ \cup H_-$, with $\tau(\rho(\gamma))>0$ for every $\gamma\in \pi_1(T_i)\cap H_+$ and $\tau(\rho(\gamma))<0$ for every $\gamma\in \pi_1(T_i)\cap H_-$. Then $\tau(\rho(\gamma))>0$ leads to $\rho(\gamma)(x)>x$ for all $x\in \RR$, and $\rho(\gamma)(x)<x$ when $\tau(\rho(\gamma))<0$. In particular, this shows that $[\mathcal{L}(\rho|_{\pi_1(T_i)}, 0)] = [\alpha_i]$ for all $i =1, \dots, r$, and so R1 holds.

    The observation that $\rho(\gamma)(x)>x$ for all $x\in \RR$ or $\rho(\gamma)(x)<x$ for all $x\in \RR$ whenever $\gamma \in H_+ \cup H_- \subset H_1(T_i;\RR)$ also shows that $(\rho, \rho(g)(0)) \in \mathcal{R}^*(\pi_1(M))$ for all $g \in \pi_1(M)$. With this fact in hand, the fact that conjugation leaves the translation number invariant implies that R2 is true.

    Finally, given that $[\mathcal{L}(\rho|_{\pi_1(T_i)}, 0))] = [\alpha_i]$ we know that if $i\in J$ and $\tau_i=\frac{p}{q}$ is rational, then $\rho(x_i)$ is conjugate to $\sh(\tau_i)$, and thus $\rho(q x_i-ph)$ is the identity in $\hstilde$. This implies $\ker(\rho)\cap \pi_1 (T_i)= \vbracket{\alpha_i}$. Consequently, R3 is true with $\nu : \RR \rightarrow \RR$ the identity and $\phi = \rho$.
\end{proof}

The converse of the previous proposition also holds, provided we place a restriction on our representations. The proof is most easily presented using left-orderings and the notion of cofinal elements: An element $h \in G$ is \emph{cofinal} relative to a left-ordering $\ooo$ of $G$ (or $\ooo$-cofinal for short) if
\[G = \{ g \in G : \exists k \in \mathbb{Z} \mbox{ such that } h^{-k} <_{\ooo} g <_{\ooo} h^k \}.
\]

\begin{proposition}\label{prop:detectedimplyJN}
    \begin{enumerate}
        \item[]
        \item If $(J, \{1, \ldots, r\}; [\alpha_*])$ is representation-detected and $[\alpha_*]$ is horizontal, then $(J;0;\gamma_1,\dots,\gamma_n;\tau_1,\dots,\tau_r)$ is JN-realisable.
        \item If $(\emptyset,\emptyset; [\alpha_*])$ is representation-detected and $[\alpha_*]$ is horizontal, then $(\emptyset,\{1,\dots,r\}; [\alpha_*])$ is representation-detected.
    \end{enumerate}
\end{proposition}

\begin{proof} 
    To prove (1), apply Theorem \ref{prop:equbetweenorddetectandrepdetect} and let $\ooo$ be a left-ordering of $\pi_1(M)$ detecting $(J, \{1, \ldots, r\}; [\alpha_*])$. The result now follows from \cite[Proposition 5.4]{BC17}.

    For (2), apply Theorem \ref{prop:equbetweenorddetectandrepdetect} and suppose that $\ooo$ order-detects $(\emptyset, \emptyset; [\alpha_*])$. Then $h \in \pi_1(M)$ is $\ooo$-cofinal due to the relators $y_i^{a_i}=h^{b_i}$ and the assumption $[\alpha_j] \neq [h]$ which implies that $h$ is $\ooo|_{\pi_1(T_j)}$-cofinal in $\pi_1(T_j)$ for each $i=1,\dots,r$. Since $h$ is $\ooo$-cofinal, we conclude $\ooo$ is boundary-cofinal \cite[Definition 5.7]{BC23}, from which it follows that $\ooo$ order-detects $(\emptyset, \{1, \dots, r \}; [\alpha_*])$ by \cite[Proposition 7.2]{BC23}. Therefore $(\emptyset, \{1, \dots, r \}; [\alpha_*])$ is representation-detected by Theorem \ref{prop:equbetweenorddetectandrepdetect}.
\end{proof}




It is possible to calculate when $(J;0;\gamma_1,\dots,\gamma_n;\tau_1,\dots,\tau_r)$ is JN-realisable, and the technical tools required to do so are in Appendix \ref{sec:JN appendix}. We will need them shortly.



\section{Cable spaces, bases and gluing two pieces}\label{sec:cable}

\subsection{Notations and conventions}\label{subsec:cpqnotations}

Consider a fibered solid torus $V\cong D^2\times S^1$, which can be obtained by taking the cylinder $D^2\times I$ and identifying each $(x,0)$ with $(xe^{2\pi i\frac{p}{q}},1)$, where $p\geq1$ and $q> 1$ are coprime numbers. Let $h$ denote a regular fiber in the interior of the fibered solid torus, and let $n(h)$ be an open regular neighbourhood around $h$. The complement of $n(h)$ in $V$ is denoted by $\cpq$, namely, \[\cpq=V - n(h).\]

To calculate the fundamental group $\pi_1(\cpq)$, we employ the following strategy. Let $\frac{1}{2}D^2$ be a concentric disk within $D^2$ with half the radius of $D^2$. We set $V_0=\frac{1}{2}D^2 \times S^1$ and choose a regular fiber $h$ that lies on $\partial V_0$. In this context, $\cpq$ can be viewed as the gluing of $V_0 - n(h)$ and $\overline{(V-V_0)}-n(h)$ along an annulus, whose central curve is a regular fiber.

Note that $\overline{(V-V_0)}-n(h)$, being homotopy equivalent to a thickened torus, has the fundamental group $\ZZ\oplus \ZZ$. This group is generated by the standard meridian and longitude basis ${\mu,\lambda}$. Here, $\mu$ is the homotopy class of the curve $\partial D^2\times \{*\}$ and $\lambda$ is the homotopy class of $\{*\}\times S^1$. On the other hand, the space $V_0$ is homotopy equivalent to a solid torus with its fundamental group being infinite cyclic, generated by $t$, the homotopy class of the core curve of $V_0$. Applying the Seifert-van Kampen theorem, we deduce that the fundamental group of $\cpq$ is given by \[\pi_1(\cpq)=(\ZZ\oplus \ZZ)*_{\mu^p\lambda^q=t^q}\ZZ.\]
Let $K$ be a given knot in $S^3$. Consider the cable knot $\cpq(K)$. Its knot complement is constructed by gluing $\cpq$ to $S^3-n(K)$ along their respective boundaries, $\partial V$ and $\partial (S^3-n(K))$. This gluing operation is performed by identifying the generators $\mu$ and $\lambda$ of the fundamental group of $\cpq$ with the canonical meridian and longitude of $\partial (S^3-n(K))$, respectively.

\begin{figure}[ht]
    \begin{center}
        \includegraphics[width=8cm]{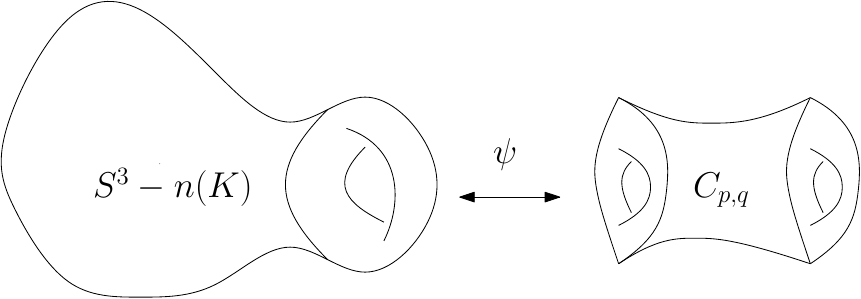}
    \end{center}
    \caption{Cable knot complement obtained by gluing $\cpq$ to ${S^3-n(K)}$}
    \label{fig:cableknot}
\end{figure}

In other words, the knot complement of the cable knot $\cpq(K)$ in $S^3$ can be expressed as
\[{S^3-n(\cpq(K))} = (S^3-n(K))\cup_\psi \cpq,\] where $\psi$ denotes the identification map. Refer to Figure \ref{fig:cableknot} for an illustrative depiction. We can compute the meridian and longitude of the cable knot $\cpq(K)$, denoted by $\mu_C$ and $\lambda_C$ respectively. They are given by the formulas \[\mu_C=t^s\mu^r\lambda^{-s}\mbox{ and } \lambda_C=\mu_C^{-pq}t^q,\] where $ps+qr=1$ and $-q<s<0<r\leq p$.

\begin{remark}
    The numbers $s$ and $r$ are known as B\'ezout's coefficients. The restriction $-q<s<0<r\leq p$ is not mandatory to obtain a correct expressions for $\mu_C$ and $\lambda_C$ in terms of $\mu,\lambda$ and $t$, but it is for computational convenience in later sections. Note that $r=p$ occurs if and only if $p=1$.
\end{remark}

Note that the subspace $\cpq$ is a Seifert-fibered space with incompressible tori as its boundary components. Therefore, the results of the previous section apply, and there exists an alternative construction for the cable space $\cpq$ that provides a different presentation of $\pi_1(\cpq)$ that agrees with the previous section. This alternative approach is more useful when we calculate sets of detected slopes.

We start with the same fibered solid torus as before and define $U=V-(\frac{1}{2}D^2\times S^1)$. Let $\mu$ and $\lambda$ denote the meridional and longitudinal classes in $\pi_1(U)$. Suppose $A$ is an annulus such that $U$ is homeomorphic to $A\times S^1$. Here, the $S^1$ factor corresponds to a regular fiber, $\mu^p\lambda^q$. The boundary $\partial A$ consists of curves $x_1$ and $y'$, representing $\mu^r\lambda^{-s}$ and $\mu^{-r}\lambda^s$ respectively, with $-q<s<0<r\leq p$ and $ps+qr=1$.

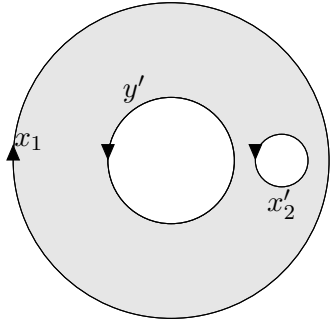
\begin{figure}[ht]
    \begin{center}
        \begin{tikzpicture}[line cap=round,line join=round,>=triangle 45,x=1.0cm,y=1.0cm, scale=0.7]
            \clip(-3.3,-3.3) rectangle (3.3,3.3);
            \draw [rotate around={0.:(0.,0.)},color=black] (0.,0.) ellipse (1.2cm and 1.2cm);
            \draw [rotate around={0.:(0.,0.)},color=black,fill=black,fill opacity=0.1] (0.,0.) ellipse (3.cm and 3.cm);
            \draw [rotate around={0.:(0.,0.)},color=black,fill=white] (0.,0.) ellipse (1.2cm and 1.2cm);
            \draw [rotate around={0.:(0.,0.)},color=black,fill=white] (2.1,0.) ellipse (0.5cm and 0.5cm);
            \draw [rotate around={0.:(0.,0.)},decoration={markings, mark=at position 0.5 with {\arrow{>}}}, postaction={decorate}] (0.,0.) ellipse (1.2cm and 1.2cm);
            \draw [rotate around={0.:(0.,0.)},decoration={markings, mark=at position 0.5 with {\arrow{>}}}, postaction={decorate}] (2.1,0.) ellipse (0.5cm and 0.5cm);
            \draw [rotate around={0.:(0.,0.)},decoration={markings, mark=at position 0.5 with {\arrow{<}}}, postaction={decorate}] (0.,0.) ellipse (3.cm and 3.cm);
            \node[above] at (-2.7,0) {$x_1$};
            \node[above] at (-0.7,0.9) {$y'$};
            \node[above] at (2.1,-1.3) {$x_2'$};
        \end{tikzpicture}
    \end{center}
    \caption{The planar surface $P$}
    \label{fig:surfaceP}
\end{figure}

By removing an open regular neighbourhood of a regular fiber from $U$, we obtain a space homeomorphic to $P\times S^1$, where $P$ is a planar surface with three boundary components. See Figure \ref{fig:surfaceP} for an illustration. The fundamental group of $U-n(h)$ is then given by \[\pi_1(U-n(h))=\langle x_1,x_2',y',h\ |\ h \mbox{ central, } y'x_1x_2' = 1 \rangle,\]
where $h$ is the class of a regular fiber. The cable space $\cpq$ can be obtained by attaching the solid torus $\frac{1}{2}D^2\times S^1$ back to $U-n(h)$ through mapping $\partial(\frac{1}{2}D^2)\times \{*\}$ to $(y')^{-q}h^s$. According to the Seifert-van Kampen theorem, the fundamental group of $\cpq$ is \[\pi_1(\cpq)=\langle x_1,x_2',y',h\ |\ h\mbox{ central, } (y')^q=h^s, \, y'x_1x_2' = 1 \rangle.\]

This presentation is different from the one discussed in the last section, since $\frac{s}{q} < 0$. To align it with that section, we define $y=y'h$ and $x_2=x_2'h^{-1}$ to obtain \[\pi_1(\cpq)=\langle x_1,x_2,y,h\ |\ h\mbox{ central, } y^q=h^{q+s}, \, yx_1x_2 = 1 \rangle.\]
A concrete isomorphism linking this presentation of $\pi_1(\cpq)$ and the previously computed presentation $\pi_1(\cpq)=(\mathbb{Z}\oplus \mathbb{Z})*_{\mu^p\lambda^q=t^q}\mathbb{Z}$ is specified by $\phi(x_1)=\mu^r\lambda^{-s},\, \phi(h)=\mu^p\lambda^q, \phi(y)=t^{q+s}, \phi(x_2)=\mu^{-r}\lambda^st^{-(s+q)}$.

Consider $\cpq$ as a subspace of the knot complement of the cable knot $\cpq(K)$. Each peripheral subgroup of $\pi_1(\cpq)$ now naturally has two bases, arising from our two presentations of the fundamental group above. The change of basis, computed via $\phi$, allows us to translate between slopes expressed relative to one basis and slopes relative to the other.

For the ``inner'' boundary torus $T_1$ of $\cpq$, we have two bases arising from the identification of $\partial (S^3-n(K))$ with $T_1$, namely $\mathcal{B}_K=\{\mu,\lambda\}$ and $\mathcal{B}_1=\{h,-x_1\}$. The change of basis matrix from $\{h, -x_1\}$ to $\{\mu,\lambda\}$, computed from $\phi$, is \[\begin{bmatrix} p & -r \\ q & s \end{bmatrix}.\]

The ``outer'' boundary torus $T_2$ of $\cpq$ has bases $\mathcal{B}_2=\{h,-x_2\}$ and $\mathcal{B}_C=\{\mu_C,\lambda_C\}$, with the corresponding change of basis matrix \[\begin{bmatrix} pq & pq+1 \\ 1 & 1 \end{bmatrix}.\]

To discuss relative JN-realisability and order-detection and apply results from previous sections, fix the convention that $[\alpha_*]=([\alpha_1],[\alpha_2]) \in \mathcal{S}(\cpq)$ means that $[\alpha_1]$ is slope on the ``inner'' boundary torus, and $[\alpha_2]$ on the ``outer'' boundary torus.

\subsection{Transition between bases and non-horizontal slopes}\label{subsec:slopetransition}

Use the notation developed above. In particular, there are integers $p,q,r,s$ with $p\geq1$ and $q> 1$ being coprime, $-q<s<0<r\leq p$ and $ps+qr=1$. Each of the bases $\mathcal{B}_K,\mathcal{B}_C,\mathcal{B}_1,\mathcal{B}_2$ gives rise to a homeomorphism $\mathcal{S}(T_i)\to \RR\cup\{\infty\}\cong S^1$ in the usual way. Specifically, for a basis $\mathcal{B} = \{b_1, b_2\}$ we map $[nb_1 + mb_2] \mapsto \frac{n}{m}$. Therefore, the change of basis matrix from $\mathcal{B}_K$ to $\mathcal{B}_1$ yields a homeomorphism $f:\RR\cup\{\infty\}\to \RR\cup\{\infty\}$ given by $f(x)=\frac{sx+r}{-qx+p}$, where $f(p/q)=\infty, f(-r/s)=0$ and $f(\infty)=-s/q\in (0,1)$. It is easy to verify that $f$ is strictly increasing on $(-\infty,p/q)$ and on $(p/q,\infty)$ and so that \[f((-\infty,p/q))=(-s/q,\infty)\ \mbox{ and }\ f((p/q,\infty))=(-\infty,-s/q).\]
Thus we have

\begin{lemma}\label{lemma:bktob1}
    For $b\in \RR$, we have \[f((-\infty,b))=\begin{cases}
            (-s/q,f(b))                      & \ \mbox{ if }b \leq p/q, \\
            [-\infty,f(b))\cup (-s/q,\infty] & \ \mbox{ if }b > p/q.
        \end{cases}\]
\end{lemma}
Similarly, the change of basis matrix from $\mathcal{B}_2$ to $\mathcal{B}_C$ yields a homeomorphism $g:\RR\cup\{\infty\}\to \RR\cup\{\infty\}$ given by $g(x)=pq+\frac{1}{x+1}$, where $g(-1)=\infty$ and $g(\infty)=pq$. Analogously, we have
\begin{lemma}\label{lemma:b2tobc}
    For $b\in \RR$ with $p-qb>0$, we have \[g((-1-\frac{1}{p-qb},-1))=(-\infty, pq-p+qb);\] if $p-qb<0$, then \[g((-1,-1-\frac{1}{p-qb}))=(pq-p+qb ,\infty).\]
\end{lemma}

In order to implement the results concerning JN-realisability, recall that we require $[\alpha_*]=([\alpha_1],[\alpha_2])$ to be horizontal, that is, $[\alpha_i]\not=[h]$ for $i=1,2$, meaning $[\alpha_i] \neq [p\mu+q\lambda]$ relative to the basis $\mathcal{B}_K$.
We will need the following to deal with the non-horizontal cases separately.

\begin{lemma}\label{lemma:nonhorizontalslope}
    For the cable space $\cpq$, $(\{1,2\},\{1,2\};[h],[h])$ is order-detected.
\end{lemma}

\begin{proof}
    Consider the short exact sequence \[1\to K\to \pi_1(\cpq)\to \ZZ\to 1,\] where the map $\pi_1(\cpq) \rightarrow \mathbb{Z}$ is the result of killing the fiber class $h\in \pi_1(\cpq)$ and any resulting torsion. In other words, $K=\vbracket{\vbracket{h,y}}$. Since $K \leq \pi_1(\cpq)$, $K$ is left-orderable. Now any left-ordering $\ooo$ obtained lexicographically from the short exact sequence above order-detects $(\{1,2\},\{1,2\};[h],[h])$.
\end{proof}

\begin{lemma}\label{lemma:onenonhorizontalslopeimplies2}
    Suppose that for the cable space $\cpq$, $(\emptyset, \emptyset; [\alpha_1], [\alpha_2])$ is order-detected. Then either $[\alpha_1] = [\alpha_2] = [h]$ or $[h] \notin \{ [\alpha_1], [\alpha_2]\}$.
\end{lemma}

\begin{proof}
    Let $\ooo$ be the ordering of $\pi_1(\cpq)$ order-detecting $(\emptyset, \emptyset; [\alpha_1], [\alpha_2])$. By the previous lemma, we need only to prove that if $[h] \in \{ [\alpha_1], [\alpha_2]\}$ then $[\alpha_1] = [\alpha_2] = [h]$. To this end, we prove that if $[h] = [\alpha_1]$ then $[\alpha_2] = [h]$, with the case of $[h] = [\alpha_2]$ implying that $[h] = [\alpha_1]$ is similar.

    We do this by showing the contrapositive. Suppose that $[h] \neq [\alpha_2]$ and recall the generators $x_1, x_2, y, h$ of $\pi_1(\cpq)$ from Section \ref{subsec:cpqnotations}. Then observe that $h$ is cofinal with respect to $\ooo|_{\pi_1(T_2)}$, and that
    \[ H = \{ g \in \pi_1(\cpq) : \exists k \in \mathbb{Z} \mbox{ such that } h^{-k} <_{\ooo} g <_{\ooo} h^k \}
    \]
    is therefore a subgroup containing $x_2$. However, $H$ also contains $y$, and hence $x_1$, and thus $H = \pi_1(\cpq)$. It follows that $h$ is $\ooo$-cofinal, so that $[\alpha_1]$ cannot be equal to $[h]$.
\end{proof}

Note that the two lemmas above ensure that when translating slopes from one basis to the other, we are able to take care of the slopes which are expressed as $\infty$ relative to the given bases.

\subsection{Special cases of the gluing theorems for use when cabling}

In Section \ref{sec:finally cables}, we are interested in the case of gluing $\cpq$ to $\partial M$ where $M$ is a knot manifold. We highlight this special case of our gluing theorems here.

\begin{corollary}\label{corollary:gluingdetected}
    Let $M$ be a knot manifold with boundary torus $T_{1,1}$ and $\cpq$ a cable space with two torus boundary components $T_{2,1}$ and $T_{2,2}$. Fix a choice of peripheral subgroup $\pi_1(T_{i,j})$ for each boundary torus. Let $f:T_{1,1}\to T_{2,1}$ be a homeomorphism that identifies the slope $[\alpha_1]\in \mathcal{S}(\pi_1(T_{1,1}))$ with the slope $[\alpha_2]\in \mathcal{S}(\pi_1(T_{2,1}))$.
    \begin{enumerate}
        \item Suppose that $\ooo\in \LO(\pi_1(M))$ detects $(\emptyset,\emptyset;[\alpha_1])$ and that $\ooo'\in\LO(\pi_1(\cpq))$ detects $(\emptyset,\{2\};[\alpha_2],[\alpha_3])$. Then $M\cup_{f}\cpq$ is left-orderable and $(\emptyset,\emptyset;[\alpha_3])$ is order-detected by some left-ordering in $\LO(\pi_1(M\cup_f \cpq))$.
        \item
              Suppose that $\ooo\in \LO(\pi_1(M))$ detects $(\emptyset,\{1\};[\alpha_1])$ and that $\ooo'\in\LO(\pi_1(\cpq))$ detects $(\emptyset,\{2,3\};[\alpha_2],[\alpha_3])$. Then $M\cup_{f}\cpq$ is left-orderable and $(\emptyset,\{1\};[\alpha_3])$ is order-detected by some left-ordering in $\LO(\pi_1(M\cup_f \cpq))$.
        \item Suppose that $\ooo\in \LO(\pi_1(M))$ detects $(\{1\},\{1\};[\alpha_1])$ and that $\ooo'\in\LO(\pi_1(\cpq))$ detects $(\{2,3\},\{2,3\};[\alpha_2],[\alpha_3])$. Then $M\cup_{f}\cpq$ is left-orderable and $(\{1\},\{1\};[\alpha_3])$ is order-detected by some left-ordering in $\LO(\pi_1(M\cup_f \cpq))$.
    \end{enumerate}
\end{corollary}

\begin{proof}
    Case (1) is a special application of Theorem \ref{thm:special_gluing_case}.
    First, we note that $\ooo'$ can be replaced by an ordering detecting $(\emptyset, \{1, 2\}; [\alpha_2], [\alpha_3])$ by Proposition \ref{prop:detectedimplyJN} and Theorem \ref{prop:equbetweenorddetectandrepdetect}. Next, we verify that $\cpq$ satisfies the conditions required of $M_2$ in that theorem as follows. To this end,
    let $\cpq(\alpha)$ denote the Dehn filling of $\cpq$ obtained by attaching a solid torus to $T_{2,1}$ in such a way that the meridian is identified with the slope $[\alpha] \in \mathcal{S}(T_{2,1})$. Now let $[\alpha] \in \mathcal{S}(T_{1,2})$ be an arbitrary rational slope other than the fiber slope. By \cite[Proposition 5]{heil74}, $\cpq(\alpha)$ is an irreducible Seifert fibered manifold with one torus boundary component. It follows that $\pi_1(\cpq(\alpha))$ is left-orderable, since it surjects onto $\mathbb{Z}$ \cite[Theorem 1.1]{BRW05}. Therefore, a short exact sequence argument, using
    \[ 1 \rightarrow \langle \langle \alpha \rangle \rangle \rightarrow \pi_1(\cpq) \rightarrow \pi_1(\cpq(\alpha)) \rightarrow 1
    \] to lexicographically order $\pi_1(\cpq)$, shows that $(\{1 \}, \{1\}; [\alpha], [\beta])$ is order-detected for some $[\beta] \in \mathcal{S}(T_{2,2})$. When $[\alpha]$ is the slope of the fiber, we may use Lemma \ref{lemma:nonhorizontalslope}. Finally, we observe that $\mathrm{LO}(\pi_1(\cpq))$ is compact and that the slope map $s_{1,1}$ is continuous, having all rational slopes in its image. As rational slopes form a dense set in $\mathcal{S}(T_{1,2})$, the map $s_{1,1}$ must be surjective. Thus, whenever $[\alpha]$ is irrational, $(\emptyset, \{1\}; [\alpha], [\beta])$ is order-detected for some $[\beta] \in \mathcal{S}(T_{2,2})$ by arguments analogous to \cite[Proposition 7.6]{BC23}.
    Cases (2) and (3) are applications of Theorem \ref{thm:main_gluing_thm}.
\end{proof}

\section{Computations of relatively JN-realisable slopes in cable spaces}
\label{sec:cpqcalculations}

We adopt the notation introduced in Sections \ref{sec:seifertmanifold} and \ref{sec:cable}. Consider a Seifert fibered manifold as described in Section \ref{sec:seifertmanifold}, and let $J$ be a subset of $\{1,\dots,r-1\}$ and $\tau_*=(\tau_1,\dots,\tau_{r-1})\in \RR^{r-1}$. Define \[\TT{M;J;\tau_*}=\{\tau'\in \RR : (J;0;\gamma_1,\dots,\gamma_n;\tau_1,\dots,\tau_{r-1},\tau') \mbox{ is JN-realisable}\}\] and \[\TTT{M;J;\tau_*}=\{\tau'\in \RR : (J\cup\{r\};0;\gamma_1,\dots,\gamma_n;\tau_1,\dots,\tau_{r-1},\tau') \mbox{ is JN-realisable}\}.\]

For this section, and in the Appendix, we follow \cite{BC17} and for a fixed tuple $\tau_*=(\tau_1,\dots,\tau_{r-1})$, we set \begin{itemize}
    \item $r_1=|\{j : \tau_j\notin \ZZ, 1\leq j\leq r-1\}|$, the number of non-integral $\tau_j$;
    \item $s_0=|\{j : \tau_j\in \ZZ \mbox{ and } j\in \{1,2,\dots,r-1\}\backslash J\}|$, the number of integral $\tau_j$ whose indices are not in $J$;
    \item $b_0=-(\floor{\tau_1}+\dots+\floor{\tau_{r-1}})$;
    \item $m_0=b_0-(n+r_1+s_0-1)$;
    \item $m_1=b_0+s_0-1$.
\end{itemize}


We can explicitly calculate how $\TT{\cpq;\emptyset;\tau}$ changes for different values of $\tau$. Recall that $p\geq1$ and $q> 1$ are coprime and $r, s$ are chosen so that $ps + qr = 1$ and $-q<s<0<r\leq p$. For $\tau\in \RR$, let $\bar{\tau}$ denote its fractional part, that is, $\bar{\tau}=\tau-\floor{\tau}$. Then $(J; 0; \gamma_1, \dots, \gamma_n; \tau_1, \dots, \tau_r)$ is JN-realisable if and only if $(J; b; \gamma_1, \dots, \gamma_n; \bar{\tau_1}, \dots, \bar{\tau_r})$ is JN-realisable, where
\[b = -(\lfloor \tau_1 \rfloor + \dots + \lfloor \tau_r \rfloor ).\]

\begin{proposition}\label{prop:tcpqinterval}
    For $\tau\in \mathbb{R}$, we have $\TT{\cpq;\emptyset;\tau}\subset (-\lfloor{\tau}\rfloor-2,-\lfloor{\tau}\rfloor]$. More precisely, there are two rational numbers $\eta(\tau)\in (-\lfloor{\tau}\rfloor-2,-\lfloor{\tau}\rfloor-1)$ and $\xi(\tau)\in (-\lfloor{\tau}\rfloor-1,-\lfloor{\tau}\rfloor)$ such that
    \begin{equation*}
        \TT{\cpq;\emptyset;\tau} =
        \begin{cases}
            [-\tau-1,-\tau]                      & \text{if } \bar{\tau}=0,                                  \\
            [-\lfloor{\tau}\rfloor-1,\xi(\tau)]  & \text{if } \bar{\tau}>0 \text{ and } \gamma+\bar{\tau}<1, \\
            \{-\lfloor{\tau}\rfloor-1\}          & \text{if } \bar{\tau}>0 \text{ and } \gamma+\bar{\tau}=1, \\
            [\eta(\tau),-\lfloor{\tau}\rfloor-1] & \text{if } \bar{\tau}>0 \text{ and } \gamma+\bar{\tau}>1. \\
        \end{cases}
    \end{equation*}

\end{proposition}

\begin{proof}
    Note that we have $n=1,\, \gamma=\frac{q+s}{q}$, $r_1+s_0=1$, $b_0=-\lfloor{\tau}\rfloor$ and $m_0=-\lfloor{\tau}\rfloor-1$.
    If $\tau$ is an integer, then $\bar{\tau} = 0$ and Theorem \ref{thm:integraltau} applies. So $(\emptyset; 0; \frac{q+s}{q}; \tau, \tau')$ is JN-realisable if and only if $(\emptyset; -\lfloor \tau \rfloor - \lfloor \tau' \rfloor; \frac{q+s}{q}; 0, \bar{\tau'})$ is JN-realisable, which happens only when $\tau = \lfloor \tau \rfloor =-\lfloor \tau' \rfloor$ yielding $\tau' \in [- \tau -1, - \tau ]$.

    If $\tau$ is not an integer, then $n+r_1=2$ and we can resolve this case by appealing to Theorems \ref{thm:rawallpositive} and \ref{thm:rawcases}(2)(b), \ref{thm:rawcases}(3)(b), \ref{thm:rawcases}(4)(b), considering the cases $\bar{\tau}+\gamma = 1$, $\bar{\tau}+\gamma>1$, and $\bar{\tau}+\gamma<1$. Observing that $\TT{\cpq;\emptyset;\tau} \subset ( -\lfloor \tau \rfloor -2, -\lfloor \tau \rfloor)$, we consider the following cases.

    First, if $\bar{\tau}+\gamma = 1$ then Theorem \ref{thm:rawcases}(4)(b) applies and we find $\TT{\cpq;\emptyset;\tau} = \{ m_0 \} = \{ -\lfloor \tau \rfloor -1 \}$.

    Next, if $\gamma+\bar{\tau}<1$, then the conditions of Theorem \ref{thm:rawcases}(3)(b)(i) are met while the conditions of Theorem \ref{thm:rawcases}(2)(b)(i) are not. Therefore, we choose $\xi(\tau)\in (-\lfloor{\tau}\rfloor-1,-\lfloor{\tau}\rfloor)$ that is maximal subject to the property that there are coprime integers $0<a<m$ and a permutation $\{a_1,a_2,a_3\}$ of $\{a,m-a,1\}$ such that $\gamma<\frac{a_1}{m}$, $\bar{\tau}\leq \frac{a_2}{m}$, and $\xi(\tau)-\floor{\xi(\tau)}=\xi(\tau)+\floor{\tau}+1\leq \frac{a_3}{m}$. Using Theorem \ref{thm:rawallpositive}(4), we conclude that $\TT{\cpq;\emptyset;\tau} = [-\lfloor{\tau}\rfloor-1,\xi(\tau)]$.

    If $\bar{\tau}\not=0$ and $\gamma+\bar{\tau}>1$ then the conditions of Theorem \ref{thm:rawcases}(2)(b)(i) are met and the conditions of Theorem \ref{thm:rawcases}(2)(3)(i) are not, so we choose $\eta(\tau)\in (-\lfloor{\tau}\rfloor-2,-\lfloor{\tau}\rfloor-1)$ that is minimal subject to the property that there are coprime integers $0<a<m$ and a permutation $\{a_1,a_2,a_3\}$ of $\{a,m-a,1\}$ such that $1-\gamma<\frac{a_1}{m}$, $1-\bar{\tau}\leq \frac{a_2}{m}$, and $1-(\eta(\tau)-\floor{\eta(\tau)})=-\eta(\tau)-\floor{\tau}-1\leq \frac{a_3}{m}$. Then by Theorem \ref{thm:rawallpositive}(4) we conclude that $\TT{\cpq;\emptyset;\tau} = [\eta(\tau),-\lfloor{\tau}\rfloor-1]$.
\end{proof}

Recall that $-q<s<0$, and note $1-\gamma=\frac{-s}{q}$. To visualize how the sets change with $\tau$, we suppose that $n\leq \tau \leq n+1$ for some integer $n$. Proposition \ref{prop:tcpqinterval} tells us that:
\begin{equation*}
    \TT{\cpq;\emptyset;\tau} =
    \begin{cases}
        [-(n+1),-n]         & \text{if }\tau=n,                     \\
        [-(n+1),\xi(\tau)]  & \text{if } n<\tau < n+\frac{-s}{q},   \\
        \{-(n+1)\}          & \text{if }\tau = n+\frac{-s}{q},      \\
        [\eta(\tau),-(n+1)] & \text{if } n+\frac{-s}{q} <\tau< n+1, \\
        [-(n+2),-(n+1)]     & \text{if } \tau=n+1,
    \end{cases}
\end{equation*}
where $\eta(\tau)\in (-n-2,-n-1)$ and $\xi(\tau)\in (-n-1,-n)$ are some rational numbers depending on the value of $\tau$.

\begin{proposition}\label{prop:tcpqintervalbehav}
    Assume $n\leq \tau_1\leq\tau_2 \leq n+1$ for some integer $n$.
    \begin{enumerate}
        \item If $0\leq \bar{\tau_1}\leq \bar{\tau_2}\leq 1-\gamma$, then \[\{-(n+1)\}\subset \TT{\cpq;\emptyset;\tau_2}\subset \TT{\cpq;\emptyset;\tau_1}\subset [-(n+1),-n].\]
        \item If $1-\gamma\leq \bar{\tau_1}\leq \bar{\tau_2}< 1$, then \[\{-(n+1)\}\subset \TT{\cpq;\emptyset;\tau_1}\subset \TT{\cpq;\emptyset;\tau_2}\subset [-(n+2),-(n+1)].\]
    \end{enumerate}
\end{proposition}

\begin{proof}
    First, consider the case where $0\leq \bar{\tau_1}\leq \bar{\tau_2}\leq 1-\gamma$. It follows immediately from Proposition \ref{prop:tcpqinterval} that
    \[\{-(n+1)\}\subset \TT{\cpq;\emptyset;\tau_1} \cap \TT{\cpq;\emptyset;\tau_2} \subset \TT{\cpq;\emptyset;\tau_1}\cup \TT{\cpq;\emptyset;\tau_2} \subset [-(n+1),-n].\]
    It remains to show that $\TT{\cpq;\emptyset;\tau_2}\subset \TT{\cpq;\emptyset;\tau_1}$.

    Suppose $x\in \TT{\cpq;\emptyset;\tau_2}$. If $x$ is an integer, then either $x=-(n+1)\in \TT{\cpq;\emptyset;\tau_1}$ or $x=-n$ with $\tau_2=n$. In the latter case, since $n\leq \tau_1\leq\tau_2=n$, we have $\tau_1=\tau_2=n$ and so $x\in \TT{\cpq;\emptyset;\tau_1}$. Now, assuming $x$ is non-integral, then by definition, $(\emptyset;0;\gamma;\tau_2,x)$ is JN-realisable which implies that $(\emptyset;-(\floor{\tau_2}+\floor{x});\gamma;\bar{\tau_2},\bar{x})$ is also JN-realisable. Note that $n \leq \tau_1 \leq \tau_2$ and $-(n+1) < x<-n$. However, when $\tau_2 = n+1$, $\TT{\cpq;\emptyset;\tau_2} = [-(n+2), -(n+1)]$ so that $-(n+1) < x<-n$ is not possible. Therefore, $-(\lfloor \tau_2 \rfloor + \lfloor x \rfloor) = -(n-(n+1))=1$.



    Therefore, $-(\floor{\tau_2}+\floor{x}) = 1$ and there are coprime integers $0<a<m$ and a permutation $\{a_1,a_2,a_3\}$ of $\{a,m-a,1\}$ such that $\gamma<\frac{a_1}{m}$, $\bar{\tau_2}\leq \frac{a_2}{m}$, and $\bar{x}\leq \frac{a_3}{m}$, according to Theorem \ref{thm:rawallpositive}(4). Because $\bar{\tau_1}\leq \bar{\tau_2} \leq \frac{a_2}{m}$, $(\emptyset;1;\gamma;\bar{\tau_1},\bar{x})$ is JN-realisable and so $x\in \TT{\cpq;\emptyset;\tau_1}$. It follows that (1) is proved.

    Now, suppose $1-\gamma \leq \bar{\tau_1}\leq \bar{\tau_2}< 1$. Proposition \ref{prop:tcpqinterval} tells us that
    \[\{-(n+1)\}\subset \TT{\cpq;\emptyset;\tau_1} \cap \TT{\cpq;\emptyset;\tau_2} \subset \TT{\cpq;\emptyset;\tau_1}\cup \TT{\cpq;\emptyset;\tau_2} \subset [-(n+2),-(n+1)].\]
    To complete the proof, we need to show $\TT{\cpq;\emptyset;\tau_1}\subset \TT{\cpq;\emptyset;\tau_2}$.

    Consider an element $x\in \TT{\cpq;\emptyset;\tau_1}$. If $x$ is an integer, then $x=-(n+1)\in \TT{\cpq;\emptyset;\tau_2}$, since $\tau_1$ cannot be an integer by assumption. Now, suppose that $x$ is not integral. By definition, $(\emptyset;0;\gamma;\tau_1,x)$ is JN-realizable and, therefore, equivalently, $(\emptyset;-(\floor{\tau_1}+\floor{x});\gamma;\bar{\tau_1},\bar{x})$ is JN-realisable. But $\lfloor \tau_1 \rfloor =1$ and $\lfloor x \rfloor = -n-2$, so the latter condition is that $(\emptyset;2;\gamma;\bar{\tau_1},\bar{x})$ is JN-realisable. By Theorem \ref{thm:rawallpositive}(3), this happens if and only if $(\emptyset;1;1-\gamma;1-\bar{\tau_1},1-\bar{x})$ is JN-realisable, and so we can find coprime integers $0<a<m$ and a permutation $\{a_1,a_2,a_3\}$ of $\{a,m-a,1\}$ so that $1-\gamma<\frac{a_1}{m}$, $1-\bar{\tau_1}\leq \frac{a_2}{m}$, and $1-\bar{x}\leq \frac{a_3}{m}$ by Theorem \ref{thm:rawallpositive}(4). Given $1-\gamma \leq \bar{\tau_1}\leq \bar{\tau_2}< 1$, we have $1-\bar{\tau_2}\leq 1-\bar{\tau_1}\leq \frac{a_2}{m}$. It follows that $(\emptyset;2;\gamma;\bar{\tau_2};\bar{x})$ is JN-realisable, and so $x\in \TT{\cpq;\emptyset;\tau_2}$. This completes our proof.
\end{proof}

\begin{remark}
    According to the previous two propositions, the sets $\TT{\cpq;\emptyset;\tau}$ act as an ``inch-worm'' moving towards $-\infty$ and $\infty$ as $\tau$ increases and decreases to $\pm \infty$ respectively. To see this, fix $n\in \mathbb{Z}$. Starting from $\tau = n $ to $\tau = n+ \frac{-s}{q}$ the inchworm pulls up its tail from $-n$ to $-(n+1)$, while keeping its head fixed at $-(n+1)$. Then, as $\tau$ increases from $n+\frac{-s}{q}$ to $n+1$ the inchworm moves its head forward from $-(n+1)$ to $-(n+2)$. 
\end{remark}

\begin{corollary}\label{cor:tcpqbehav}
    Let $\TT{\cpq;\emptyset;\tau}=[f(\tau),g(\tau)]$. The functions $\tau\mapsto f(\tau)$ and $\tau\mapsto g(\tau)$ are non-increasing. Moreover, $f$ decreases only over intervals where $g$ is constant and vice versa.
\end{corollary}

For each integer $n$, we have that \[[-(n+1),-n]= \TT{\cpq;\emptyset;n}=\bigcup_{n-1+\frac{-s}{q}\leq \tau \leq n+\frac{-s}{q}} \TT{\cpq;\emptyset;\tau},\] and it follows that \[\bigcup_{n\leq \tau} \TT{\cpq;\emptyset;\tau} = (-\infty,-n]\quad \mbox{ and }\quad \bigcup_{\tau \leq n} \TT{\cpq;\emptyset;\tau} = [-(n+1),\infty).\]

And for general $\tau_1, \tau_2\in \RR$, we have
\begin{align*}
    \bigcup_{\tau \geq \tau_1} \TT{\cpq;\emptyset;\tau} & =
    \begin{cases}
        (-\infty,-\floor{\tau_1}]                                  & \text{if }\bar{\tau_1}=0,             \\
        (-\infty,-\floor{\tau_1}-1]\cup \TT{\cpq;\emptyset;\tau_1} & \text{if } 0<\bar{\tau_1}<1-\gamma,   \\
        (-\infty,-\floor{\tau_1}-1]                                & \text{if } \bar{\tau_1}\geq 1-\gamma,
    \end{cases} \\
    \bigcup_{\tau \leq \tau_2} \TT{\cpq;\emptyset;\tau} & =
    \begin{cases}
        [-\floor{\tau_2}-1 ,\infty)                                & \text{if } \bar{\tau_2}\leq 1-\gamma, \\
        \TT{\cpq;\emptyset;\tau_2}\cup [-\floor{\tau_2}-1 ,\infty) & \text{if } \bar{\tau_2}> 1-\gamma.
    \end{cases}
\end{align*}

Analogues of these results can be obtained for $\TTT{\cpq;J;\tau}$. If $J=\emptyset$, then it follows from Theorem \ref{thm:rawcases} that when $\TT{\cpq;\emptyset;\tau}$ is a nondegenerate interval, $\TTT{\cpq;\emptyset;\tau}$ is its interior, and when $\TT{\cpq;\emptyset;\tau}$ is a degenerate interval, $\TTT{\cpq;\emptyset;\tau}$ coincides with it. If $J=\{1\}$, then Theorem \ref{thm:rawcases} still applies but some special attention is needed, for example, $\TTT{\cpq;\{1\};\tau} = \{-\tau-\gamma\}$ is no longer an interval, if $\tau$ is an integer.

\section{Detected slopes and cabling}
\label{sec:finally cables}

Let $M$ be a 3-manifold with incompressible torus boundary components $T_1,\dots,T_n$ and fixed choices of peripheral subgroups $\pi_1(T_i) \subset \pi_1(M)$. We recall our notation from earlier sections and fix the new notation needed to complete our computations relative to certain choices of bases for the peripheral subgroups $\pi_1(T_i)$.

Recall from Section \ref{sec:defs} that our sets of order-detected and representation-detected slopes are written $\mathcal{D}_{ord}(J,K;M)$ and $\mathcal{D}_{rep}(J,K;M)$ respectively, with each of these being subsets of $\mathcal{S}(M) = \mathcal{S}(T_1) \times \ldots \times \mathcal{S}(T_n)$.

Fix bases $\mathcal{B}_j=\{h_j, h_j^*\}$ for $\pi_1(T_j)$, $j=1,\dots,n$, and set $\mathcal{B}=\{\mathcal{B}_j : j=1,\dots,n\}$, and subsets $J \subset K \subset \{1, \ldots, n\}$. Then we define
\[ \mathcal{T}_{rep}(J,K;M;\mathcal{B}) = \{ (\tau_1, \dots, \tau_n) \in (\mathbb{R} \cup \{\infty\})^n : ([\tau_1 h_1+h_1^*],\dots,[\tau_n h_n+h_n^*]) \in \mathcal{D}_{rep}(J,K;M)\},
\]
were $\infty$ appears in the $i$-th coordinate in place of the slope $[h_i]$.
We similarly define $ \mathcal{T}_{ord}(J,K;M;\mathcal{B})$, and note that in our new notation, Theorem \ref{prop:equbetweenorddetectandrepdetect} says precisely that $ \mathcal{T}_{rep}(J,K;M;\mathcal{B}) = \mathcal{T}_{ord}(J,K;M;\mathcal{B}')$ as long as $\mathcal{B} = \mathcal{B}'$; however, these sets could differ if $\mathcal{B} \neq \mathcal{B}'$.



\subsection{Attaching cable spaces to knot manifolds}
Our next theorem shows how, with appropriate changes of bases, we can combine the notion of JN-realisability with that of representation and order-detection to calculate how detected slopes on the boundary of a knot manifold behave with respect to gluing on a copy of the cable space $\cpq$. We first deal with the slope $\infty$ separately.

\begin{proposition}\label{prop:infinite_slope}
    Assume that $p\geq1$ and $q> 1$ are coprime integers. Suppose that $M' = M \cup \cpq$ is a knot manifold that can be expressed as a union of a knot manifold $M$ and a cable space $\cpq$ whose inner boundary torus (see Section \ref{subsec:cpqnotations} for our inner and outer conventions) is identified with $\partial M$.
    Fix bases $\mathcal{B}_1=\{h,-x_1\}$ and $\mathcal{B}_2=\{h,-x_2\}$ of $\pi_1(\partial M)$ and $\pi_1(\partial M')$ as described in Section \ref{subsec:cpqnotations}. Then $\infty \in \mathcal{T}_{ord}(J, K; M; \mathcal{B}_1)$ if and only if $\infty \in \mathcal{T}_{ord}(J, K; M'; \mathcal{B}_2)$, where $(J,K)$ is any one of $(\emptyset, \emptyset), (\emptyset, \{1\}), (\{1\}, \{1\})$.
\end{proposition}

\begin{proof}
    Begin with $\infty \in \mathcal{T}_{ord}(J, K; M; \mathcal{B}_1)$ and let $\ooo$ be a left-ordering of $\pi_1(M)$ that order-detects $(J, K; [h])$. By Lemma \ref{lemma:nonhorizontalslope} there is an ordering $\ooo'$ of $\pi_1(\cpq)$ that detects $(\{1,2\}, \{1,2\}; [h], [h])$. By Corollary \ref{corollary:gluingdetected}, $(J, K; [h])$ is order-detected by some left-ordering of $\pi_1(M')$.
\end{proof}

The next theorem deals with the remaining slopes via JN-realisability.

\begin{theorem}\label{thm:cpqtord}
    Assume that $p\geq1$ and $q> 1$ are coprime integers. Suppose that $M' = M \cup \cpq$ is a knot manifold that can be expressed as a union of a knot manifold $M$ and a cable space $\cpq$ whose inner boundary torus (see Section \ref{subsec:cpqnotations} for our inner and outer conventions) is identified with $\partial M$.
    Fix bases $\mathcal{B}_1=\{h,-x_1\}$ and $\mathcal{B}_2=\{h,-x_2\}$ of $\pi_1(\partial M)$ and $\pi_1(\partial M')$ as described in Section \ref{subsec:cpqnotations}. Then, we have
    \begin{enumerate}
        \item $ \displaystyle \bigcup_{\tau\in \mathcal{T}_{ord}(\emptyset, \emptyset; M; \mathcal{B}_1)\setminus \{ \infty\} }\TT{\cpq;\emptyset;\tau} = \mathcal{T}_{ord}(\emptyset, \emptyset; M'; \mathcal{B}_2)\setminus \{ \infty\} ,$
        \item $ \displaystyle \bigcup_{\tau\in \mathcal{T}_{ord}(\emptyset, \{1\}; M; \mathcal{B}_1)\setminus \{ \infty\} } \TT{\cpq;\emptyset;\tau} \subset \mathcal{T}_{ord}(\emptyset, \{1\}; M'; \mathcal{B}_2)\setminus \{ \infty\} ,$
        \item $\displaystyle \bigcup_{\tau\in \mathcal{T}_{ord}(\{1\}, \{1\}; M; \mathcal{B}_1)\setminus \{ \infty\} } \TTT{\cpq;\{1\};\tau} \subset \mathcal{T}_{ord}(\{1\}, \{1\}; M'; \mathcal{B}_2)\setminus \{ \infty\} .$
    \end{enumerate}
\end{theorem}

\begin{proof}
    Choose $s, r$ such that $ps+qr=1$ with $-q<s<0<r\leq p$.
    First, we prove (1). Let $\tau'\in \bigcup_{\tau\in \mathcal{T}_{ord}(\emptyset, \emptyset; M; \mathcal{B}_1)\setminus \{ \infty\} }\TT{\cpq;\emptyset;\tau} $ be given. Then we can find some $\tau \in \mathcal{T}_{ord}(\emptyset, \emptyset; M; \mathcal{B}_1)\setminus \{ \infty\} $ such that $\tau'\in \TT{\cpq;\emptyset;\tau}$. In other words, $(\emptyset;0; \frac{q+s}{q};\tau,\tau')$ is JN-realisable. By Proposition \ref{prop:jnimplyrepdetected} and Theorem \ref{prop:equbetweenorddetectandrepdetect}, $(\emptyset,\{1,2\};[\tau h -x_1],[\tau' h - x_2])$ is order-detected in $\pi_1(\cpq)$. Since $(\emptyset,\emptyset;[\tau h -x_1])$ is also order-detected in $\pi_1(M)$, $(\emptyset,\emptyset;[\tau' h- x_2])$ is order-detected in $\pi_1(M')$ by Corollary \ref{corollary:gluingdetected}(1). Therefore, $\tau'\in \mathcal{T}_{ord}(\emptyset, \emptyset; M'; \mathcal{B}_2)\setminus \{ \infty\} $.

    For the inclusion in the other direction in part (1), take $\tau'\in \mathcal{T}_{ord}(\emptyset, \emptyset; M'; \mathcal{B}_2)\setminus \{ \infty\} $. Then there exists a left-ordering $\ooo$ of $\pi_1(M')$ such that $[\mathcal{L}(\ooo|_{\pi_1(\partial M')})] = [\tau' h - x_2]$. Then the restriction $\ooo|_{\pi_1(\partial M)}$ determines a slope $[\alpha] \in \mathcal{S}(M)$. By Lemma \ref{lemma:onenonhorizontalslopeimplies2}, $[\alpha] = [\tau h -x_1]$ for some $\tau \in \mathbb{R}$. Thus, $(\emptyset, \emptyset; \tau') \in \mathcal{T}_{ord}(\emptyset, \emptyset ; M; \mathcal{B}_1)\setminus \{ \infty\} $ and $(\emptyset, \emptyset; [\tau h -x_1], [\tau' h - x_2])$ is order-detected in $\pi_1(\cpq)$ by $\ooo|_{\pi_1(\cpq)}$. By Propositions \ref{prop:equbetweenorddetectandrepdetect} and \ref{prop:detectedimplyJN}(2), $(\emptyset, \{1,2\}; [\tau h -x_1], [\tau' h - x_2])$ is representation-detected. But then by Proposition \ref{prop:detectedimplyJN}(1), $(\emptyset;0; \frac{q+s}{q};\tau,\tau')$ is JN-realisable. Hence $\tau'\in \TT{\cpq;\emptyset;\tau}$.

    To show part (2), let $\tau'\in \bigcup_{\tau\in \mathcal{T}_{ord}(\emptyset, \{1\}; M; \mathcal{B}_1)\setminus \{ \infty\} } \TT{\cpq;\emptyset;\tau} $. Then we can find some $\tau \in \mathcal{T}_{ord}(\emptyset, \{1\}; M; \mathcal{B}_1)\setminus \{ \infty\} $ such that $\tau'\in \TT{\cpq;\emptyset;\tau}$. In other words, $(\emptyset;0; \frac{q+s}{q}; \tau,\tau')$ is JN-realisable. By Proposition \ref{prop:jnimplyrepdetected} and Theorem \ref{prop:equbetweenorddetectandrepdetect}, $(\emptyset,\{1,2\};[\tau h -x_1],[\tau' h +x_2])$ is order-detected in $\pi_1(\cpq)$. Since $(\emptyset,\{1\};[\tau h -x_1])$ is also order-detected in $\pi_1(M)$, $(\emptyset,\{1\};[\tau' h- x_2])$ is order-detected in $\pi_1(M')$ by Corollary \ref{corollary:gluingdetected}(2). Therefore, $\tau'\in \mathcal{T}_{ord}(\emptyset, \{1\}; M'; \mathcal{B}_2)\setminus \{ \infty\} $.


    For part (3), we proceed as in the previous case. Take $\tau'\in \TTT{\cpq;\{1\};\tau}$ for some $\tau\in \mathcal{T}_{ord}(\{1\}, \{1\}; M; \mathcal{B}_1)\setminus \{ \infty\} $. Then $(\{1\}, \{1\}, [\tau h -x_1])$ is order-detected in $\pi_1(M)$ and by Proposition \ref{prop:jnimplyrepdetected} and Theorem \ref{prop:equbetweenorddetectandrepdetect} $(\{1,2\},\{1,2\};[\tau h -x_1],[\tau' h -x_2])$ is order-detected in $\pi_1(\cpq)$. Then apply Corollary \ref{corollary:gluingdetected}(3) to conclude.
\end{proof}

Under certain assumptions, we can improve the second containment in the previous theorem so as to become an equality.

\begin{corollary}\label{cor:specialcasedetect}
    With the same assumptions as in Theorem \ref{thm:cpqtord}, if $\mathcal{T}_{ord}(\emptyset,\emptyset;M;\mathcal{B}_1)\not=\RR \cup \{\infty\}$ or if $\mathcal{T}_{ord}(\emptyset,\emptyset;M;\mathcal{B}_1)=\mathcal{T}_{ord}(\emptyset,\{1\};M;\mathcal{B}_1)$, then \[ \bigcup_{\tau\in \mathcal{T}_{ord}(\emptyset, \{1\}; M; \mathcal{B}_1)\setminus \{ \infty\} } \TT{\cpq;\emptyset;\tau} = \mathcal{T}_{ord}(\emptyset, \{1\}; M'; \mathcal{B}_2)\setminus \{ \infty\} .\]
\end{corollary}
\begin{proof}
    Note that $\mathcal{T}_{ord}(\emptyset,\emptyset;M;\mathcal{B}_1)\not=\RR \cup \{\infty\}$ implies $\mathcal{T}_{ord}(\emptyset,\emptyset;M;\mathcal{B}_1)=\mathcal{T}_{ord}(\emptyset,\{1\};M;\mathcal{B}_1)$ by \cite[Theorem 1.2]{BC23}. Thus, we assume that $\mathcal{T}_{ord}(\emptyset,\emptyset;M;\mathcal{B}_1)=\mathcal{T}_{ord}(\emptyset,\{1\};M;\mathcal{B}_1)$ in order to complete the proof. As Theorem \ref{thm:cpqtord}(2) holds, the proof will be finished if we show the reverse inclusion.

    Let $\tau'\in \mathcal{T}_{ord}(\emptyset, \{1\}; M'; \mathcal{B}_2)\setminus \{ \infty\}$. Then there is a left-ordering $\ooo$ of $\pi_1(M')$ with $[\mathcal{L}(\ooo|_{\pi_1(\partial M')})]=[\tau' h -x_2]$. The restriction $\ooo|_{\pi_1( M)}$, regarding $\pi_1(M)$ as a subgroup of $\pi_1(M')$, order-detects $(\emptyset,\emptyset;[\tau h-x_1])$ for some $\tau$ by Lemma \ref{lemma:onenonhorizontalslopeimplies2}. Since $\mathcal{T}_{ord}(\emptyset,\emptyset;M;\mathcal{B}_1)=\mathcal{T}_{ord}(\emptyset,\{1\};M;\mathcal{B}_1)$, we know that $(\emptyset,\{1\};[\tau h-x_1])$ is also order-detected. Similarly, the restriction $\ooo|_{\pi_1(\cpq)}$, regarding $\pi_1(\cpq)$ as a subgroup of $\pi_1(M')$, order-detects $(\emptyset,\emptyset;[\tau h-x_1],[\tau'h -x_2])$. By Proposition \ref{prop:detectedimplyJN}, $(\emptyset;0; \frac{q+s}{q};\tau,\tau')$ is JN-realisable. Therefore, $\tau'\in \TT{\cpq;\emptyset;\tau}$ with $\tau \in \mathcal{T}_{ord}(\emptyset, \{1\}; M; \mathcal{B}_1)\setminus \{ \infty\} $.
\end{proof}

In particular, if $\mathcal{T}_{ord}(\emptyset, \{1\}; M; \mathcal{B}_1)$ is known to be a proper subinterval of $\RR \cup \{\infty\}$, then the results of Section \ref{sec:cpqcalculations} allow us to compute the union $\bigcup_{\tau\in \mathcal{T}_{ord}(\emptyset, \{1\}; M; \mathcal{B}_1)} \TT{\cpq;\emptyset;\tau} $ and thus the set of detected slopes on the boundary of the manifold resulting from attaching a cable space.

\subsection{Cable knots in \texorpdfstring{$S^3$}{S3}}

We are able to give much more precise results in the case of cable knots in $S^3$. First, we require several technical computations of the intervals $\TT{\cpq;\emptyset;\tau}$ for specific values of $\tau$.

\begin{lemma}\label{lemma:brange}
    Suppose that $p\geq1$ and $q> 1$ are coprime and $ps+qr=1$ with $-q<s<0<r\leq p$. If $b$ is an integer satisfying $0\leq b\leq \frac{p}{q}$, then \[0<\min\{\frac{bs+r}{p-qb}, \frac{q+s}{q}\}\leq \frac{1}{2} <\max\{\frac{bs+r}{p-qb} ,\frac{q+s}{q}\}\leq 1.\] Moreover, $\frac{-s}{q}<\frac{(b-1)s+r}{p-q(b-1)}<\frac{bs+r}{p-qb}$.
\end{lemma}

\begin{proof}
    If $p=1$, then we have $s=1-q$ and $r=1$. The only possible choice of $b$ is $b=0$. The statements can be easily verified by direct calculations. So we may assume $p,q\geq 2$ and $-q<s<0<r< p$.

    Since $q(bs+r)+s(p-qb)=1$, $bs+r$ and $p-qb$ are coprime. Let us show the first set of inequalities. The assumption $-q<s<0$ gives $0<\frac{q+s}{q}<1$ immediately. Since $p,q\geq 2$ are coprime, $\frac{p}{q}$ is not integral and so $0\leq b \leq \frac{p-1}{q} < \frac{p}{q}$ and $p-qb>0$. It follows that $\frac{bs+r}{p-qb}> \frac{(\frac{p}{q})s +r}{p-qb}=\frac{1}{q(p-qb)}>0$. Therefore, we have $0<\min\{\frac{bs+r}{p-qb} ,\frac{q+s}{q}\}$. Observe that $\frac{bs+r}{p-qb}+\frac{q+s}{q}=1+\frac{1}{q(p-qb)}>1$. And hence $\frac{1}{2} <\max\{\frac{bs+r}{p-qb} ,\frac{q+s}{q}\}$.

    Next, suppose that $\frac{bs+r}{p-qb}>\frac{1}{2}$. Then $2(bs+r)-(p-qb)>0$ and so that $\frac{bs+r}{p-qb}-\frac{1}{2}= \frac{2(bs+r)-(p-qb)}{2(p-qb)}\geq \frac{1}{2(p-qb)}\geq \frac{1}{q(p-qb)}$. Therefore, we have $\frac{q+s}{q}=1+\frac{1}{q(p-qb)}-\frac{bs+r}{p-qb}\leq 1 +\frac{1}{q(p-qb)} - (\frac{1}{2}+\frac{1}{q(p-qb)})=\frac{1}{2}$. It follows that $\min\{\frac{bs+r}{p-qb}, \frac{q+s}{q}\}\leq \frac{1}{2}$. Finally, note that $\frac{bs+r}{p-qb}\leq 1$ if and only if $b\leq \frac{p-r}{s+q}$. The latter is true since $b \leq \frac{p-1}{q}$ and $\frac{p-1}{q}\leq \frac{p-r}{s+q}$ if and only if $1 \leq q+s$. This completes the proof of the first set of inequalities.

    For the first part of the last inequality, note that $\frac{-s}{q}<\frac{(b-1)s+r}{p-q(b-1)}$ is equivalent to $(-s)(p-q(b-1))<q((b-1)s+r)$, which reduces to $0< ps+qr =1$ and therefore holds. The second part follows from observing that the function $f(x) = \frac{xs+r}{p-qx}$ is increasing on $(\infty, \frac{p}{q})$.
\end{proof}

\begin{proposition}\label{prop:tpqspecialcase}
    Assume that $p\geq1$ and $q> 1$ are coprime integers and $r,s$ are chosen such that $ps+qr=1$ and $-q<s<0<r\leq p$. Let $b$ be an integer with $0\leq b\leq \frac{p}{q}$ and $\cpq$ be a cable space as defined in Section \ref{subsec:cpqnotations}. Then $0 < \frac{bs+r}{p-qb} \leq 1$ and \[\TT{\cpq;\emptyset;\frac{bs+r}{p-qb}} = [-1-\frac{1}{p-qb},-1].\] Moreover, if $p,q\geq 2$, then
    \[
        \TT{\cpq;\{1\};\frac{bs+r}{p-qb}}=
        \begin{cases}
            [-1-\frac{1}{p-q(b-1)},-1], & \frac{bs+r}{p-qb}<1  \\
            \{-\frac{2q+s}{q}\},        & \frac{bs+r}{p-qb}=1.
        \end{cases}
    \]

\end{proposition}

\begin{proof}
    If $p=1$, then we have $s=1-q$ and $r=1$. The only possible choice of $b$ is $b=0$ and so $\frac{bs+r}{p-qb}=1$. And by Proposition \ref{prop:tcpqinterval}, we see that $\TT{\cpq;\emptyset;1}=[-2,-1]$. So the claim holds. Now, we assume $p,q\geq 2$ and so $-q<s<0<r< p$.

    From Lemma \ref{lemma:brange}, we see that $0< \frac{bs+r}{p-qb} \leq 1$ and that ${bs+r}$ and ${p-qb}$ are coprime.
    First, we consider the case $\frac{bs+r}{p-qb}=1$. It follows that $bs+r=p-qb=1$. By Proposition \ref{prop:tcpqinterval}, we have $\TT{\cpq;\emptyset; \frac{bs+r}{p-qb}} =[-2,-1]$, which is consistent with the formula $\TT{\cpq;\emptyset;\frac{bs+r}{p-qb}} = [-1-\frac{1}{p-qb},-1]$. On the other hand, by the discussion at the beginning of Appendix \ref{thm:integraltau}, $(\{1\};0;\frac{q+s}{q};\frac{bs+r}{p-qb},\tau')$ is JN-realisable if and only if $(\emptyset;-1;\frac{q+s}{q};\tau')$ is JN-realisable, which happens if and only if $\tau' + \frac{q+s}{s} = -1$, meaning $\tau'=-\frac{2q+s}{q}$, that is, $\TT{\cpq;\{1\};\frac{bs+r}{p-qb}}=\{-\frac{2q+s}{q}\}$. So, the proposition holds in this case.

    By Lemma \ref{lemma:brange}, the remaining cases are $0<\frac{q+s}{q}\leq \frac{1}{2}<\frac{bs+r}{p-qb}<1$ and $0<\frac{bs+r}{p-qb}\leq \frac{1}{2}<\frac{q+s}{q}<1$. Note that regardless of whether $J=\{1\}$ or $J=\emptyset$ we have $\{-1\}\subset \TT{\cpq;J;\frac{bs+r}{p-qb}}\subset (-2,0)$ by Theorem \ref{thm:rawcases}(1). Furthermore, since $\frac{bs+r}{p-qb}+\frac{q+s}{q}=1+\frac{1}{q(p-qb)}>1$, we have \[ \TT{\cpq;J;\frac{bs+r}{p-qb}}\cap (-1,0)=\emptyset \mbox{ and } \TT{\cpq;J;\frac{bs+r}{p-qb}}\cap (-2,-1)\not=\emptyset\] by Theorem \ref{thm:rawcases}(2)(b) and (3)(b). Therefore, we have $\TT{\cpq;J;\frac{bs+r}{p-qb}}\subset (-2,-1]$.

    Next for $J=\{1\}$ or $J=\emptyset$, we need to determine all possible $\tau'$ with $-2<\tau' < -1$ such $(J;0;\frac{q+s}{q};\frac{bs+r}{p-qb},\tau')$ is JN-realisable. Since $-2<\tau' < -1$, $(J;0;\frac{q+s}{q};\frac{bs+r}{p-qb},\tau')$ is JN-realisable if and only if $(J;2;\frac{q+s}{q};\frac{bs+r}{p-qb},2+\tau')$ is JN-realisable, which is equivalent to JN-realisability of $(J;1;1-\frac{q+s}{q};1-\frac{bs+r}{p-qb},-1-\tau')$ by Theorem \ref{thm:rawallpositive}(3). Moreover, by Theorem \ref{thm:rawallpositive}(4), $(J;1;1-\frac{q+s}{q};1-\frac{bs+r}{p-qb},-1-\tau')$ is JN-realisable if and only if there are coprime $A,N$ with $0<A<N$ and a permutation $\{A_1,A_2,A_3\}$ of $\{A,N-A,1\}$ such that

    \[1-\frac{q+s}{q}< \frac{A_1}{N},\ 1-\frac{bs+r}{p-qb} \leq \frac{A_2}{N}\ \mbox{ and } -1-\tau'\leq \frac{A_3}{N}\] if $J=\emptyset$; if $J=\{1\}$, then the second inequality is strict. We rewrite these inequalities as \[ \frac{q+s}{q}>1- \frac{A_1}{N},\ \frac{bs+r}{p-qb} \geq 1-\frac{A_2}{N}\ (\mbox{resp. }\frac{bs+r}{p-qb} > 1-\frac{A_2}{N})\ \mbox{ and } \tau'\geq -1 -\frac{A_3}{N}\]
    to prepare for the following case-by-case analysis.

    \begin{casesp}
        \item $0<\frac{q+s}{q}\leq \frac{1}{2}<\frac{bs+r}{p-qb}<1$.

        Since $0<\frac{q+s}{q}\leq \frac{1}{2}$ and $\frac{q+s}{q}>1- \frac{A_1}{N}$, we must have $A_1\not=1$. So without loss of generality by replacing $A$ with $N-A$, we may assume $A_1=N-A$. It follows that $\frac{q+s}{q}>\frac{A}{N}$. We claim that $A_2=A$ and $A_3=1$. To see this, suppose not, and proceed as follows:

        From $\frac{q+s}{q}>\frac{A}{N}$ and $\frac{bs+r}{p-qb}\geq 1-\frac{1}{N}$, we have $s>q(\frac{A-N}{N})$ and $bs+r\geq (p-qb)(\frac{N-1}{N})$. If $A\geq 2$, then \begin{align*}
            1=(p-qb)s+q(bs+r) & >q(p-qb)(\frac{A-N}{N})+q(p-qb)(\frac{N-1}{N}) \\
                              & = q(p-qb)(\frac{A-1}{N})                       \\
                              & \geq q(p-qb)\frac{1}{N}                        \\
                              & \geq q(p-qb)(1-\frac{bs+r}{p-qb})              \\
                              & = q((p-qb)-(bs+r)) >0.
        \end{align*}
        Since $q((p-qb)-(bs+r))$ is an integer, and we have $1> q((p-qb)-(bs+r)) >0$, this is a contradiction. Therefore, $A=1$. If $A=1$, then $\frac{bs+r}{p-qb}\geq 1-\frac{1}{N}$ becomes $\frac{bs+r}{p-qb}\geq 1-\frac{A}{N}$ as claimed (with strict inequality if $J=\{1\}$).

        \begin{casesp}
            \item $J=\emptyset$.
            In this case, whether or not $\tau' \in \TT{\cpq; \emptyset;\frac{bs+r}{p-qb}}$ is determined by whether or not we can find two coprime integers $A,N$ with $0<A<N$ and \[\frac{q+s}{q}>\frac{A}{N},\ \frac{bs+r}{p-qb} \geq 1-\frac{A}{N}\ \mbox{ and } \tau'\geq -1 -\frac{1}{N}. \]
            Therefore $\TT{\cpq; \emptyset;\frac{bs+r}{p-qb}}=[-1-\frac{1}{M},-1],$ where $M\geq 2$ is the smallest $N$ for which such a coprime pair $A, N$ exist. We claim $M=p-qb$ to complete the proof in this case. Recall that $0<\frac{q+s}{q}\leq \frac{1}{2}<\frac{bs+r}{p-qb}<1$ and $0\leq b\leq \frac{p-1}{q}<\frac{p}{q}$.

            Taking $A=(p-qb)-(bs+r)>0$ and $M=p-qb$ gives $\frac{bs+r}{p-qb}=1-\frac{A}{M}\geq 1-\frac{A}{M}$. Also note $\frac{q+s}{q} > \frac{A}{M}$ if and only if $\frac{-s}{q}<\frac{bs+r}{p-qb}$, which is true by Lemma \ref{lemma:brange}. We also have $0<A<M$ since $M-A=bs+r>0$, and we know that $A,M$ are coprime since $(s+q)M-qA=1$. To see any such $N$ satisfies $N\geq M=p-qb$ so that $M$ is minimal, we proceed as follows.

            The inequalities $\frac{q+s}{q}>\frac{A}{N},\ \frac{bs+r}{p-qb} \geq 1-\frac{A}{N}$ are equivalent to $\frac{q+s}{q}>\frac{A}{N}\geq \frac{(p-qb)-(bs+r)}{p-qb}$, which is also equivalent to \[(p-qb)(q+s)N>Aq(p-qb)\geq((p-qb)-(bs+r))qN.\]
            It follows that $0<(p-qb)(q+s)N-Aq(p-qb) = (p-qb)((q+s)N-Aq)$ and thus $p-qb\leq (p-qb)(q+s)N-Aq(p-qb)$. Next note that $0\leq Aq(p-qb)+((bs+r)-(p-qb))qN$, adding this inequality to the previous one we arrive at
            \[(Aq(p-qb)+((bs+r)-(p-qb))qN)+((p-qb)(q+s)N-Aq(p-qb))\geq p-qb.\]
            Since $(p-qb)s+q(bs+r)=1$, the left-hand side of the inequality is actually equal to $N$. This completes the proof of minimality.

            \item $J=\{1\}$.
            In this case, we investigate pairs of coprime integers $A,N$ with $0<A<N$ and \[\frac{q+s}{q}>\frac{A}{N},\ \frac{bs+r}{p-qb} > 1-\frac{A}{N}\ \mbox{ and } \tau'\geq -1 -\frac{1}{N}. \]
            As above, it follows that $\TT{\cpq; \{1\};\frac{bs+r}{p-qb}}=[-1-\frac{1}{M},-1],$ where $M\geq 2$ is the smallest choice of $N$ for which such a coprime pair $A, N$ exists. We claim $M=p-q(b-1)$.

            Firstly, take $A=(p-q(b-1))-((b-1)s+r)=((p-qb)-(bs+r))+(q+s)>0$ and $M=p-q(b-1)=(p-qb)+q>0$. Since $M-A=(b-1)s+r>0$ and $(s+q)M-qA=1$, we have $0<A<M$ and $M,A$ are coprime. Furthermore, we observe that $\frac{bs+r}{p-qb}>1-\frac{A}{M}$ and $\frac{q+s}{q}>\frac{A}{M}$ are equivalent to $\frac{bs+r}{p-qb}>1-\frac{A}{M}> \frac{-s}{q}$, and since $1-\frac{A}{M}=\frac{(b-1)s+r}{p-q(b-1)}$ these inequalities hold by Lemma \ref{lemma:brange}.

            To see that this choice of $M$ is minimal, we proceed as the same as in Subcase 1(i) and suppose that $A, N$ with $0<A<N$ are another coprime pair satisfying the inequalities $\frac{q+s}{q}>\frac{A}{N}$ and $\frac{bs+r}{p-qb} > 1-\frac{A}{N}$. These are equivalent to $\frac{q+s}{q}>\frac{A}{N}> \frac{(p-qb)-(bs+r)}{p-qb}$, which is also equivalent to \[(p-qb)(q+s)N>Aq(p-qb)>	((p-qb)-(bs+r))qN.\]
            It follows that $0<(p-qb)(q+s)N-Aq(p-qb) = (p-qb)((q+s)N-Aq)$ and $0< Aq(q-qb)+((bs+r)-(p-qb))qN = q(A(q-qb)+((bs+r)-(p-qb))N)$. We conclude $p-qb\leq (p-qb)(q+s)N-Aq(p-qb)$ and $q\leq Aq(q-qb)+((bs+r)-(p-qb))qN$. Adding these inequalities, we see that \[(Aq(p-qb)+((bs+r)-(p-qb))qN)+((p-qb)(q+s)N-Aq(p-qb))\geq p-qb+q.\]
            By using $(p-qb)s+q(bs+r)=1$, we see that the left-hand side is equal to $N$, so our choice of $M$ above is minimal.
        \end{casesp}
        \item $0<\frac{bs+r}{p-qb}\leq \frac{1}{2}<\frac{q+s}{q}<1$.

        Since $\frac{1}{2}\geq \frac{bs+r}{p-qb}$ and $\frac{bs+r}{p-qb}\geq 1-\frac{A_2}{N}$ with the inequality being strict if $J=\{1\}$, we must have $A_2\not=1$. We may assume $A_2=N-A$ by replacing $A$ by $N-A$ if necessary. Using exactly the same reasoning as in Case 1, we can assume $A_1=N$ and $A_3=1$. So, the problem reduces to determining the existence of coprime $A, N$ with $0<A<N$ such that \[ \frac{q+s}{q}>1- \frac{A}{N},\ \frac{bs+r}{p-qb} \geq \frac{A}{N}\ \mbox{ and } \tau'\geq -1 -\frac{1}{N},\] with the second inequality being strict if $J=\{1\}$. To complete the proof, we mimic Case 1.

        \begin{casesp}
            \item $J=\emptyset$.
            In this case, we aim to find two coprime integers $A,N$ with $0<A<N$ and \[\frac{q+s}{q}>1-\frac{A}{N},\ \frac{bs+r}{p-qb} \geq \frac{A}{N}\ \mbox{ and } \tau'\geq -1 -\frac{1}{N}, \]
            and with $N \geq 2$ minimal, so that $\TT{\cpq; \emptyset;\frac{bs+r}{p-qb}}=[-1-\frac{1}{N},-1]$. For clarity, we write $M$ for the minimal such integer and claim $M=p-qb$ to complete the proof in this case. 

            Let $M=p-qb>0$ and $A=bs+r>0$. We see that $0<A<M$ and $A,M$ are coprime. Moreover $\frac{bs+r}{p-bq}=\frac{A}{M}$ and $\frac{q+s}{q}>1-\frac{A}{M}$ by direct calculations and an application of Lemma \ref{lemma:brange}. To see that $M$ is minimal, we proceed as follows.

            The inequalities $\frac{q+s}{q}>1-\frac{A}{N}$ and $\frac{bs+r}{p-qb}\geq \frac{A}{N}$ are equivalent to $-\frac{s}{q}<\frac{A}{N} \leq \frac{bs+r}{p-qb}$, which is also equivalent to $-(p-qb)sN<Aq(p-qb)\leq q(bs+r)N$. It follows that $0<(p-qb)sN+Aq(p-qb)$, which factors as $(p-qb)(sN+qA)$, and thus $p-qb\leq (p-qb)sN+Aq(p-qb)$. Similarly, $0\leq q(bs+r)N-Aq(p-qb)$. Also note that $(p-qb)s+q(bs+r)=1$ and we finish the proof in this case by calculating \[N=(q(bs+r)N-Aq(p-qb))+((p-qb)sN+Aq(p-qb))\geq p-qb.\]

            \item $J=\{1\}$.
            In this case, we aim to find two coprime integers $A,N$ with $0<A<N$ and \[\frac{q+s}{q}>1-\frac{A}{N},\ \frac{bs+r}{p-qb} > \frac{A}{N}\ \mbox{ and } \tau'\geq -1 -\frac{1}{N}, \]
            and with $N \geq 2$ minimal, so that $\TT{\cpq; \{1\};\frac{bs+r}{p-qb}}=[-1-\frac{1}{N},-1]$. As before, we write $M$ for the minimal such integer and claim that $M=p-q(b-1)$.

            We proceed as in Subcase 2(i), however, $0<q(bs+r)N-Aq(p-qb)$ is now a strict inequality. Since it factors as $q((bs+r)N-A(p-qb))$, we have $q\leq (bs+r)qN-Aq(p-qb)$. Recalling that $(p-qb)s+q(bs+r)=1$, we compute \[N=(q(bs+r)N-A(q(p-qb)))+((p-qb)sN+Aq(p-qb))\geq p-qb+q.\]
        \end{casesp}
    \end{casesp}
\end{proof}



\begin{lemma}\label{lemma:brange2}
    Suppose that $p\geq1$ and $q> 1$ are coprime and $ps+qr=1$ with $-q<s<0<r\leq p$. If $b$ is an integer with $b\geq \frac{p}{q}$, then $b\geq \frac{r}{-s}>\frac{p}{q}$ and \[0\leq \min\{\frac{bs+r}{p-qb}, \frac{q+s}{q}\}< \frac{1}{2} \leq \max\{\frac{bs+r}{p-qb} ,\frac{q+s}{q}\}< 1.\] Moreover, if $\frac{bs+r}{p-qb}\not=0$, then $\frac{-s}{q}>\frac{bs+r}{p-qb}\geq \frac{-s}{q+1}$ and $b>\frac{r}{-s}$.
\end{lemma}

\begin{proof}
    It is clear that $0<\frac{q+s}{q}<1$. From $ps+qr=1$, we have $\frac{r}{-s}=\frac{p}{q}+\frac{1}{q(-s)}$. But $b\geq \frac{p}{q}$ is an integer, so $b\geq \frac{p}{q}+\frac{1}{q}\geq \frac{p}{q}+\frac{1}{q(-s)}=\frac{r}{-s}$, since $p\geq1$ and $q> 1$ are coprime integers and $-q<s<0<r\leq p$. Also note that $bs+r$ and $p-qb$ are coprime.

    Observe that $\frac{bs+r}{p-qb} < \frac{-s}{q}$ holds, as it is equivalent to $0< s(p-qb)+q(bs+r)$, the right-hand side of which is equal to $1$. Thus
    $\frac{q+s}{q}+\frac{bs+r}{p-qb}< \frac{q+s}{q}+\frac{-s}{q}=1$ and it follows that \[0\leq \min\{\frac{bs+r}{p-qb}, \frac{q+s}{q}\}< \frac{1}{2} \quad \text{ and }\quad \max\{\frac{bs+r}{p-qb}, \frac{q+s}{q}\}<1.\]

    Note that $b\geq \frac{r}{-s}$ is an integer. If $b=\frac{r}{-s}$, then $s=-1$; it follows that $\frac{q+s}{q}\geq\frac{1}{2}$ and this is the only case where $\frac{bs+r}{p-qb}=0$, so the claimed inequalities of the lemma hold in this case. On the other hand, if $b>\frac{r}{-s}$, then $b\geq\frac{r+1}{-s}$. The inequality $\frac{bs+r}{p-qb} \geq \frac{-s}{q+1}$ reduces to $bs+r \leq -1$ upon using the identity $ps+qr=1$, and thus it holds. Therefore
    \[\max\{\frac{bs+r}{p-qb},\frac{q+s}{q}\} \,\geq\, \max\{\frac{-s}{q+1}, \frac{q+s}{q}\}.\]
    If $-s\leq \frac{q}{2}$, then $\frac{q+s}{q}\geq \frac{1}{2}$; if $\frac{q}{2}<-s$, then $\frac{q}{2} + \frac{1}{2}\leq-s$ and so $\frac{-s}{q+1}\geq\frac{1}{2}$. It follows that $\max\{\frac{q+s}{q},\frac{-s}{q+1}\}\geq \frac{1}{2}$ and hence \[\frac{1}{2} \leq \max\{\frac{bs+r}{p-qb} ,\frac{q+s}{q}\}.\]
\end{proof}

\begin{proposition}\label{prop:tpqspecialcase2}
    Assume $p\geq1$ and $q> 1$ are coprime integers and $r,s$ satisfy $ps+qr=1$ and $-q<s<0<r\leq p$. Let $b$ be an integer with $b\geq \frac{p}{q}$, and $\cpq$ be the cable space defined in Section \ref{subsec:cpqnotations}. Then \[\TT{\cpq;\emptyset;\frac{bs+r}{p-qb}} = [-1, -1+\frac{1}{bq-p}].\]
\end{proposition}

\begin{proof}

    Lemma \ref{lemma:brange2} applies here. We see that $0\leq \frac{bs+r}{p-qb}<1$. Moreover, $\frac{bs+r}{p-qb}=0$ occurs only when $s=-1$ and $b=\frac{r}{-s}=r$. If this is the case, then $\TT{\cpq;\emptyset;\frac{bs+r}{p-qb}}=[-1,0]$ by Proposition \ref{prop:tcpqinterval} and so the statement holds since $bq-p=ps+qr=1$.

    By Lemma \ref{lemma:brange2} we are left to consider the cases $0<\frac{bs+r}{p-qb}<\frac{1}{2}\leq \frac{q+s}{q}<1$ and $0<\frac{q+s}{q}<\frac{1}{2}\leq \frac{bs+r}{p-qb}<1$. In either case, $bs+r$ and $p-qb$ are coprime and $\{-1\}\subset \TT{\cpq;\emptyset;\frac{bs+r}{p-qb}} \subset [-1,0]$ by Proposition \ref{prop:tcpqinterval}, since $\frac{bs+r}{p-qb}<\frac{-s}{q}$ by Lemma \ref{lemma:brange2}. Hence, we need to determine all possible $-1<\tau'<0$ such that $(\emptyset;0;\frac{q+s}{q};\frac{bs+r}{p-qb},\tau')$ is JN-realisable by the definition of $\TT{\cpq;\emptyset;\frac{bs+r}{p-qb}}$. Since $-1<\tau'<0$, $(\emptyset;0;\frac{q+s}{q};\frac{bs+r}{p-qb},\tau')$ is JN-realisable if and only if $(\emptyset;1;\frac{q+s}{q};\frac{bs+r}{p-qb},1+\tau')$ is JN-realisable from the discussion at the beginning of the Appendix. Moreover, by Theorem \ref{thm:rawallpositive}(4), $(\emptyset;1;\frac{q+s}{q};\frac{bs+r}{p-qb},1+\tau')$ is JN-realisable if and only if there are coprime numbers $A,N$ with $0<A<N$ and a permutation $\{A_1, A_2,A_3\}$ of $\{A, N-A, 1\}$ such that
    \[\frac{q+s}{q}<\frac{A_1}{N},\ \frac{bs+r}{p-bq}\leq \frac{A_2}{N}\ \text{ and }\ 1+\tau'\leq \frac{A_3}{N}.\]

    \begin{casesp}
        \item $0<\frac{bs+r}{p-qb}<\frac{1}{2}\leq \frac{q+s}{q}<1$.

        Since $\frac{1}{2}\leq \frac{q+s}{q} <\frac{A_1}{N}$, we must have $A_1\not=1$ and $N>2$. We may assume $A_1=N-A$ by replacing $A$ by $N-A$ if necessary.

        \begin{casesp}
            \item\label{case:1.i} $A_2=1$.
            In this case, we aim to find two coprime integers $A,N$ with $0<A<N$ and \[ \frac{q+s}{q}<\frac{N-A}{N},\ \frac{bs+r}{p-bq}\leq \frac{1}{N}\ \text{ and }\ 1+\tau'\leq \frac{A}{N}.\] The first inequality $\frac{q+s}{q}<\frac{N-A}{N}$ is equivalence to $\frac{-s}{q}>\frac{A}{N}$. By Lemma \ref{lemma:brange2}, we have $\frac{1}{N} \geq \frac{bs+r}{p-bq}\geq \frac{-s}{q+1}$. Therefore, $\frac{-s}{q}>\frac{-s}{q+1}A$ and so $A<\frac{q+1}{q}$. The only possible choice is $A=1$ and so the next subcase will finish the proof of Case 1.
            \item\label{case:1.ii} $A_3=1$.
            In this case, we aim to find two coprime integers $A,N$ with $0<A<N$ and \[ \frac{q+s}{q}<\frac{N-A}{N},\ \frac{bs+r}{p-bq}\leq \frac{A}{N}\ \text{ and }\ 1+\tau'\leq \frac{1}{N},\] and with $N$ minimal, so $\TT{\cpq;\emptyset;\frac{bs+r}{p-qb}}=[-1,-1+\frac{1}{N}]$. As before we write $M$ for the minimal such integer and prove $M=qb-p$ to complete the proof in this case. Recall that from Lemma \ref{lemma:brange2}, $0<\frac{-s}{q+1}\leq \frac{bs+r}{p-bq}<\frac{-s}{q}$.

            Firstly, note taking $A=-(bs+r)$ and $M=qb-p$ gives $\frac{bs+r}{p-bq} =\frac{A}{M}\leq \frac{A}{M}$ and $\frac{q+s}{q}<\frac{M-A}{M}$, since $\frac{bs+r}{p-bq}<\frac{-s}{q}$. Also $0<A<M$ and $A,M$ are coprime. To see that any such $N$ satisfies $N\geq M=qb-p$ so that $M$ is minimal, we proceed as follows.

            The inequalities $\frac{q+s}{q}<\frac{N-A}{N}$ and $\frac{bs+r}{p-bq}\leq \frac{A}{N}$ are equivalent to $\frac{bs+r}{p-bq}\leq\frac{A}{N}<\frac{-s}{q}$. Observing that both $bs+r$ and $p-qb$ are negative, we arrive at $Nq(bs+r)\geq Aq(p-qb)$ and $sN(p-qb)>-Aq(p-qb)$. Since both $sN(p-qb)$ and $-Aq(p-qb)$ are multiples of $(p-qb)$, we have $sN(p-qb)\geq -Aq(p-qb)+(qb-p)$. Since $(p-qb)s+q(bs+r)=1$, we see that \[N=Ns(p-qb)+Nq(bs+r)\geq -Aq(p-qb)+(qb-p)+Aq(p-qb)=qb-p.\]
        \end{casesp}

        \item $0<\frac{q+s}{q}<\frac{1}{2}\leq \frac{bs+r}{p-qb}<1$.

        If $A_2=1$, then $\frac{1}{2}\leq \frac{bs+r}{p-qb}\leq \frac{1}{N}$ implies $N=2$. Moreover, $N=2$ and $A=1$ is a possible pair if and only if $\frac{bs+r}{p-qb}=\frac{1}{2}$. If this is the case, $bs+r=-1$ and $p-qb=-2$, since $bs+r$ and $p-qb$ are coprime. Moreover $N=2$ is the minimal $N$ satisfying the desired inequalities, so $\TT{\cpq;\emptyset;\frac{bs+r}{p-qb}} = [-1, -\frac{1}{2}]$. This agrees with the formula $\TT{\cpq;\emptyset;\frac{bs+r}{p-qb}} = [-1, -1+\frac{1}{bq-p}]$.

        Having dealt with $A_2=1$, we may assume $A_2=A$ by replacing $A$ by $N-A$ if necessary.
        \begin{casesp}
            \item $A_1=1$. In this case, we aim to find two coprime integers $A,N$ with $0<A<N$ and \[ \frac{q+s}{q}<\frac{1}{N},\ \frac{bs+r}{p-bq}\leq \frac{A}{N}\ \text{ and }\ 1+\tau'\leq \frac{N-A}{N}.\]
            Since $-q<s<0$, we have $\frac{1}{q}\leq \frac{q+s}{q} <\frac{1}{N}$. Therefore, $N<q$. Lemma \ref{lemma:brange2} says $\frac{bs+r}{p-bq}\geq \frac{-s}{q+1}$, which gives us \[\frac{q}{q+1}=\frac{q+s}{q+1}+\frac{-s}{q+1}<\frac{q+s}{q}+\frac{bs+r}{p-qb} < \frac{1}{N}+\frac{A}{N}.\] So if $A\leq N-2$, then we have \[N<q \ \text{ and }\ \frac{q}{q+1}<\frac{A+1}{N} \leq \frac{N-1}{N},\] which cannot be simultaneously true. The only possibility is that $A=N-1$ and so we can resolve Case 2 by addressing the following subcase.
            \item $A_1=N-A$. In this case, we aim to find two coprime integers $A,N$ with $0<A<N$ and \[ \frac{q+s}{q}<\frac{N-A}{N},\ \frac{bs+r}{p-bq}\leq \frac{A}{N}\ \text{ and }\ 1+\tau'\leq \frac{1}{N},\] and with $N$ minimal, so that $\TT{\cpq;\emptyset;\frac{bs+r}{p-qb}}=[-1,-1+\frac{1}{N}]$. The rest of the proof is exactly the same as Subcase 1(ii).
        \end{casesp}
    \end{casesp}
\end{proof}

\begin{theorem}\label{thm:torusknot}
    Assume $p\geq1$ and $q> 1$ are coprime integers and $K$ is a knot in $S^3$. Let $M=S^3-n(K)$ and $M'=S^3-n(\cpq(K))$ denote the knot complements of $K$ and its cable $\cpq(K)$, respectively. Fix bases $\mathcal{B}_K=\{\mu,\lambda\}$ and $\mathcal{B}_C=\{\mu_C,\lambda_C\}$ of $\pi_1(\partial M)$ and $\pi_1(\partial M')$, respectively, as explained in Section \ref{subsec:cpqnotations}. Then, we have:
    \begin{enumerate}
        \item If $K$ is the unknot and $p,q\geq 2$, then $\cpq(K)$ is the torus knot $T_{p,q}$ and we have \[\mathcal{T}_{ord}(\emptyset,\{1\};M';\mathcal{B}_C)=[-\infty,pq-p-q] \] and \[ \mathcal{T}_{ord}(\{1\},\{1\};M';\mathcal{B}_C)=(-\infty,pq-p-q).\]
        \item If $K$ is a nontrivial knot, then $[-\infty,pq-p]\subset \mathcal{T}_{ord}(\emptyset,\{1\};M';\mathcal{B}_C).$ Moreover,
              \begin{enumerate} \item If $2g(K)-1<\frac{p}{q}$, then \[[-\infty,pq-p+(2g(K)-1)q]\subset \mathcal{T}_{ord}(\emptyset,\{1\};M';\mathcal{B}_C);\]
                  \item if $2g(K)-1>\frac{p}{q}$, then\[ [pq-p+ (2g(K)-1)q,\infty]\subset\mathcal{T}_{ord}(\emptyset,\{1\};M';\mathcal{B}_C).\]
              \end{enumerate}
    \end{enumerate}
\end{theorem}

\begin{proof}
    For part (1), choose $r$ and $s$ such that $ps+qr=1$ and $-q<s<0<r<p$.
    Note that in this case $M'$ is the knot complement of the torus knot $T_{p,q}$, which can be viewed as the result of gluing a solid torus to $\cpq$ by identifying the canonical meridian of the solid torus to $\lambda$ in $\mathcal{B}_K$, or to the slope $f(0)=\frac{r}{q}$ with respect to the basis $\mathcal{B}_1$. Here $f$ is the transition map between the bases $\mathcal{B}_K$ and $\mathcal{B}_1$ discussed in Section \ref{subsec:slopetransition}. Now $M'$ is a Seifert manifold and $\tau' \in \mathcal{T}_{ord}(\emptyset,\{1\};M';\mathcal{B}_2)$ if and only if
    $(\emptyset;0;\frac{q+s}{q},\frac{r}{p};\tau')$ is JN-realisable, which is also equivalent to saying $(\{1\};0;\frac{q+s}{q};\frac{r}{p},\tau')$ is JN-realisable.
    Therefore, we have \[\mathcal{T}_{ord}(\emptyset,\{1\};M';\mathcal{B}_2)=\TT{\cpq;\{1\};\frac{r}{p}}=[-\frac{p+q+1}{p+q},-1]\] by Proposition \ref{prop:tpqspecialcase} with $b=0$. According to the change of basis described in Lemma \ref{lemma:b2tobc}, we get
    \[\mathcal{T}_{ord}(\emptyset,\{1\};M';\mathcal{B}_C)=[-\infty, pq-p-q].\]
    Similarly, $\mathcal{T}_{ord}(\{1\},\{1\};M';\mathcal{B}_C)=(-\infty, pq-p-q)$ by Theorem \ref{thm:rawcases}(4)(b).

    To see part (2), suppose that $K$ is a nontrivial knot and choose $r$ and $s$ such that $ps+qr=1$ and $-q<s<0<r\leq p$. Then $0\in\mathcal{T}_{ord}(\emptyset,\{1\};M;\mathcal{B}_K)$ by \cite[Example 6.3]{BGYpreprinta}, and so $f(0)=\frac{r}{q}\in \mathcal{T}_{ord}(\emptyset, \{1\}; M; \mathcal{B}_1)$. By Theorem \ref{thm:cpqtord} and Proposition \ref{prop:tpqspecialcase} with $b=0$, we have $[-1-\frac{1}{p},-1]=\TT{\cpq;\emptyset;\frac{r}{q}}\subset \mathcal{T}_{ord}(\emptyset,\{1\};M';\mathcal{B}_2)$. By Lemma \ref{lemma:b2tobc}, we have \[[-\infty,pq-p]\subset \mathcal{T}_{ord}(\emptyset,\{1\};M';\mathcal{B}_C).\]

    In addition, note that $2g(K)-1\in\mathcal{T}_{ord}(\emptyset,\{1\};M;\mathcal{B}_K)$ by \cite[Corollary 1.4]{BGYpreprinta}. For ease of writing, set $b=2g(K)-1$. Then $f(b)=\frac{bs+r}{p-qb}\in \mathcal{T}_{ord}(\emptyset, \{1\}; M; \mathcal{B}_1)$ and hence $\TT{\cpq;\emptyset;\frac{bs+r}{p-qb}}\subset \mathcal{T}_{ord}(\emptyset,\{1\};M';\mathcal{B}_2)$ by Theorem \ref{thm:cpqtord}. If $1\leq b \leq \frac{p}{q}$, then Proposition \ref{prop:tpqspecialcase} applies and $\TT{\cpq;\emptyset;\frac{bs+r}{p-qb}}=[-1-\frac{1}{p-qb},-1]$; if $b \geq \frac{p}{q}$, then Proposition \ref{prop:tpqspecialcase2} applies and $\TT{\cpq;\emptyset;\frac{bs+r}{p-qb}}=[-1,-1+\frac{1}{qb-q}]$. Lemma \ref{lemma:b2tobc} then gives us that \[[-\infty,pq-p-q+2g(K)q]\subset \mathcal{T}_{ord}(\emptyset,\{1\};M';\mathcal{B}_C),\] when $1\leq 2g(K)-1 \leq \frac{p}{q}$; and \[[pq-p-q+2g(K)q,\infty]\subset \mathcal{T}_{ord}(\emptyset,\{1\};M';\mathcal{B}_C),\] when $2g(K)-1>\frac{p}{q}$.
\end{proof}

\begin{remark}
    The previous result is based on \cite[Corollary 1.4]{BGYpreprinta}, which shows that the slope $2g(K)-1$ is always order-detected. In fact, \cite[Corollary 1.4]{BGYpreprinta} shows that every rational slope whose distance from the longitudinal slope divides $2g(K)-1$ is order-detected. Because of this, the conclusion of Theorem \ref{thm:torusknot}(2)(b) can be improved in cases where $2g(K)-1$ is not prime.
\end{remark}

With these results, we are able to mirror the following result of Hedden and Hom, which completely describes how L-space knots behave with respect to cabling. 

\begin{theorem}\cite{Hedden09, Hom11}
    \label{thm:heddenhom}
    Suppose that $p,q \geq 2$ are relatively prime.	The $(p, q)$-cable of a knot $K \subset S^3$ is an L-space knot if and only if $K$ is an L-space knot and $p/q \geq 2 g(K) - 1$.
\end{theorem}

A knot $K \subset S^3$ is an L-space knot if and only if the set of NLS-detected slopes is precisely $[-\infty, 2g(K)-1]$, otherwise all slopes are NLS-detected (see e.g. \cite[Corollary 1.12]{RR17}, combined with \cite[Proposition 2.1]{OS05}). Theorem \ref{thm:heddenhom} therefore has a natural interpretation in terms of intervals of NLS-detected slopes and their behaviour with respect to cabling. We have an analogous result for the behaviour of order-detected slopes, and their behaviour with respect to cabling.

\begin{corollary}
    \label{cor:cableintervals}
    Assume $p\geq1$ and $q> 1$ are coprime integers. Let $M=S^3-n(K)$ and $M'=S^3-n(\cpq(K))$ be the knot complements of a nontrivial knot $K$ and its cable $\cpq(K)$, respectively. Fix bases $\mathcal{B}_K=\{\mu,\lambda\}$ and $\mathcal{B}_C=\{\mu_C,\lambda_C\}$ of $\pi_1(\partial M)$ and $\pi_1(\partial M')$ as explained in Section \ref{subsec:cpqnotations}.
    \begin{enumerate}
        \item If $[-\infty,2g(K)-1]= \mathcal{T}_{ord}(\emptyset,\emptyset;M;\mathcal{B}_K)$ and $2g(K)-1 <\frac{p}{q}$, then $[-\infty,pq-p-q+2g(K)q] = \mathcal{T}_{ord}(\emptyset,\{1\};M';\mathcal{B}_C)$.
        \item If $[-\infty,2g(K)-1]= \mathcal{T}_{ord}(\emptyset,\{1\};M;\mathcal{B}_K)$ and $2g(K)-1 >\frac{p}{q}$, then $\mathcal{T}_{ord}(\emptyset,\{1\};M';\mathcal{B}_C)=\RR\cup\{\infty\}$.
        \item If $ \mathcal{T}_{ord}(\emptyset,\{1\};M;\mathcal{B}_K)=\RR\cup\{\infty\}$, then $\mathcal{T}_{ord}(\emptyset,\{1\};M';\mathcal{B}_C)=\RR\cup\{\infty\}$.
    \end{enumerate}
\end{corollary}

\begin{proof}
    For convenience, write $b=2g(K)-1$.

    To show part (1), we argue as follows. Choose $r,s$ with $ps+qr=1$ and $-q<s<0<r\leq p$. Since $\mathcal{T}(\emptyset,\emptyset;M;\mathcal{B}_K)=[-\infty,b] \neq \RR$, we have $\mathcal{T}_{ord}(\emptyset, \emptyset;M;\mathcal{B}_K)=\mathcal{T}(\emptyset,\{1\};M;\mathcal{B}_K)=[-\infty,b]$ by \cite[Theorem 1.2]{BC23}. From $[-\infty,b]= \mathcal{T}_{ord}(\emptyset,\{1\};M;\mathcal{B}_K)$ and $b<\frac{p}{q}$, we have $f([-\infty,b])=[-\frac{s}{q},\frac{sb+r}{p-qb}]= \mathcal{T}_{ord}(\emptyset,\{1\};M;\mathcal{B}_1)$ by Lemma \ref{lemma:bktob1}, where $f$ is the map arising from the change of basis. It follows from Corollary \ref{cor:specialcasedetect} that \[\bigcup_{\tau\in [-\frac{s}{q},\frac{sb+r}{p-qb}]}\TT{\cpq;\emptyset;\tau} = \mathcal{T}_{ord}(\emptyset,\{1\};M';\mathcal{B}_2).\]
    Since $b <\frac{p}{q}$, Proposition \ref{prop:tpqspecialcase} applies and gives us $$\TT{\cpq;\emptyset;\frac{sb+r}{p-qb}}=[-1-\frac{1}{p-qb},-1].$$ From Proposition \ref{prop:tcpqinterval}, we have $\TT{\cpq;\emptyset;\frac{-s}{q}}=\{-1\}$. By Corollary \ref{cor:tcpqbehav} and its accompanying discussion, \[\bigcup_{\tau\in [-\frac{s}{q},\frac{sb+r}{p-qb}]}\TT{\cpq;\emptyset;\tau}= [-1-\frac{1}{p-qb},-1],\] and so $[-1-\frac{1}{p-qb},-1]= \mathcal{T}_{ord}(\emptyset,\{1\};M';\mathcal{B}_2)$. It follows from Lemma \ref{lemma:b2tobc} that $[-\infty,pq-p-q+2g(K)]= \mathcal{T}_{ord}(\emptyset,\{1\};M';\mathcal{B}_C)$.

    Part (2) is argued in a similar way. Since $[-\infty,b]= \mathcal{T}_{ord}(\emptyset,\{1\};M;\mathcal{B}_K)$ and $b>\frac{p}{q}$, Lemma \ref{lemma:bktob1} says $f([-\infty,b])=[-\infty,\frac{sb+r}{p-qb}]\cup[\frac{-s}{q},\infty]= \mathcal{T}_{ord}(\emptyset,\{1\};M;\mathcal{B}_1)$, where $f$ is the map arising from the change of basis. It follows from Theorem \ref{thm:cpqtord} that
    \[\bigcup_{\tau\in (-\infty,\frac{sb+r}{p-qb}]\cup[\frac{-s}{q},\infty)}\TT{\cpq;\emptyset;\tau} \subset \mathcal{T}_{ord}(\emptyset,\{1\};M';\mathcal{B}_2).\]
    We apply Corollary \ref{cor:tcpqbehav} and its accompanying discussion as above. 
    We find that
    \[\bigcup_{\tau\in [\frac{-s}{q},\infty)}\TT{\cpq;\emptyset;\tau}=(-\infty,-1].\]
    Since $b> \frac{p}{q}$, Proposition \ref{prop:tpqspecialcase2} gives us $\TT{\cpq;\emptyset;\frac{sb+r}{p-qb}} =[-1,-1+\frac{1}{qb-p}]$. It follows that $$\bigcup_{\tau\in (-\infty,\frac{sb+r}{p-qb}]}\TT{\cpq;\emptyset;\tau} = [-1,\infty) .$$
    Since the meridional slope (i.e., the slope $\infty$) is always order-detected \cite[Corollary 1.4]{BGYpreprinta}, we conclude that $\RR\cup\{\infty\}=\mathcal{T}_{ord}(\emptyset,\{1\};M';\mathcal{B}_2)$ and so $\mathcal{T}_{ord}(\emptyset,\{1\};M';\mathcal{B}_C)=\RR\cup\{\infty\}$.

    Part (3) follows immediately from the fact \[\bigcup_{\tau\in \RR}\TT{\cpq;\emptyset;\tau} =\RR\subset \mathcal{T}_{ord}(\emptyset,\{1\};M';\mathcal{B}_2),\] together with the same observations as above.
\end{proof}

Note that if $K'$ is the $(p,q)$-cable of a knot $K$ in $S^3$, then $g(K') = (p-1)(q-1)/2 + qg(K)$. Thus, Theorem \ref{thm:introcables} follows from Theorem \ref{thm:torusknot} and Corollary \ref{cor:cableintervals}.

\begin{remark}
    In a similar manner, one can analyse the set of strongly order-detected slopes: if some interval is contained in $\mathcal{T}_{ord}(\{1\},\{1\};M;\mathcal{B})$, then $\mathcal{T}_{ord}(\{1\},\{1\};M';\mathcal{B}_C)$ contains a corresponding interval. While there are many special cases where intervals of strongly order-detected slopes are shown to exist, there are no general results similar to \cite[Corollary 1.4]{BGYpreprinta} that allow for a clean statement like Theorem \ref{thm:cpqtord} which applies to all knots in $S^3$. Since our results would only be piecemeal, we do not pursue an explicit computation of strongly-detected intervals of slopes on the boundary of cable knots in this manuscript.
\end{remark}

\appendix

\section*{Appendix: Tools for computing JN-realisable tuples}
\label{sec:JN appendix}

\setcounter{section}{1}
\setcounter{theorem}{0}

We follow the notation introduced at the beginning of Section \ref{sec:seifertmanifold} and in Section \ref{sec:cpqcalculations} that matches that of \cite{BC17}. These results are from \cite[Appendix]{BC17}, the proofs of which follow from an analysis of \cite{EHN81, JN85a, JN85b, Naimi94}.

Set $\bar{\tau_i}=\tau_i-\floor{\tau_i}\in [0,1)$ for $i=1,\dots, r$ and $b=-(\floor{\tau_1}+\dots+\floor{\tau_r})$. Then $(J;0;\gamma_1,\dots,\gamma_n;\tau_1,\dots,\tau_r)$ is JN-realisable if and only if $(J;b;\gamma_1,\dots,\gamma_n;\bar{\tau_1},\dots,\bar{\tau_r})$ is JN-realisable.

Firstly, consider the case where there exists $j$ such that $\tau_j$ is an integer, and use $J^0$ to denote $J\backslash \{j: \tau_j\in \ZZ\}$. Since JN realizability is invariant under permutation of $(\tau_1, \ldots, \tau_r)$, we may assume that the $\tau_j$'s are indexed in such a way that $\tau_1,\dots,\tau_{r_1}$ are not integers, and $\tau_{r_1+1},\dots,\tau_r$ are integers and $J\cap \{r_1+1,\dots, r\}=\{r_2+1,\dots , r\}$ for some $r_2 \geq r_1$. Then $(J;b;\gamma_1,\dots,\gamma_n;\bar{\tau_1},\dots,\bar{\tau_r})$ is JN-realisable if and only if $(J^0;b;\gamma_1,\dots,\gamma_n;\bar{\tau_1},\dots,\bar{\tau_{r_2}})$ is JN-realisable since $\tau_j\in \ZZ$ with $j\in\ZZ$ forces the function $g_j$ corresponding to $\bar{\tau_j}$ to be the identity. Therefore, in the case where there exists $j$ such that $\tau_j$ is an integer, we can reduce to considering the case where $j\in J$ implies $\tau_j\notin\ZZ$. To handle this case, we have the following theorem.

\begin{theorem}[\cite{BC17}, Theorem A.1 or \cite{JN85b} Theorem 1]\label{thm:integraltau}
    Suppose that if $j \in J$ then $\tau_j \notin \mathbb{Z}$, and let $s$ be the number of $\tau_j$ which
    are integers. If $s>0$, then $(J;b;\gamma_1,\dots,\gamma_n;\bar{\tau_1},\dots,\bar{\tau_r})$ is JN-realisable if and only if $2-s\leq b\leq n+r-2$.
\end{theorem}

On the other hand, if no $\tau_j$ is an integer, then we have:

\begin{theorem}[\cite{EHN81}, \cite{JN85b}, \cite{Naimi94} and \cite{BC17} Theorem A.2]\label{thm:rawallpositive}
    Suppose that $n+r\geq 3$, $J\subset \{1,2,\dots,r\}$, $b\in \ZZ$ and $0<\gamma_1,\dots,\gamma_n,\bar{\tau_1},\dots,\bar{\tau_r}<1$. Then we have the following.
    \begin{enumerate}
        \item If $(J;b;\gamma_1,\dots,\gamma_n;\bar{\tau_1},\dots,\bar{\tau_r})$ is JN-realisable, then $1\leq b \leq n+r-1$.
        \item If $2\leq b\leq n+r-2$, then $(J;b;\gamma_1,\dots,\gamma_n;\bar{\tau_1},\dots,\bar{\tau_r})$ is JN-realisable in $\widetilde{\PSL(2,\RR)}$.
        \item $(J;n+r-1;\gamma_1,\dots,\gamma_n;\bar{\tau_1},\dots,\bar{\tau_r})$ is JN-realisable if and only if $(J;1;1-\gamma_1,\dots,1-\gamma_n;1-\bar{\tau_1},\dots,1-\bar{\tau_r})$ is JN-realisable.
        \item $(J;1;\gamma_1,\dots,\gamma_n;\bar{\tau_1},\dots,\bar{\tau_r})$ is JN-realisable if and only if there are coprime integers $0<A<N$ and some permutation $(\frac{A_1}{N},\frac{A_2}{N},\dots,\frac{A_{n+r}}{N})$ of $(\frac{A}{N},1-\frac{A}{N},\frac{1}{N},\dots,\frac{1}{N})$ such that \begin{itemize}
                  \item $\gamma_i<\frac{A_i}{N}$ for all $1\leq i\leq n$;
                  \item $\bar{\tau_j}<\frac{A_{n+j}}{N}$ for all $j\in J$;
                  \item $\bar{\tau_j}\leq \frac{A_{n+j}}{N}$ for all $j\in \{1,2,\dots,r\}\backslash J$.
              \end{itemize}
    \end{enumerate}
\end{theorem}

We can use these theorems to develop a notion of ``relative JN-realisability" as follows. Following \cite[Appendix]{BC17}, for a fixed tuple $\tau_*=(\tau_1,\dots,\tau_{r-1})$, we set \begin{itemize}
    \item $r_1=|\{j: \tau_j\notin \ZZ, 1\leq j\leq r-1\}|$, the number of non-integral $\tau_j$;
    \item $s_0=|\{j:\tau_j\in \ZZ \mbox{ and } j\in \{1,2,\dots,r-1\}\backslash J\}|$, the number of integral $\tau_j$ whose indices not in $J$;
    \item $b_0=-(\floor{\tau_1}+\dots+\floor{\tau_{r-1}})$;
    \item $m_0=b_0-(n+r_1+s_0-1)$;
    \item $m_1=b_0+s_0-1$.
\end{itemize}

For a Seifert fibered manifold as in Section \ref{sec:seifertmanifold}, fix $J\subset \{1,\dots,r-1\}$ and $\tau_*=(\tau_1,\dots,\tau_{r-1})\in \RR^{r-1}$.
Set \[\TT{M;J;\tau_*}=\{\tau'\in \RR: (J;0;\gamma_1,\dots,\gamma_n;\tau_1,\dots,\tau_{r-1},\tau') \mbox{ is JN-realisable}\}\]
and \[\TTT{M;J;\tau_*}=\{\tau'\in \RR: (J\cup\{r\};0;\gamma_1,\dots,\gamma_n;\tau_1,\dots,\tau_{r-1},\tau') \mbox{ is JN-realisable}\}.\]

\begin{theorem}[\cite{BC17} Proposition A.4\footnote{The statements of 2(a)(iii) and 3(a)(iii) are slightly changed from that of \cite{BC17} Proposition A.4, where there is a small error in the case where $n+r_1=2$.}]\label{thm:rawcases}
    Fix $J\subset \{1,\dots,r-1\}$ and $\tau_*=(\tau_1,\dots,\tau_{r-1})\in \RR^{r-1}$. Suppose that $n+r_1+s_0\geq 2$. Then we have the following.
    \begin{enumerate}
        \item \begin{enumerate}
                  \item $(m_0,m_1)\subset \TTT{M;J;\tau_*}\subset \TT{M;J;\tau_*}\subset (m_0-1,m_1+1)$.
                  \item $[m_0,m_1]\subset \TT{M;J;\tau_*}$.
                  \item If $s_0>0$, then $m_0<m_1$ and $(m_0,m_1)= \TTT{M;J;\tau_*}\subset \TT{M;J;\tau_*}= [m_0,m_1]$.
              \end{enumerate}
        \item \begin{enumerate}
                  \item If $\TT{M;J;\tau_*}\cap (m_0-1,m_0)\not= \emptyset$, then
                        \begin{enumerate}
                            \item $s_0=0$;
                            \item $1\geq |\{i: \gamma_i\leq \frac{1}{2}\}|+|\{j\in J: 0<\bar{\tau_j}\leq \frac{1}{2}\}|+|\{j\notin J: 0<\bar{\tau_j}\leq \frac{1}{2}\}|$;
                            \item
                                  if $n+r_1 \geq 3$ then there is some $\eta \in(m_0-1,m_0)\cap \QQ$ such that $\TTT{M;J;\tau_*}\cap(m_0-1,m_0]=(\eta,m_0]$ and $\TT{M;J;\tau_*}\cap(m_0-1,m_0]=[\eta,m_0]$, or if $n + r_1 = 2$ then there is some $\eta \in(m_0-1,m_0) \cap \QQ$ such that $\TTT{M;J;\tau_*}=(\eta,m_0)$ and $\TT{M;J;\tau_*}=[\eta,m_0]$.
                        \end{enumerate}
                  \item If $n+r_1=2$, then $\TT{M;J;\tau_*}\cap (m_0-1,m_0)\not= \emptyset$ if and only if \begin{itemize}
                            \item either $-(\gamma_1+\dots+\gamma_n+\tau_1+\dots+\tau_{r-1})<m_0$,
                            \item or $n=0$, $-\sum_{j=1}^r \tau_j=m_0$, $\tau_j\in \QQ$ for all $j=1,\dots,r$ and there is some $j\in J$ with $\tau_j$ not integral.
                        \end{itemize}
              \end{enumerate}
        \item \begin{enumerate}
                  \item If $\TT{M;J;\tau_*}\cap (m_1,m_1+1)\not= \emptyset$, then
                        \begin{enumerate}
                            \item $s_0=0$;
                            \item $1\geq |\{i: \gamma_i\geq \frac{1}{2}\}|+|\{j\in J: 0<\bar{\tau_j}\geq \frac{1}{2}\}|+|\{j\notin J: 0<\bar{\tau_j}\geq \frac{1}{2}\}|$;
                            \item
                                  If $n+r_1 \geq 3$ then there is some $\xi\in (m_1,m_1+1)\cap \QQ$ such that $\TTT{M;J;\tau_*}\cap [m_1,m_1+1)=[m_1,\xi)$ and $\TT{M;J;\tau_*}\cap [m_1,m_1+1)=[m_1,\xi]$ or if $n + r_1 = 2$ then there is some $\xi\in (m_1,m_1+1)\cap \QQ$ such that $\TTT{M;J;\tau_*}=(m_1,\xi)$ and $\TT{M;J;\tau_*}=[m_1,\xi]$.
                        \end{enumerate}
                  \item If $n+r_1=2$, then $\TT{M;J;\tau_*}\cap(m_1,m_1+1)\not=\emptyset$ if and only if
                        \begin{itemize}
                            \item either $-(\gamma_1+\dots+\gamma_n+\tau_1+\dots+\tau_{r-1})>m_0$,
                            \item or $n=0$, $-\sum_{j=1}^r \tau_j=m_0$, $\tau_j\in \QQ$ for all $j=1,\dots,r$ and there is some $j\in J$ with $\tau_j$ not integral.
                        \end{itemize}
              \end{enumerate}
        \item \begin{enumerate}
                  \item $\TT{M;J;\tau_*}$ is a closed subinterval of $(m_0-1,m_1+1)$ with rational endpoints.
                  \item Either $\TTT{M;J;\tau_*}$ is the interior of $\TT{M;J;\tau_*}$ or $s_0=0$, $n+r_1=2$ and $\TTT{M;J;\tau_*}=\TT{M;J;\tau_*}=\{m_0\}$.
                  \item $\TTT{M;J;\tau_*}=\{m_0\}$ if and only if $s_0=0$, $n+r_1=2$, $m_0= -(\gamma_1+\dots+\gamma_n+\tau_1+\dots+\tau_{r-1})$ and either $n\not=0$ or $\tau_j\notin \QQ$ for some $j\in\{1,\dots,r\}$ or there is some $j\in J$ with $\tau_j$ not integral.
              \end{enumerate}
    \end{enumerate}
\end{theorem}

\bibliographystyle{plain}

\bibliography{cablesbib}

\def\cprime{$'$}
\begin{thebibliography}{10}

\bibitem{BG09}
V.~V. Bludov and A.~M.~W. Glass.
\newblock Word problems, embeddings, and free products of right-ordered groups
  with amalgamated subgroup.
\newblock {\em Proceedings of the London Mathematical Society}, 99(3):585--608,
  2009.

\bibitem{BC23}
Steven Boyer and Adam Clay.
\newblock Order-detection of slopes on the boundaries of knot manifolds.
\newblock Preprint, available via https://arxiv.org/abs/2206.00848.

\bibitem{BC17}
Steven Boyer and Adam Clay.
\newblock Foliations, orders, representations, {L}-spaces and graph manifolds.
\newblock {\em Adv. Math.}, 310:159--234, 2017.

\bibitem{BGYpreprintb}
Steven Boyer, Cameron~McA Gordon, and Ying Hu.
\newblock {JSJ} decompositions of knot exteriors, {D}ehn surgery and the
  {L}-space conjecture.
\newblock Preprint, available via http://arxiv.org/abs/2307.06815.

\bibitem{BGYpreprinta}
Steven Boyer, Cameron~McA Gordon, and Ying Hu.
\newblock Slope detection and toroidal $3$-manifolds.
\newblock Preprint, available via http://arxiv.org/abs/2106.14378.

\bibitem{BGW13}
Steven Boyer, Cameron~McA. Gordon, and Liam Watson.
\newblock On {L}-spaces and left-orderable fundamental groups.
\newblock {\em Math. Ann.}, 356(4):1213--1245, 2013.

\bibitem{BRW05}
Steven Boyer, Dale Rolfsen, and Bert Wiest.
\newblock Orderable 3-manifold groups.
\newblock {\em Ann. Institut Fourier (Grenoble)}, 55:243--288, 2005.

\bibitem{CalWal11}
Danny Calegari and Alden Walker.
\newblock Ziggurats and rotation numbers.
\newblock {\em J. Mod. Dyn.}, 5(4):711--746, 2011.

\bibitem{CW11}
Adam Clay and Liam Watson.
\newblock On cabled knots, {D}ehn surgery, and left-orderable fundamental
  groups.
\newblock {\em Math. Res. Lett.}, 18(6):1085--1095, 2011.

\bibitem{EHN81}
David Eisenbud, Ulrich Hirsch, and Walter Neumann.
\newblock Transverse foliations of {S}eifert bundles and self-homeomorphism of
  the circle.
\newblock {\em Comment. Math. Helv.}, 56(4):638--660, 1981.

\bibitem{HW23}
Jonathan Hanselman and Liam Watson.
\newblock A calculus for bordered {F}loer homology.
\newblock {\em Geom. Topol.}, 27(3):823--924, 2023.

\bibitem{Hedden09}
Matthew Hedden.
\newblock On knot {F}loer homology and cabling. {II}.
\newblock {\em Int. Math. Res. Not. IMRN}, 2009(12):2248--2274, 2009.

\bibitem{heil74}
Wolfgang Heil.
\newblock Elementary surgery on {S}eifert fiber spaces.
\newblock {\em Yokohama Math. J.}, 22:135--139, 1974.

\bibitem{Hom11}
Jennifer Hom.
\newblock A note on cabling and {L}-space surgeries.
\newblock {\em Algebr. Geom. Topol.}, 11(1):219--223, 2011.

\bibitem{JN85a}
Mark Jankins and Walter Neumann.
\newblock Homomorphisms of {F}uchsian groups to {${\rm PSL}(2,{\rm R})$}.
\newblock {\em Comment. Math. Helv.}, 60(3):480--495, 1985.

\bibitem{JN85b}
Mark Jankins and Walter~D. Neumann.
\newblock Rotation numbers of products of circle homeomorphisms.
\newblock {\em Math. Ann.}, 271(3):381--400, 1985.

\bibitem{Juh15}
Andr\'{a}s Juh\'{a}sz.
\newblock A survey of {H}eegaard {F}loer homology.
\newblock In {\em New ideas in low dimensional topology}, volume~56 of {\em
  Ser. Knots Everything}, pages 237--296. World Sci. Publ., Hackensack, NJ,
  2015.

\bibitem{KM96}
Valeri{\u\i}~M. Kopytov and Nikola{\u\i}~Ya. Medvedev.
\newblock {\em Right-ordered groups}.
\newblock Siberian School of Algebra and Logic. Consultants Bureau, New York,
  1996.

\bibitem{Naimi94}
Ramin Naimi.
\newblock Foliations transverse to fibers of {S}eifert manifolds.
\newblock {\em Comment. Math. Helv.}, 69(1):155--162, 1994.

\bibitem{Navas10a}
Andr\'{e}s Navas.
\newblock On the dynamics of (left) orderable groups.
\newblock {\em Annales de l'institut Fourier}, 60(5):1685--1740, 2010.

\bibitem{OS05}
Peter Ozsv{\'a}th and Zolt{\'a}n Szab{\'o}.
\newblock On knot {F}loer homology and lens space surgeries.
\newblock {\em Topology}, 44(6):1281--1300, 2005.

\bibitem{RR17}
Jacob Rasmussen and Sarah~Dean Rasmussen.
\newblock Floer simple manifolds and {L}-space intervals.
\newblock {\em Advances in Mathematics}, 322:738--805, 2017.

\bibitem{Ras17}
Sarah~Dean Rasmussen.
\newblock {L}-space intervals for graph manifolds and cables.
\newblock {\em Compos. Math.}, 153(5):1008--1049, 2017.

\bibitem{Sik04}
Adam~S. Sikora.
\newblock Topology on the spaces of orderings of groups.
\newblock {\em Bull. London Math. Soc.}, 36(4):519--526, 2004.

\end{thebibliography}

\end{document}